\documentclass[final, reqno, 11pt]{amsart}
\usepackage{amsmath}
\usepackage{amsfonts}
\usepackage{amssymb}
\usepackage{amsthm}
\thanks{This work is part of the author's Ph.D.\ thesis at the University of Chicago.}

\newtheorem{theorem}{Theorem}[section]
\newtheorem{lemma}[theorem]{Lemma}
\newtheorem{proposition}[theorem]{Proposition}
\newtheorem{corollary}[theorem]{Corollary}
\theoremstyle{remark}
\newtheorem{observation}[theorem]{Remark}
\theoremstyle{definition}

\newtheorem{definition}{Definition}[section]
\newcommand{\set}{\mathbb}
\newcommand{\dl}{\nabla}
\renewcommand{\frak}{\mathfrak}

\newcommand{\mc}{\mathcal}
\newcommand{\be}{\begin{equation}}
\newcommand{\ee}{\end{equation}}
\newcommand{\bee}{\begin{align}}
\newcommand{\eee}{\end{align}}

\newcommand{\ba}{\begin{array}}
\newcommand{\ds}{\displaystyle}

\newcommand{\ea}{\end{array}}
\newcommand{\bpm}{\begin{pmatrix}}
\newcommand{\epm}{\end{pmatrix}}
\newcommand{\lb}{\label}

\DeclareMathOperator{\Ran}{Ran}

\DeclareMathOperator{\Rere}{Re}
\DeclareMathOperator{\Imim}{Im}

\newcommand{\ov}{\overline}
\newcommand{\dd}{{\,}{d}}

\newcommand{\R}{\mathbb R}
\newcommand{\C}{\mathbb C}
\renewcommand{\Re}{\Rere}
\renewcommand{\Im}{\Imim}
\newcommand{\les}{\lesssim}

\newcommand{\Dil}{\mathrm{Dil}}
\newcommand{\sol}{\mathrm{sol}}
\newcommand{\Sol}{\mathrm{Sol}}
\DeclareMathOperator{\Epsilon}{\Phi}

\usepackage [frenchlinks=true]{hyperref}

\title[Critical Centre-Stable Manifolds]
{A Critical Centre-Stable Manifold for Schr\"{o}dinger's Equation in Three Dimensions}
\author{Marius Beceanu}
\address{Rutgers Math.\ Dept., 110 Frelinghuysen Rd., Piscataway, NJ, 08854}
\email{mbeceanu@rci.rutgers.edu}
\begin{document}
\numberwithin{equation}{section}
\begin{abstract}
Consider the focusing $\dot H^{1/2}$-critical semilinear Schr\"{o}dinger equation in $\R^3$
\be
i \partial_t \psi + \Delta \psi + |\psi|^2 \psi = 0.
\lb{NLS}
\ee
It admits an eight-dimensional manifold of special solutions called ground state solitons.

We exhibit a codimension-one critical real-analytic manifold $\mc N$ of asymptotically stable solutions of~(\ref{NLS}) in a neighborhood of the soliton manifold. We then show that $\mc N$ is centre-stable, in the dynamical systems sense of Bates--Jones, and globally-in-time invariant.

Solutions in $\mc N$ are asymptotically stable and separate into two asymptotically free parts that decouple in the limit --- a soliton and radiation. Conversely, in a general setting, any solution that stays $\dot H^{1/2}$-close to the soliton manifold for all time is in $\mc N$.

The proof uses the method of modulation. 
New elements include a different linearization and an endpoint Strichartz estimate for the time-dependent linearized equation.

The proof also uses the fact that the linearized Hamiltonian has no nonzero real eigenvalues or resonances. This has recently been established in the case treated here --- of the focusing cubic NLS in $\R^3$ --- by the work of Marzuola--Simpson and Costin--Huang--Schlag.
\end{abstract}
\maketitle



\section{Introduction}
\subsection{Main result}
Consider the equation
$$
i \partial_t \psi + \Delta \psi + |\psi|^2\psi = 0.
$$
Positive solutions to this equation at fixed energy $-\alpha^2$ are of the form $e^{it\alpha^2} \phi(x, \alpha)$, where $\phi(x, \alpha) > 0$ are solutions of the stationary Schr\"{o}dinger equation
\be
-\Delta_x \phi(x, \alpha) + \alpha^2\phi(x, \alpha) = \phi^3(x, \alpha).
\lb{phi}
\ee
The manifold of ground state solitons, here denoted $\sol$, consists of all solutions $\phi(x, \alpha)$, subjected to translations, boost, rescaling, and gauge transformations:
\be\lb{0.2}
\sol: = \big\{w(p) = w(\alpha, \Gamma, v_k, D_k) \in \dot H^{1/2} \mid w(p) = w(\alpha, \Gamma, v_k, D_k) = e^{i(v\cdot x + \Gamma)} \phi(x-D, \alpha)\big\}.
\ee
$\sol$ is diffeomorphic to $\R^7 \times (\R / \set Z)$, which can be lifted to $\R^8$. 

Initial data $w(p) = w(\alpha, \Gamma, v_k, D_k)$ at $t=0$ on the manifold $\sol$ evolves under (\ref{NLS}) into
\be\lb{free_sol}
w(\alpha, \Gamma - t |v|^2 + \alpha^2 t, v_k, 2t v_k + D_k) = e^{i(\Gamma + v\cdot x - t|v|^2 + \alpha^2 t)} \phi(x-2tv-D, \alpha).
\ee

Letting the parameters of motion also progress along a time-dependent parameter path
$$
\pi(t) = \big(v_k(t), D_k(t), \alpha(t), \Gamma(t)\big)
$$
with $\|\dot \pi\|_{L^1_t \cap L^{\infty}_t} < \infty$, we define asymptotically free moving solitons $w_{\pi}(t)$ by
\be\begin{aligned}\lb{1.1}
w_{\pi}(t)(x) &= e^{i \theta(x, t)} \phi(x-y(t), \alpha(t)) \\
\theta(x, t) &\ds= v(t) \cdot x + \int_0^t(\alpha^2(s)-|v(s)|^2) \dd s + \Gamma(t)\\
y(t) &\ds= 2\int_0^t v(s) \dd s + D(t).
\end{aligned}
\ee
$w_{\pi}(t)$ moves on the manifold $\sol$ both under the action of (\ref{NLS}) and by changes in $\pi(t)$.
\begin{theorem}[Main result]\lb{theorem_1}
There exists a codimension-one real analytic manifold $\mc N \subset \dot H^{1/2}$, inside a neighborhood of $\sol$
$$
V(\sol) = \{\psi_0 \in \dot H^{1/2} \mid d_{\dot H^{1/2}}(\psi_0, \sol) := \min_{w \in \sol} \|\psi_0 - w\|_{\dot H^{1/2}} < \epsilon\},
$$
such that for any initial data $\psi(0) = \psi_0 \in \mc N$ the equation (\ref{NLS}) has a global solution $\psi(t)$ that exists on the interval $t \in (t_0, \infty)$, $t_0<0$.

$\mc N$ and solutions $\psi(t)$, for $\psi(0) = \psi_0 \in \mc N$, have the following properties:
\begin{list}{\labelitemi}{\leftmargin=1em}
\item[1.] $\psi(t) = r(t) + w_{\pi}(t)$ is asymptotically stable, in the sense that it is the sum of a moving soliton $w_{\pi}(t)$, parametrized by $\pi$ as in \eqref{1.1}, and a dispersive term $r$, such that
\be\begin{aligned}
\|\dot\pi\|_1 &\les \alpha(0) d_{\dot H^{1/2}} (\psi_0, \sol) \\
\|r\|_{L^{\infty}_t \dot H^{1/2}_x \cap L^2_t \dot W^{1/2, 6}_x} & \les d_{\dot H^{1/2}} (\psi_0, \sol).
\end{aligned}\ee
\item[2.] $\psi$, $r$, and $w_{\pi}$ depend real-analytically on the initial value $\psi(0) \in \mc N$.
\item[3.] The dispersive term $r$ scatters in $\dot H^{1/2}$: for initial data $r_0 \in \mc N$ there exists $\rho \in \dot H^{1/2}$ with $\|\rho\|_{\dot H^{1/2}} \les d_{\dot H^{1/2}}(\psi_0, \sol)$, such that
\be
r(t) = e^{it\Delta} \rho + o_{\dot H^{1/2}}(1).
\ee
\item[4.] $\psi(t) \in \mc N$ for all $t \in (-\epsilon, \infty)$, $\epsilon>0$, i.e.\ $\mc N$ is invariant under the nonlinear flow.
\item[5.] $\mc N$ is the centre-stable manifold of the equation \eqref{NLS}.
\item[6.] Every solution $\psi \in L^{\infty}_t \dot H^{1/2}_x$ to \eqref{NLS} such that $\sup_{t \geq 0} d_{\dot H^{1/2}}(\psi(t), \sol) <<1$ and $\sup_{t \geq 0} \inf_{\Gamma, D} \|\psi(t) - e^{i\Gamma+D\dl} \psi(0)\|_{\dot H^{1/2}}<~2 \|\phi\|_{\dot H^{1/2}}$ must belong to $\mc N$.
\end{list}
\end{theorem}


$\mc N$ is the centre-stable manifold for (\ref{NLS}) and is invariant under the time evolution.

A general condition under which results like Theorem \ref{theorem_1} hold is that the point spectrum of the linearized Hamiltonian comes from the symmetries of the equation; see Section \ref{sec_1.5}.

For a discussion of centre-stable manifolds, the reader is referred to Section \ref{Chapter_1.3}. For a definition of the norms involved in the statement of Theorem \ref{theorem_1}, see the Appendix.

A similar result holds for global backward-in-time solutions. The invariant manifold in that case is the complex conjugate, $\ov {\mc N}$. The intersection of $\mc N$ and $\ov {\mc N}$ is a codimension-two real-analytic manifold of global solutions $\psi(t)$ that exist for all $t \in \R$.

\subsection{Background and history of the problem}
From a physical point of view, the cubic focusing Schr\"{o}dinger equation in $\R^3$ (\ref{NLS}) describes, in a first approximation, the self-focusing of optical beams due to a nonlinear dependence of the refraction index on the field strength. As such, the equation was derived for the first time in 1965 in \cite{kelley} starting from one proposed by Chiao--Garmire--Townes in \cite{cgt}. The physical effect associated by Kelley to finite time blowup in (\ref{NLS}) is called ``anomalous Raman gain''.

(\ref{NLS}) can also serve as a simplified model for the Schr\"odinger map equation. The cubic nonlinearity also appears in the Gross--Pitaevskii equation, which describes Bose--Einstein condensates. It also arises in the mean field limit from the Hartree potential, also in other physical contexts.

In general, consider the semilinear focusing Schr\"odinger equation on $\R^d \times \R$
\be
i \partial_t \psi + \Delta \psi + |\psi|^p \psi = 0,\ \psi(0) \text{ given}.
\lb{NLS'}
\ee
It admits soliton solutions of the form $e^{it\alpha^2} \phi(x, \alpha)$, where
\be\lb{phi'}
-\Delta \phi + \alpha^2\phi = \phi^{p+1}.
\ee
Important invariant quantities for this equation include mass, defined for $L^2$ solutions,
\be
M[\psi] = \int_{\R^d} |\psi(x, t)|^2 \dd x,
\ee
momentum, which is $\dot H^{1/2}$-critical,
\be\lb{momentum}
P[\psi] = \int_{\R^d} i \dl \psi(x) \ov{\psi(x)} \dd x,
\ee
and energy, defined for $H^1$ solutions,
\be
E[\psi] = \int_{\R^d} \frac 1 2 |\dl \psi(x, t)|^2 - \frac 2 {p+2} |\psi(x, t)|^{p + 2} \dd x.
\ee

Equation (\ref{NLS'}) is invariant under the rescaling
\be
\psi(x, t) \mapsto \alpha^{2/p} \psi(\alpha x, \alpha^2 t).
\ee
For $s_c = d/2 - 2/p$, $\dot H^{s_c}$ has the same scaling. We interpret this as meaning that (\ref{NLS'}) is $\dot H^{s_c}$-critical, for $s_c = d/2 - 2/p$. Of particular interest are the $L^2$ mass-critical exponent $p = \frac 4 n$, the $\dot H^{1/2}$-critical exponent $p= \frac 4 {n-1}$, and the $\dot H^1$ energy-critical exponent $p = \frac 4 {n-2}$.

Except for this introductory discussion, we always assume that $n=3$ and $p = 2$, in which case (\ref{NLS'}) reduces to $\dot H^{1/2}$-critical (\ref{NLS}) and (\ref{phi'}) reduces to (\ref{phi}).

\subsection{Orbital and asymptotic stability}
The stability of soliton solutions (\ref{phi'}) of the Schr\"o\-din\-ger equation (\ref{NLS'}) under small perturbations has been extensively studied.

Orbital stability was proved, by the Lyapunov method, in the $L^2$-subcritical case by Cazenave--Lions \cite{cazenave} and Weinstein \cite{wein1}, \cite{wein2}.

Soffer--Weinstein \cite{soffer1}, \cite{soffer2} introduced the modulation method to the study of asymptotic stability. Further results concerning the stability of small solitons and $L^2$-subcritical solitons belong to Pillet--Wayne \cite{pillet}, Buslaev--Perelman \cite{buslaev1}, \cite{buslaev2}, \cite{buslaev3}, Cuccagna \cite{cuc}, \cite{cuca2}, Rod\-ni\-anski--Schlag--Soffer  \cite{rod2}, \cite{rod3}, Tsai--Yau \cite{tsayau1}, \cite{tsayau2}, \cite{tsayau3}, Gang--Sigal \cite{gansig}, Cuccagna--Mizumachi \cite{cucmiz}, and Kirr--Z\u{a}rnescu \cite{kiza}, \cite{kiza2}.

Grillakis--Shatah--Strauss \cite{gril1}, \cite{gril2} proved sharp results in a general setting for soliton stability for Hamiltonian evolution equations. When applied to (\ref{NLS'}), their method shows the dichotomy between the $L^2$-subcritical case, where the ground state soliton is stable, the $L^2$-critical case, where the ground state soliton is linearly unstable, and the $L^2$-supercritical cases, where the ground state soliton is exponentially unstable.

In the $L^2$-critical case $p=\frac 4 n$, the symmetry transformations of   (\ref{NLS'}) also include the pseudoconformal symmetry. Applying it to the soliton $\phi$ (\ref{phi'}), one obtains explicit finite-time blowup solutions that blow up at a rate of $t^{-1}$.

Solutions of soliton mass form a threshold for the $L^2$-critical equation (\ref{NLS'}), $p= \frac 4 n$. Weinstein \cite{wein} showed that $H^1$ solutions $\psi$ with $M[\psi] < M[\phi]$ exist globally in time. Merle \cite{merle} showed that all threshold blowup solutions, $M[\psi] = M[\phi]$, arise from transformations of the soliton and proved scattering for $\langle x \rangle^{-1} L^2 \cap H^1$ solutions with $M[\psi] \leq M[\phi]$.

Bourgain-Wang \cite{bowa} proved the existence of a codimension-one set of solutions that blow up at a $t^{-1/2}$ rate. Perelman \cite{perelman} proved the existence of a stable blowup rate of $t^{-1/2} \log \log t$ in a neighborhood of the soliton.

Schlag \cite{schlag} extended the method of modulations to the $L^2$-supercritical case. He proved the existence of a codimension-one Lipschitz manifold of $W^{1, 1} \cap H^1$ initial data that generate asymptotically stable solutions to (\ref{NLS}). Subsequent results in this direction include those of Buslaev--Perelman \cite{buslaev1}, Krieger--Schlag \cite{krisch1}, Cuccagna \cite{cuca2}, Beceanu \cite{bec}, Marzuola \cite{mar}, and Stanislavova--Stefanov \cite{stst}.

In the $L^2$-supercritical case, negative energy initial data in $\Sigma = \langle x\rangle^{-1} L^2 \cap H^1$ leads to finite time blowup, due to the virial identity (see Glassey \cite{glassey}). For a relaxation of this condition and a survey of existing results see Sulem--Sulem \cite{sulem} and Cazenave \cite{caz2}. Berestycki--Cazenave \cite{bercaz} showed that blow-up can occur for arbitrarily small perturbations of ground state solitons such as (\ref{phi'}). The blowup is self-similar, as shown by results of Merle--Raphael \cite{mer2} and Krieger--Schlag \cite{krisch2}.

A similar result was obtained in 2006 by Kenig--Merle \cite{kenig} for the energy-critical equation (\ref{NLS'}), $p=\frac 4 {n-2}$, for $n \geq 3$. Taking radial $\psi \in \dot H^1$ data with $E[\psi] < E[\phi]$, they showed the following dichotomy: if $\|\dl \psi(0)\|_2 < \|\dl \phi\|_2$, then $\phi$ exists globally and scatters, while if $\|\dl \psi(0)\|_2 > \|\dl \phi\|_2$ and $\psi(0) \in L^2$ then $\psi$ blows up in finite time. In this regime, the equality $\|\dl \psi(0)\|_2 = \|\dl \phi\|_2$ cannot occur. The behavior of solutions at the energy threshold, $E[\psi] = E[\phi]$, was then classified by Duyckaerts--Merle \cite{duymer}: the same two cases are present, together with three others.

Following this approach, Holmer--Roudenko \cite{holrou}, Duyckaerts--Holmer--Rou\-den\-ko \cite{duhoro}, and Duyckaerts--Roudenko \cite{duyrou} established corresponding results for the $\dot H^{1/2}$-critical equation (\ref{NLS}). Their main findings may be summarized as follows:
\begin{theorem}\lb{holmer}
Assume that $\psi$ is a solution of \eqref{NLS} in the region
\be
\psi \in B =\big\{\psi \in H^1 \mid M[\psi] E[\psi] - 2 P[\psi]^2 \leq M[\phi] E[\phi]\big\},
\ee
where $\phi>0$ is a soliton given by \eqref{phi}. 
Then one of the following holds:
\begin{enumerate}
\item If $M[\psi] \|\dl \psi\|_2^2 - 2 P[\psi]^2 < M[\phi] \|\dl \phi\|_2^2$, then $\psi$ exists globally and scatters or equals a special solution, $\phi_-$, up to symmetries: Galilean coordinate changes, scaling, complex phase change, or conjugation.
\item If $M[\psi] \|\dl \psi\|_2^2 - 2 P[\psi]^2 = M[\phi] \|\dl \phi\|_2^2$, then $\psi$ equals $e^{it} \phi(\cdot, 1)$ up to symmetries.
\item If $M[\psi] \|\dl \psi\|_2^2 - 2 P[\psi]^2 > M[\phi] \|\dl \phi\|_2^2$ and $\psi \in \langle x \rangle^{-1} L^2$ is radial, then $\psi$ blows up in finite time or must equal, up to symmetries, a special solution $\phi_+$.
\end{enumerate}
\end{theorem}
The special solutions $\phi_-$ and $\phi_+$ are defined by the following properties: $\phi_-(t)$ scatters  as $t \to -\infty$ and converges at an exponential rate to $e^{it\alpha^2} \phi(x, \alpha)$ as $t \to + \infty$, while $\phi_+(t)$ blows up in finite time for $t<0$ and converges exponentially fast to a soliton as $t \to + \infty$.

Theorem \ref{theorem_1} then shows that the boundary $\partial B$ 
is not a smooth manifold. In fact, in the neighborhood of the soliton manifold $\sol$, $B$ is contained between two $\dot H^{1/2}$ hypersurfaces, $\mc N$ and $\ov {\mc N}$, that meet transversally.

The contacts $B \cap \mc N$ and $B \cap \ov{\mc N}$ occur along nine-dimensional manifolds. These are precisely the special solutions $\phi_+(t)$ and $\phi_-(t)$, at large $t$, subject to the symmetry transformations.

This fact fits well with several other natural observations. Firstly, eliminating the soliton manifold $\sol$, an eight-dimensional set, from $B$, an infinite-dimensional set, divides $B$ into two disconnected components; this certainly could not happen if $B$ were smooth.

Secondly, consider the functional that defines $B$, $F[\psi] = M[\psi] E[\psi] - 2 P[\psi]^2$. Due to the extremizing property of the soliton $\phi$ in the Gagliardo--Ni\-ren\-berg--Sobolev inequality, the first differential of $F$ at $\phi$ is identically zero: for any $h \in H^1$,
\be
dF(h) := \lim_{\epsilon \to 0} \frac {F[\phi + \epsilon h] - F[\phi]}{\epsilon} = 0.
\ee
Therefore, the tangent cone to $B$ at $\phi$ is determined by the sign of second differential of $F$, which is a quadratic form of indefinite sign; hence the lack of smoothness.

Recently, Nakanishi--Schlag \cite{nasc} have obtained a complete description of solution dynamics for (\ref{NLS}) on a neighborhood of the ground state soliton, in radial $H^1$. For an overview of their method, applied to several equations, see the monography \cite{nasc2}.

They show all solutions fall into one of nine disjoint regions, of which four are open and five are finite codimension manifolds. The five finite codimension regions lie on $\mc N \cup \ov {\mc N}$. Two of the open regions lie in the set $B$ described above. The other two open regions have higher energy than the soliton and their characterization is completely new: they contain solutions that blow up at one end and completely disperse at the other.

Directly relevant to the discussion of Theorem \ref{theorem_1} are results of Schlag \cite{schlag}, \cite{bec}, Cuccagna \cite{cuca2}, and Kirr--Zarnescu \cite{kiza}.

In \cite{schlag}, Schlag extended the method of modulation to the $L^2$-supercritical case. He proved that locally around each ground-state soliton there exists a codimension-one Lipschitz submanifold of $H^{1}(\R^3) \cap W^{1,1}(\R^3)$ of initial data that lead to global $H^1 \cap W^{1, \infty}$ solutions to (\ref{NLS}). These solutions decompose into a moving soliton and a dispersive term.

\cite{bec} showed that, for initial data on a codimension-one local Lipschitz manifold in $\Sigma = \langle x \rangle^{-1} L^2\cap H^{1}$ , there exists a global solution to (\ref{NLS}) in the same space $\Sigma$. Furthermore, the manifold is identified as the centre-stable manifold for (\ref{NLS}) within $\Sigma$. In particular, the solution stays on the manifold for some positive finite time.

Cuccagna \cite{cuca2} studied asymptotically stable solutions for the mass-supercritical Schr\"odinger equation (\ref{NLS'}) in $\R \times \R$:
\be
iu_t+u_{xx}+ |u|^p u=0,\ 5<p<\infty.
\ee
$p=5$ is the $L^2$-critical exponent in $\R$, while every exponent is $H^1$-subcritical. Starting from even $H^1$ initial data, Cuccagna obtained a set of stable solutions, from a codimension-one set near the manifold, without manifold structure.


Unlike previous results, Theorem \ref{theorem_1} holds in a critical space, $\dot H^{1/2}$, for equation (\ref{NLS}). It was not known before whether asymptotically stable manifolds exist in the critical norm or are epiphenomena related to using supercritical norms in the study of the equation. Theorem \ref{theorem_1} settles this point.

Working in a critical space also leads to several specific improvements. Firstly, we identify the asymptotically stable manifold $\mc N$ as a centre-stable manifold in the sense of \cite{bates}, with a simpler proof than in \cite{bec}. Secondly, the critical norm of the solution, as a rule, stays bounded for all time --- unlike supercritical norms, which may grow polynomially or exponentially. Thus we prove that $\mc N$ is globally-in-time invariant: solutions starting on $\mc N$ exist globally and remain on the manifold $\mc N$, for all positive time.

Finally, the fact that the manifold $\mc N$ is real-analytic raises the possibility of finding its analytic continuation beyond the neighborhood $V(\mc S)$ of the soliton manifold. It becomes interesting to identify the global object that corresponds to the local centre-stable manifold.

\subsection{Centre-stable manifolds}\lb{Chapter_1.3}
The notion of invariant (stable, centre-stable, or unstable) manifolds was introduced in the study of finite-dimensional and discrete-time dynamical systems; see Kelley \cite{alkelley} and Smale \cite{smale}.

Invariant manifolds were then used in the study of dissipative equations; see, among others, the work of Chafee--Infante \cite{chin} and Henry's monograph \cite{henry} regardining the semilinear heat equation and Keller \cite{keller} concerning the damped wave equation.

Bates--Jones \cite{bates} proved for a large class of semilinear equations that the space of solutions locally decomposes into an unstable and a centre-stable manifold. Their result is as follows. Consider a Banach space $X$ and the semilinear equation
\be\lb{117}
u_t = A u + f(u),
\ee
under the assumptions
\begin{enumerate}
\item[H1] $A:X \to X$ is a closed, densely defined linear operator that generates a $C_0$ group.
\item[H2]\lb{H2} The spectrum of $A$, $\sigma(A) = \sigma_s(A) \cup \sigma_c(A) \cup \sigma_u(A)$, decomposes into left half-plane (stable), imaginary (centre), and right half-plane (unstable) components. The stable and unstable components, $\sigma_s(A)$ and $\sigma_u(A)$, are bounded.
\item[H3]\lb{H3} The nonlinearity $f$ is locally Lipschitz, $f(0) = 0$, and for every $\epsilon>0$ there exists a neighborhood of $0 \in X$ on which $\|f(x) - f(y)\| \leq \epsilon \|x-y\|$.
\end{enumerate}
Let $X^u$, $X^c$, and $X^s$ be the $A$-invariant subspaces corresponding to $\sigma_u$, $\sigma_c$, and $\sigma_s$ and let $S^c(t)$ be the evolution generated by $A$ on $X^c$. \cite{bates} further assume that
\begin{enumerate}
\item[C1-2] $\dim X^u$ and $\dim X^s$ are finite.
\item[C3]\lb{C3} $S^c$ has subexponential growth: $\forall \rho>0$ $\exists M>0$ such that $\|S^c(t)\| \leq M e^{\rho|t|}$.
\end{enumerate}

Let $\Upsilon$ be the flow on $X$ generated by (\ref{117}). $\mc N \subset U$ is called $t$-invariant if $\Upsilon(s)v \in U$ for all $s \in [0, t]$ implies that $\Upsilon(s)v \in \mc N$ for $s \in [0, t]$.

\begin{definition}\lb{unstable}
Let the \emph{unstable} manifold $W^u \subset U$ be the set of solutions that remain in $U$ for all $t<0$ and decay exponentially as $t \to -\infty$:
$$
W^u = \{u \in U \mid \forall t \leq 0\ \Upsilon(t) u \in U,\ \exists C_1>0\ \forall t \leq 0\ \|\Upsilon(t) u\|_X \les e^{C_1 t} \}.
$$
\end{definition}
Also consider the canonical direct sum spectral projection $\pi^{cs}$ onto the centre-stable part of the spectrum: $\pi^{cs}(X) = X^c \oplus X^s$.
\begin{definition}
A \emph{centre-stable} manifold $\mc N \subset U$ is a Lipschitz manifold (i.e.\ parametrized by Lipschitz maps) such that $\mc N$ is $t$-invariant relative to $U$, $\pi^{cs}(\mc N)$ contains a neighborhood of $0$ in $X^c \oplus X^s$, and $\mc N \cap W^u = \{0\}$.
\lb{centr}
\end{definition}

The conclusion of \cite{bates} is then
\begin{theorem}
Under assumptions H1-H3 and C1-C3, locally around $0$, there exist an unstable Lipschitz manifold $W^u$ tangent to $X^u$ at $0$ and a centre-stable manifold $W^{cs}$ tangent to $X^{cs}$ at $0$.
\lb{t1}
\end{theorem}

Gesztesy--Jones--Latushkin--Stanislavova \cite{ges} 
established a spectral mapping theorem for the semigroup obtained by linearizing the equation
\be\lb{123}
iu_t - \Delta u - f(x, |u|^2)u - \beta u = 0
\ee
around an exponentially decreasing standing wave solution. In particular, the spectral condition H2 then implies the semigroup norm estimate C3.

By \cite{ges}, all the conditions of \cite{bates} are met for (\ref{eq_2188}), leading to the existence of a centre-stable manifold $\mc N_{BaJo}$ in this setting.

Indeed, in the Banach algebra $H^s \subset L^{\infty}$, $s>3/2$, the Hamiltonian $H$ is a closed, densely defined operator with the required spectral properties H1--H3 and C1--C3. In particular, the nonlinearity $N(Z_1, W(\pi_0(t)))$ has the Lipschitz property with arbitrarily small constant:
$$
\|N(Z_1, W(\pi_0)) - N(Z_2, W(\pi_0))\|_{H^s} \les \max(\|Z_1\|_{H^s}, \|Z_2\|_{H^s}) \|Z_1-Z_2\|_{H^s}.
$$
This shows that property H3 holds in $H^s$, for $s>3/2$. Thus, there exists a local centre-stable manifold for (\ref{123}) in $H^s$, $s>3/2$.

The global existence of solutions on the centre-stable manifold is not addressed in \cite{bates} and \cite{ges}. The $t$-invariance property of \cite{bates} means that a solution starting on the centre-stable manifold $\mc N$ remains there for as long as it stays small in norm. However, once a solution on $\mc N$ leaves the specified neighborhood $U$, its behavior can no longer be known by this method and its existence is not guaranteed any longer.

On the other hand, knowing, for reasons specific to the equation, that solutions on $\mc N$ remain in a small neighborhood of $0$ for all time implies the orbital or asymptotic stability of solutions on $\mc N$.

Several existing results can be interpreted in light of this theory. Chafee-Infante's \cite{chin}, Henry's \cite{henry} or Keller's \cite{keller} results required that the equation should be dissipative or contain a damping term, but held globally in time. Then, the stable manifold has finite codimension and the centre and stable manifolds are finite-dimensional.

Schlag \cite{schlag}, in the absence of a damping term, proved a global asymptotic stability result, but the manifold was not time-invariant, because it was defined in $W^{1, 1} \cap H^1$.

The centre-stable manifold of \cite{bec} exists in the $\Sigma = H^1 \cap \langle x \rangle^{-1} L^2$ norm, which grows linearly with $t$ for Schr\"{o}dinger's equation (\ref{NLS}). More generally, the same method works in $\Sigma^s = \dot H^s \cap \langle x \rangle^{-s} L^s$, $s \in (3/4, 1]$, whose norm grows at a rate of $\langle t \rangle^s$. The centre-stable manifold is locally-in-time invariant in $\Sigma^s$, but solutions may leave the manifold after finite time because of the $\Sigma^s$ norm growth.

On the other hand, the critical $\dot H^{1/2}$ norm of solutions does not grow with time, so the manifold constructed in $\dot H^{1/2}$ is globally-in-time invariant. This paper identifies the critical centre-stable manifold $\mc N$ for (\ref{NLS}) in $\dot H^{1/2}$ and shows that solutions starting on $\mc N$ remain on $\mc N$ for all time.

It is important to note that $\dot H^{1/2}$ is not an algebra and that the conditions of Theorem \ref{t1} of \cite{bates} are not met in $\dot H^{1/2}$. Indeed, if $\psi \in \dot H^{1/2}$, it does not follow that $|\psi|^2 \psi \in \dot H^{1/2}$, much less that this nonlinearity is Lipschitz continuous, which is Condition H3 of Theorem \ref{t1}.

However, even though Condition H3 fails in $\dot H^{1/2}$, the conclusion of \cite{bates} --- the existence of a centre-stable manifold $\mc N$ --- still holds.

\subsection{Setting and notations}\lb{Chapter_1.4}

Equation (\ref{NLS}) admits periodic solutions $e^{it\alpha^2} \phi(x, \alpha)$, where
\be\lb{derrick-pohozaev}
\phi = \phi(x, \alpha) = \alpha \phi(\alpha x, 1)
\ee
is a solution of the semilinear elliptic equation (\ref{phi})
$$
-\Delta \phi + \alpha^2\phi = \phi^3.
$$
We consider positive $L^2$ solutions of (\ref{phi}), called ground states. Ground states are unique up to translation, radially symmetric, smooth, and exponentially decreasing. Berestycki--Lions proved the existence of ground states in \cite{bere} and showed they are infinitely differentiable and exponentially decaying. In a more general context, (\ref{derrick-pohozaev}) becomes the Derrick--Pohozaev identity, see \cite{bere}.

Uniqueness of ground states was established by Coffman \cite{coffman} for cubic and Kwong \cite{kwong} and McLeod--Serrin \cite{mcl} for more general nonlinearities.

Equation (\ref{NLS}) is invariant under its symmetry transformations, which are translations, boost, rescaling, and gauge transformations.  Since $\psi$ is complex-valued, gauge transformations have the form $f \mapsto e^{i\Gamma} f$, $\Gamma \in \R$:
\be
\frak g(t)\, f(x, t) := e^{i(\Gamma+v \cdot x-t|v|^2)} \alpha f(\alpha (x-2tv-D), \alpha^2 t).
\lb{coord}
\ee
When $\psi(t) \in H^{1/2}$ is a solution to (\ref{NLS}), so is $\frak g(t) \psi(t)$. However, boost transformations do not preserve $\dot H^{1/2}$.

Applying symmetry transformations to the ground state soliton $f(x, t) = e^{it} \phi(x, 1)$, we obtain an eight-parameter family (\ref{free_sol}) of solutions to (\ref{NLS}):
$$
\frak g(t)\, e^{it} \phi(x, 1) = e^{i(\Gamma + v \cdot x - t|v|^2 + \alpha^2 t)} \phi(x-2t v- D, \alpha).
$$
Such solutions are called solitons --- or standing waves if they have no boost, i.e.\ $v=0$.


We find solutions $\Psi = \bpm \psi \\ \ov \psi \epm$ to (\ref{NLS}) that approach $\sol$ asymptotically as $t$ goes to infinity. Such solutions have the form
\be
\Psi = W_{\pi}(x, t) + R(x, t) = \bpm w_{\pi}(x, t) \\ \ov w_{\pi}(x, t) \epm + \bpm r(x, t) \\ \ov r(x, t) \epm,
\ee
where $W_{\pi} = \bpm w_{\pi} \\ \ov w_{\pi} \epm$ is a moving soliton and $R=\bpm r \\ \ov r \epm$ is a small correction term that disperses like the solution of the free Schr\"{o}dinger equation as $t \to +\infty$.



We parametrize the moving soliton $w_{\pi}(t)$ by setting, following (\ref{1.1}),
$$\begin{aligned}
w_{\pi}(t) & = e^{i \theta(x, t)} \phi\big(x - y(t), \alpha(t)\big) \\
& \textstyle = e^{i(\Gamma(t) + \int_0^t (\alpha^2(s) - |v(s)|^2) \dd s + v(t) x)} \phi(x - 2\int_0^t v(s) \dd s - D(t), \alpha(t)).
\end{aligned}$$
The parameters of this formula,
\be
\alpha(t),\ \Gamma(t),\ v(t) = \big(v_1(t), v_2(t), v_3(t)), \text{ and } D(t) = (D_1(t), D_2(t), D_3(t)\big),
\ee
are called \emph{modulation parameters} and $\pi(t) = (\alpha(t), \Gamma(t), v(t), D(t))$ is called the \emph{parameter path} or \emph{modulation path}. For each $t$, $\pi(t) \in \R^8$ contains eight parameters, since $v$, $D \in \R^3$.

Due to the nonlinear interaction between $w_{\pi}$ and $r$, the modulation parameters are not constant in general; they are time-dependent. However, in the course of the proof they are not allowed to vary too much. A minimal condition, which we impose henceforth, is that $\pi$ has bounded variation and
\be
\dot \alpha, \dot \Gamma, \dot v_k, \dot D_k \in L^1_t
\ee
are small in this norm. We assume no stronger rate of decay. This implies that the modulation parameters converge as $t \to \infty$ to final values, but at no particular rate, and that their range is contained within arbitrarily small intervals.

\subsection{Proof outline}\lb{sec_1.5}
Linearizing the equation (\ref{NLS}) around a moving soliton $w_{\pi^0}(t)$ driven by a modulation path $\pi^0(t)$ as in (\ref{1.1}) leads to the time-dependent Hamiltonian
$$\begin{aligned}
H_{\pi^0}(t) &= \bpm \Delta + 2|w_{\pi^0}(t)|^2 & (w_{\pi^0}(t))^2 \\ -(\ov w_{\pi^0}(t))^2 & -\Delta - 2|w_{\pi^0}(t)|^2 \epm.
\end{aligned}$$
By Lemma \ref{lemma1.1}, we reduce the study of $i\partial_t + H_{\pi^0}(t)$ to that of $i \partial_t + H$, where
\be\lb{1.29}
H = \bpm \Delta - 1 + 2\phi^2(\cdot, 1) & \phi^2(\cdot, 1) \\ -\phi^2(\cdot, 1) & -\Delta + 1 - 2\phi^2(\cdot, 1) \epm.
\ee

The spectrum of $H$ is completely known. Marzuola--Simpson gave a computer-assisted proof  in \cite{marsim} to the fact that $\sigma_{ac}(H) = (-\infty, -1] \cup [1, \infty)$ contains no embedded eigenvalues. Costin--Huang--Schlag \cite{cohusc} proved the spectral gap property analytically. In Section \ref{spectru} we give a full description of the spectral properties of $H$.

The proof of Theorem \ref{theorem_1} is based on a fixed point principle. We solve the linearized equation (\ref{ec_liniara}), jointly with the linearized modulation equations, in the space $X$ (\ref{X}). Considering the solution $(R, \pi)$ as a function $\Epsilon$ of the auxiliary variables $(R^0, \pi^0)$, we show that $\Epsilon$ is a contraction for fixed initial data.

We prove the stability of a small ball in $X$ under $\Epsilon$, Proposition \ref{prop9}, then prove that $\Epsilon$ is a contraction on this ball in the space $Y$, Lemma \ref{lemma_5}. This is enough to conclude the existence of a fixed point for $\Epsilon$, in Proposition \ref{prop27}, which is an asymptotically stable solution of (\ref{NLS}).

At this point, the solution size $\|R\|_{L^{\infty}_t \dot H^{1/2}_x \cap L^2_t \dot W^{1/2, 6}_x} + \|\dot \pi\|_{L^1_t \cap L^{\infty}_t}$ depends on that of $\|R(0)\|_{\dot H^{1/2}}$. Lemma \ref{lemma_10} implies that $\|R(0)\|_{\dot H^{1/2}}$ is comparable to the distance $d_{\dot H^{1/2}}(\psi(0), \sol)$.

Proposition \ref{prop_9} establishes the continuous dependence of $\Psi$ on initial data. The scattering of the solution is shown in Section \ref{sect_sc}.

We define the manifold $\mc N$ in Definition \ref{def3} and prove its local uniqueness in Proposition \ref{prop_2.13} and its invariance by Corollary \ref{stable}. $\mc N$ and its embedding into $\dot H^{1/2}$, as well as solutions' dependence on initial data, are real-analytic by Proposition \ref{prop_an}.

Finally, in Theorem \ref{batjon} we prove that $\mc N$ is a centre-stable manifold for (\ref{NLS}) in the sense of Definition \ref{centr}, see Bates--Jones \cite{bates}.

\subsection{Linear results}
In establishing the main nonlinear result, Theorem \ref{theorem_1}, a crucial ingredient is an endpoint Strichartz estimate for the time-dependent linearized Schr\"{o}dinger equation
\be\lb{1.30}
i \partial_t Z - i v(t) \dl Z + A(t) \sigma_3 Z + H Z = F,\ Z(0) \text{ given}.
\ee
Here $Z = \bpm z_1 \\ z_2 \epm$ and, in general,
\be
H = H_0 + V = \bpm \Delta-\mu & 0 \\ 0 & -\Delta+\mu \epm + \bpm W_1 & W_2 \\ -W_2 & -W_1 \epm,\ \mu > 0.
\lb{1.17}
\ee
$W_1$ and $W_2$ are real-valued and of Schwartz class.

The Hamiltonian of (\ref{1.30}) has the form
$$
- i v(t) \dl + A(t) \sigma_3 + H,
$$
hence is nonselfadjoint and time-dependent. Since the potential $V$ appears by linearizing (\ref{NLS}) around a ground state soliton $\phi$, see (\ref{ec_liniara}), $V$ is smooth and exponentially decaying. In fact, the real analyticity of $V$ gets used in proving the analiticity of the soliton manifold $\sol$ and then of the centre-stable manifold $\mc N$.

In this paper we prove endpoint Strichartz estimates for (\ref{1.30}), of the form
\be
\|Z\|_{L^2_t \dot W^{1/2, 6}_x \cap L^{\infty}_t \dot H^{1/2}_x} \les \big(\|Z(0)\|_2 + \|F\|_{L^2_t \dot W^{1/2, 6/5}_x + L^1_t \dot H^{1/2}_x}\big).
\ee

Keel--Tao \cite{tao} proved endpoint Strichartz estimates for free Schr\"{o}dinger and wave equations and, more generally, showed that certain kernel decay bounds, together with $L^2$-boundedness of the evolution, imply endpoint Strichartz estimates.

In the nonselfadjoint case (\ref{1.30}), these dispersive estimates have to be proved anew. They do not follow in the same manner as or from the selfadjoint case, where, for example, the unitarity of the time evolution immediately implies the $L^2$ boundedness.

In \cite{schlag}, Schlag proved $L^1 \to L^{\infty}$ dispersive estimates for the Schr\"{o}dinger equation with a nonselfadjoint Hamiltonian, as well as non-endpoint Strichartz estimates. Erdo\^{g}an and Schlag \cite{erdsch2} proved $L^2$ bounds for the evolution as well. \cite{bec} proved endpoint Strichartz estimates in the nonselfadjoint case following the method of Keel--Tao. Finally, Cuccagna--Mizumatchi \cite{cucmiz} retrieved all previous results as a simple consequence of the boundedness of wave operators in the nonselfadjoint case.

A further complication appears in (\ref{1.30}) from the time-dependent Hamiltonian terms $A(t) \sigma_3$ and $iv(t) \dl$. 
Here $Z$ is the solution to (\ref{1.30}), $A(t) = \alpha^2(t) - \alpha^2(\infty)$ is the scale-dependent oscillation frequency, and $v(t)-v(\infty)$ is the translation velocity.

Instead of treating $A(t) \sigma_3$ and $iv(t) \dl$ as inhomogenous terms, we make them part of the time-dependent Hamiltonian for which we prove the endpoint Strichartz estimates.

This technique leads to the following Strichartz estimates, which improve upon and are distinct from previous estimates:
\begin{theorem}\lb{theorem_1.4}
Consider equation (\ref{1.30}), for $H = H_0 + V$ as in (\ref{1.17}) and a potential $V$ not necessarily real-valued:
\be\nonumber
i \partial_t Z - i v(t) \dl Z + A(t) \sigma_3 Z + H Z = F,\ Z(0) \text{ given},
\ee
\be\nonumber
H = \bpm \Delta - \mu & 0 \\ 0 & -\Delta + \mu \epm + \bpm W_1 & W_2 \\ -W_2 & -W_1 \epm.
\ee
Assume that $V=V_1 V_2$ and $V_1$, $V_2 \in \mc S$. Also assume that $\|A\|_{\infty}$ and $\|v\|_{\infty}$ are sufficiently small, in a manner that depends on $V$, and that $\sigma(H_0)$ contains no exceptional values of $H$. Then
\be
\|P_c Z\|_{L^{\infty}_t \dot H^{1/2}_x \cap L^2_t \dot W^{1/2, 6}_x} \leq C \Big(\|Z(0)\|_{\dot H^{1/2}} + \|F\|_{L^1_t \dot H^{1/2}_x + L^2_t \dot W^{1/2, 6/5}_x}\Big).
\ee
\end{theorem}
Here the assumptions are not sharp; for sharp results, see \cite{bec3}. $P_c$ is the projection on the continuous spectrum of $H$.
\section{The Nonlinear Result}
\subsection{Notations and preliminary results}

We linearize the original equation (\ref{NLS}) around a moving soliton $w_{\pi}(t)$. Substituting $\psi(t) = w_{\pi}(t) + r(t)$ in (\ref{NLS}),
\be
i\partial_t (w_{\pi} + r) + \Delta (w_{\pi} + r) + (\ov w_{\pi} + \ov r) (w_{\pi}+r)^2 = 0.
\ee
In keeping with (\ref{1.1}), the soliton $w_{\pi}(t)$ is described by
\be\textstyle\nonumber
w_{\pi}(t)=e^{i(\Gamma(t) + \int_0^t (\alpha^2(s) - |v(s)|^2) \dd s + v(t) \cdot x)} \phi\big(x - 2\int_0^t v(s) \dd s - D(t), \alpha(t)\big).
\ee
Note that $w_{\pi}(t)$ depends on the values $\pi$ takes on $[0, t]$. If we define $w: \R^8 \to \sol$ by 
\be\lb{wpi}
w(p)=w(\alpha, \Gamma, v, D) = e^{i(\Gamma + v \cdot x)} \phi\big(x - D, \alpha\big),
\ee
then $w_{\pi}(t) \ne w(\pi(t))$. In fact,
$$
w_\pi(t) = w\Big(\alpha(t), \Gamma(t) + \int_0^t (\alpha^2(s) - |v(s)|^2) \dd s, v(t), D(t) + 2\int_0^t v(s) \dd s\Big).
$$
Expanding the equation accordingly, note that
\begin{align}\lb{2.4}
\partial_t w_{\pi} &= (\dot \Gamma + \alpha^2 - v^2) \partial_{\Gamma} w_{\pi} + \dot \alpha \partial_{\alpha} w_{\pi} + \dot v \partial_v w_{\pi} - (2v + \dot D) \partial_D w_{\pi}\\
\Delta w_{\pi} &= \Delta e^{i \theta(x, t)} \phi(x-y(t), \alpha(t)) + 2 \dl e^{i \theta(x, t)} \dl \phi(x-y(t), \alpha(t)) + \nonumber\\
&+ e^{i \theta(x, t)} \Delta \phi(x-y(t), \alpha(t)) \\
&= (\alpha^2 -v^2) w_{\pi} + 2i v \dl w_{\pi} - |w_{\pi}|^2 w_{\pi}. \nonumber
\end{align}
Here we use the following notation: given a soliton $w=w(p)$ as in (\ref{wpi}), $\partial_{\Gamma} w(p)$, $\partial_{\alpha} w(p)$, $\partial_{D_k} w(p)$, and $\partial_{v_k} w(p)$ are the partial derivatives of $w(p)$ with respect to these parameters. They span the tangent space $T_{w(p)} \sol$ to the soliton manifold $\sol$ at the point $w(p)$:
$$
\partial_{\Gamma} w(p) = iw(p),\ \partial_{\alpha} w(p) = \partial_{\alpha} w(p),\ \partial_{D_k}w(p) = \partial_{x_k} w(p),\ \partial_{v_k} w(p) = i x_k w(p).
$$
We thus define the differential $d w(p)$ of the map $p \mapsto w=w(p): \R^8 \to \sol$. The differential $d w(p): T_p \R^8 \to T_{w(p)} \sol$ maps a tangent vector $\delta p = (\delta \alpha, \delta \Gamma, \delta v_k, \delta D_k) \in T_p \R^8$ to one in $T_{w(p)} \sol$:
$$
d w\, \delta p = \partial_{\Gamma} w\, \delta \Gamma + \partial_{\alpha} w\, \delta \alpha + \partial_{D_k} w\, \delta D_k + \partial_{v_k} w\, \delta v_k.
$$

The explicit computation (\ref{2.4}) shows the cancellation of the cubic term containing $w_{\pi}$. The equation for $r$ becomes
\begin{multline}\lb{2.7}
i\partial_t r + \Delta_x r + i(d w_{\pi}(t))\dot \pi(t) + (|r|^2 r + r^2 \ov w_{\pi} + 2|r|^2 w_{\pi} + 2r|w_{\pi}|^2 + \ov r w_{\pi}^2) = 0.
\end{multline}
Here $d w_\pi(t)$ is the differential $d$ evaluated at the point $w_\pi(t)$. Since as noted previously $w_{\pi}(t) \ne w(\pi(t))$, it follows that $d w_{\pi}(t)\, \dot \pi(t) \ne \partial_t w_{\pi}(t)$.

(\ref{2.7}) is a Schr\"{o}dinger equation in $r$. By conjugating (\ref{2.7}), we obtain an equivalent equation for $\ov r$:
\be\lb{2.8}
i\partial_t \ov r - \Delta \ov r + i \ov{d w_{\pi}} \dot \pi - (|r|^2 \ov r + (\ov r)^2 w_{\pi} + 2|r|^2 \ov w_{\pi} + 2\ov r|w_{\pi}|^2 + r (\ov w_{\pi})^2) = 0.
\ee
Since both (\ref{2.7}) and (\ref{2.8}) involve both $r$ and $\ov r$, it is most convenient to solve them together as a system or, rather, to see the pair of equations as just one equation concerning the column vector $R = \bpm r \\ \ov r \epm$.

In the sequel we employ column vectors consisting of a complex-valued function (written in lowercase) and its conjugate, e.g.
$$
\Psi = \bpm \psi \\ \ov \psi \epm,\ R = \bpm r \\ \ov r \epm,\ Z = \bpm z \\ \ov z \epm, \text{ etc.}
$$
As in \cite{schlag}, subsequent computations preserve this symmetry. Indeed, henceforth all column vectors are of the form $F = \bpm f \\ \ov f \epm$, or, otherwise put, $\sigma_2 \ov F = F$, where $\sigma_2 = \bpm 0 & 1 \\ 1 & 0 \epm$.

This fact makes the following dot-product, which we use henceforth, real-valued:
$$
\langle F, G \rangle := \langle f, g \rangle + \langle \ov f, \ov g \rangle := \int_{\R^3} \big(f(x) \ov g(x) + \ov f(x) g(x)\big) \dd x.
$$

The terms of (\ref{2.7}) and (\ref{2.8}) that are linear in $r$ and $\ov r$ give rise to a time-dependent, nonselfadjoint two-by-two matrix Hamiltonian
\be\lb{2.10}
H_{\pi}(t) = \bpm \Delta + 2|w_{\pi}(t)|^2 & w_{\pi}(t)^2 \\ -\ov w_{\pi}(t)^2 & -\Delta - 2|w_{\pi}(t)|^2 \epm = \Delta \sigma_3 + V_{\pi}(t),
\ee
whereas the other terms collect into the homogenous right-hand side of the equation:
\be\lb{F}
F = \bpm i(d w_{\pi})\dot \pi + |z|^2 z + z^2 \ov{w}_{\pi} + 2|z|^2 w_{\pi} \\ i \, \ov{d w_{\pi}} \dot \pi - |z|^2 \ov z - (\ov z)^2 w_{\pi} - 2|z|^2 \ov w_{\pi} \epm.
\ee
Let $\Sol  = \{W(p) = \bpm w(p) \\ \ov w(p) \epm \mid w(p) \in \sol\}$.
Given $W \in \Sol$, consider its partial derivatives with respect to the modulation parameters $f \in \{\alpha, \Gamma, v_k, D_k\}$, $\partial_f W = \bpm \partial_f w \\ \ov {\partial_f w} \epm$:
$$\begin{aligned}
&\partial_{\Gamma}W(p) = \bpm i w(p) \\ -i \ov w(p) \epm = i \sigma_3 W(p),\ \partial_{\alpha} W(p) = \partial_{\alpha} W(p),\\
&\partial_{D_k} W(p) = \partial_{x_k} W(p),\ \partial_{v_k} W(p) = i \sigma_3 x_k W(p).
\end{aligned}$$

Here $\sigma_3$ is the third Pauli matrix $\sigma_3 = \bpm 1 & 0 \\ 0 & -1 \epm$.

Define the differential $d W(p): T_p \R^8 \to T_{W(p)} \Sol$ of the map $p \mapsto W(p)$:
$$\begin{aligned}
d W\, \delta p &= \partial_{\Gamma} W\, \delta \Gamma + \partial_{\alpha} W\, \delta \alpha + \partial_{D_k} W\, \delta D_k + \partial_{v_k} W\, \delta v_k.
\end{aligned}$$
Let $W_{\pi}(t) = \bpm w_{\pi}(t) \\ w_{\pi}(t) \epm$. The first term in (\ref{F}) is then $i d_{\pi} W_{\pi}(t)\, \dot \pi(t) = \bpm i d_\pi w_{\pi}(t)\, \dot \pi(t) \\ i\, \ov{d_\pi w_{\pi}(t)}\, \dot \pi(t) \epm$.\\
The remaining nonlinear term in (\ref{F}) is
$$
N(R, W_{\pi}(t)) = \bpm -|r|^2 r - r^2 \ov{w}_{\pi}(t) - 2|r|^2 w_{\pi}(t) \\ |r|^2 \ov {r} + \ov {r}^2 w_{\pi}(t) + 2|r|^2 \ov w_{\pi}(t) \epm.
$$
The equation fulfilled by the vector variable $R$ becomes
\be\lb{rr}
i \partial_t R - H_{\pi}(t) R = F(t),\ F = -i d_{\pi} W_{\pi}\, \dot \pi + N(R, W_{\pi}).
\ee

To this equation for $R$ we join the modulation equations (\ref{mod}) that determine $\pi$, hence the term $-i d_{\pi} W_{\pi}\, \dot \pi$.

For reference, define the cotangent vectors $\partial_f^* W(p) \in T_{W(p)} \Sol$
\be\begin{aligned}\lb{2.19}
\partial^*_{\alpha} W &= i\sigma_3 \partial_{\Gamma} W, &\partial^*_{\Gamma} W &= i\sigma_3 \partial_{\alpha}W, \\
\partial^*_{v_k} W &= i\sigma_3 \partial_{D_k} W, &\partial^*_{D_k}W &= i\sigma_3 \partial_{v_k} W.
\end{aligned}\ee
These allow one to define the following differential form $d^*:T^* \R^8 \to T^* \Sol$:
$$\begin{aligned}
&d^* W(p): \set R^8 \to T^* \Sol,\ p \mapsto d^*_p W(p),\\
&d^* W\, \delta \pi = \partial^*_{\Gamma} W\, \delta \Gamma + \partial^*_{\alpha} W\, \delta \alpha + \partial^*_{D_k} W\, \delta D_k + \partial_{v_k}^* W\, \delta v_k.
\end{aligned}$$
Also consider the second differential $d d^*\, W(p): T T^* \set R^8 \to T T^*\Sol$, defined as the differential of $d^*\, W(p)$, for $f \in \{\alpha, \Gamma, v_k, D_k\}$.

At each time $t$ and for all $f \in \{\alpha, \Gamma, v_k, D_k\}$ we impose the \emph{orthogonality condition}
\be\lb{2.20}
\langle R(t), \partial^*_{f} W_{\pi}(t) \rangle = 0.
\ee
This condition arises as follows: to each $p = (\alpha, \Gamma, v_k, D_k) \in \set R^8$ associate the Hamiltonian
$$\begin{aligned}
H(W(p)) &= \bpm \Delta + 2|w(p)|^2 & w(p)^2 \\ -w(p)^2 & -\Delta -2|w(p)|^2 \epm + 2iv \dl - (\alpha^2-|v|^2) \sigma_3.
\end{aligned}$$
Note that $H(W_{\pi}(t)) = H_{\pi}(t) + 2iv(t) \dl - (\alpha(t)^2-|v(t)|^2) \sigma_3$. Condition (\ref{2.20}) states that $R(t)$ has no component in the zero eigenspace of $H(W_{\pi}(t))$.

As an important aside, due to the orthogonality condition, the momentum of $\psi$ decomposes into the momentum of $w$ and the momentum of $r$:
$$\begin{aligned}
P[\psi] &= \int_{\R^3} i\dl \psi(x, t) \ov {\psi(x, t)} \dd x \\
&= \int_{\R^3} i \dl w \ov w \dd x + \int_{\R^3} i \dl r \ov r \dd x + \int_{\R^3} i (\dl w \ov r + \dl r \ov w) \dd x \\
&= P[w] + P[r] + \int_{\R^3} i (\dl w \ov r - \dl \ov w r) \dd x = P[w] + P[r] + \sum_{k=1}^3 \langle r, \partial^*_{v_k} W \rangle \\
&= P[w(t)] + P[r(t)].
\end{aligned}$$
$P[\psi]$ is a constant of motion, so if (\ref{2.20}) holds then $P[w(t)] + P[r(t)]$ is constant.

Taking the time derivative in (\ref{2.20}) leads to modulation equations. We use the constant $\|\phi\|_{\dot H^{1/2}}$ within the formulae; recall $\phi$ is given by (\ref{phi}).

\begin{lemma}[The modulation equations]\lb{lemma_6} \eqref{2.20} holds for all $t$ if and only if \eqref{2.20} is satisfied at $t=0$ and if the following equations are true:
\be\begin{aligned}
\dot \alpha &= 2\alpha^2\|\phi\|_{\dot H^{1/2}}^{-2} \big(\langle R, (d_{\pi} \partial^*_{\alpha} W_{\pi})\, \dot \pi \rangle - i\langle N(R, W_{\pi}), \partial^*_{\alpha} W_{\pi}\rangle\big)\\
\dot \Gamma &= 2\alpha^2\|\phi\|_{\dot H^{1/2}}^{-2} \big(\langle R, (d_{\pi} \partial^*_{\Gamma}W_{\pi})\, \dot \pi\rangle - i\langle N(R, W_{\pi}), \partial^*_{\Gamma}W_{\pi}\rangle\big)\\
\dot v_k &= \alpha\|\phi\|_{\dot H^{1/2}}^{-2} \big(\langle R, (d_{\pi} \partial^*_{v_k}W_{\pi})\, \dot \pi\rangle - i\langle N(R, W_{\pi}), \partial^*_{v_k}W_{\pi}\rangle\big)\\
\dot D_k &= \alpha\|\phi\|_{\dot H^{1/2}}^{-2} \big(\langle R, (d_{\pi} \partial^*_{D_k}W_{\pi})\, \dot \pi\rangle - i\langle N(R, W_{\pi}), \partial^*_{D_k}W_{\pi}\rangle\big).
\lb{mod}\end{aligned}\ee
\end{lemma}
\begin{proof}
Differentiating (\ref{2.20}) with respect to $t$, we obtain
$$
\langle R, \partial_t \partial^*_f W_{\pi} \rangle = - \langle \partial_t R, \partial^*_f \rangle,
$$
where $\partial_t R$ is expressed by (\ref{rr}):
$$
\langle R, \partial_t \partial^*_f W_{\pi} \rangle = - \big\langle i H_{\pi}(t) R + d_{\pi} W_{\pi} \dot \pi + i N(R, W_{\pi}), \partial^* f \big\rangle.
$$
We process each term separately. For every $W \in \Sol$
\be\begin{aligned}
\langle \partial_{\alpha}W, \partial^*_f W \rangle &= \frac 1 {4 \alpha} \|W\|_2^2 \text{ if } f=\alpha \text{ and zero otherwise}\\
\langle \partial_{\Gamma}W, \partial^*_f W \rangle &= \frac 1 {4 \alpha} \|W^{}\|_2^2 \text{ if } f=\Gamma \text{ and zero otherwise}\\
\langle \partial_{D_k}W, \partial^*_f W \rangle &= \frac 1 2 \|W^{}\|_2^2 \text{ if } f=D_k \text{ and zero otherwise}\\
\langle \partial_{v_k}W, \partial^*_f W \rangle &= \frac 1 2 \|W^{}\|_2^2 \text{ if } f=v_k \text{ and zero otherwise}.
\end{aligned}\ee
Moreover, $\|W\|_2^2 = 2\alpha^{-1} \|\phi\|_{\dot H^{1/2}}^2$. Thus
$$\begin{aligned}
&\langle d_{\pi} W_{\pi} \dot \pi, \partial^*_f W \rangle = \frac 1 {2\alpha^2} \|\phi\|_{\dot H^{1/2}}^2 \dot f,\ f \in \{\alpha, \Gamma\};\\
&\langle d_{\pi} W_{\pi} \dot \pi, \partial^*_f W \rangle = \frac 1 {\alpha} \|\phi\|_{\dot H^{1/2}}^2 \dot f,\ f \in \{D_k, v_k\}.
\end{aligned}$$
Furthermore,
$$
\partial_t \partial^*_fW_{\pi} = (d \partial^*_fW_{\pi})\, \dot \pi + \big(H^*(W_{\pi}) - H_{\pi}^*\big) \partial^*_fW_{\pi}.
$$
Finally,
\be\begin{aligned}\lb{2.22}
H^*(W) \partial^*_{\alpha}W &= 0, &H^*(W) \partial^*_{\Gamma}W &= -2i \partial^*_{\alpha}W, \\
H^*(W) \partial^*_{v_k}W &= 0, &H^*(W) \partial^*_{D_k}W &= -2i \partial^*_{v_k}W.
\end{aligned}\ee
This leads to (\ref{mod}).
\end{proof}

Let
\be\begin{aligned}
L_{\pi^{}} R = &2\alpha^2 \sum_{f \in \alpha, \Gamma} \|\phi\|_{\dot H^{1/2}}^{-2} \langle R, (d_{\pi} \partial^*_fW_{\pi}) \dot \pi \rangle \partial_f W_{\pi} \\
+ & \alpha \sum_{f \in \{v_k, D_k\}} \|\phi\|_{\dot H^{1/2}}^{-2} \langle R, (d_{\pi} \partial^*_fW_{\pi}) \dot \pi \rangle \partial_f W_{\pi}
\end{aligned}\ee
and
\be\begin{aligned}
N_{\pi}(R, W_{\pi}) = & 2\alpha^2 \sum_{f \in \{\alpha, \Gamma\}} \|\phi\|_{\dot H^{1/2}}^{-2} i\langle N(R, W_{\pi}), \partial^*_fW_{\pi} \rangle \partial_{f}W_{\pi} \\
+ & \alpha\sum_{f \in \{v_k, D_k\}} \|\phi\|_{\dot H^{1/2}}^{-2} i\langle N(R, W_{\pi}), \partial^*_fW_{\pi}\rangle \partial_{f} W_{\pi}.
\end{aligned}\ee

\begin{observation}
The modulation equations (\ref{mod}) are equivalent to
\be
d_{\pi} W_{\pi}\, \dot \pi = L_{\pi^{}} R - i N_{\pi}(R, W_{\pi}).
\lb{modula}
\ee
\end{observation}
$L_{\pi} R$ represents the part that is linear in $R$ and $N_{\pi}(R, W)$ represents the nonlinear component $\langle N(R, W), \partial^*_fW_{\pi}\rangle$.

Let $P_0(W_{\pi}(t))$ be the zero spectral projection of $H(W_{\pi}(t))$; see Section \ref{spectru}. Then $N_{\pi}(R, W_{\pi}) = P_0(W_{\pi}) N(R, W_{\pi})$ and $L_{\pi} R$ also has a similar expression.

\subsection{The linearized equation} At this point we linearize the equation (\ref{rr}). We introduce auxiliary functions $R^0$ and $\pi^0$ to represent all terms that are quadratic or cubic in $R$ or $\pi$. We keep linear, first-order terms in the unknowns $R$ and $\pi$ for which we solve the equation. The equation becomes linear in $R$ and $\pi$ and quadratic and cubic in $R^0$ and $\pi^0$.


Starting with $\pi^0$, we construct the moving soliton $W_{\pi^0}$, its partial derivatives $\partial_f W_{\pi^0}$, cotangent vectors $\partial^*_f W_{\pi^0}$, and the differentials $d W_{\pi^0}$ and $d \partial^*_f W_{\pi^0}$ according to the given definitions:
\be\begin{aligned}\lb{w0}
\pi^0(t) &= (\alpha^0(t), \Gamma^0(t), v^0(t), D^0(t)), \\
w_{\pi^0}(t) &\textstyle = e^{i(\Gamma^0(t) + \int_0^t ((\alpha^0)^2(s) -|v^0|^2(s)) \dd s + v^0(t) \cdot x)}
\textstyle\phi(x - 2\int_0^t v^0(s) \dd s - D^0(t), \alpha^0(t)), \\
W_{\pi^0}(t) &= \bpm w_{\pi^0}(t) \\ \ov {w_{\pi^0}}(t) \epm,
\end{aligned}\ee
as well as
$$\begin{aligned}
\partial_{\Gamma} w_{\pi^0} &= iw_{\pi^0},\ \partial_{\alpha} w_{\pi^0} = \partial_{\alpha} w_{\pi^0},\ \partial_{D_k} w_{\pi^0} = \partial_{x_k} w_{\pi^0},\ \partial_{v_k} w_{\pi^0} = i x_k w_{\pi^0}, \\
\partial_{f} W_{\pi^0} &= \bpm \partial_f w_{\pi^0} \\ \partial_f \ov w_{\pi^0} \epm, \\
\partial^*_{\alpha}W_{\pi^0} &= i\sigma_3 \partial_{\Gamma} W_{\pi^0},\ \partial^*_{\Gamma}W_{\pi^0} = i\sigma_3 \partial_{\alpha} W_{\pi^0},\ 
\partial^*_{v_k}W_{\pi^0} = i\sigma_3 \partial_{D_k} W_{\pi^0},\ \partial^*_{D_k}W_{\pi^0} = i\sigma_3 \partial_{v_k} W_{\pi^0}.
\end{aligned}$$

The connection between (\ref{NLS}), (\ref{rr}), and the linearized equation (\ref{ec_liniara}) is the following:
\begin{lemma}
$\psi$ is a solution of \eqref{NLS} if and only if
\be
\Psi = \bpm \psi \\ \ov{\psi} \epm = W_{\pi^0} + R,
\ee
$W_{\pi^0}(t)$ is a moving soliton parametrized by \eqref{w0}, and $(R, \pi)$ is a fixed point of the map $(R^0, \pi^0) \mapsto (R, \pi)$, where $R$ and $\pi$ solve the following system of linear equations:
\be\begin{aligned}
&i \partial_t R + H_{\pi^0}(t) R = F,\ F= -i L_{\pi^0} R + N(R^0, W_{\pi^0}) - N_{\pi^0}(R^0, W_{\pi^0})\\
&\dot f = 2(\alpha^0)^2 \|\phi\|_{\dot H^{1/2}}^{-2} \big(\langle R, (d_{\pi} \partial^*_fW_{\pi^0})\, \dot \pi^0 \rangle - i\langle N(R^0, W_{\pi^0}), \partial^*_fW_{\pi^0}\rangle\big), f \in \{\alpha, \Gamma\}\\
&\dot f = \alpha^0 \|\phi\|_{\dot H^{1/2}}^{-2} \big(\langle R, (d_{\pi} \partial^*_fW_{\pi^0})\, \dot \pi^0 \rangle - i\langle N(R^0, W_{\pi^0}), \partial^*_fW_{\pi^0}\rangle\big), f \in \{v_k, D_k\}.
\end{aligned}\lb{ec_liniara}\ee
\lb{lemma1.1}
\end{lemma}
Here
\be\lb{221}\begin{aligned}
H_{\pi^0}(t) &= \bpm \Delta + 2|w_{\pi^0}(t)|^2 & (w_{\pi^0}(t))^2 \\ -(\ov w_{\pi^0}(t))^2 & -\Delta - 2|w_{\pi^0}(t)|^2 \epm, \\
N(R^0, W_{\pi^0}) &= \bpm -|r^0|^2 r^0 - (r^0)^2 \ov w_{\pi^0} - 2|r^0|^2 w_{\pi^0} \\ |r^0|^2 \ov {r^0} + (\ov {r^0})^2 w_{\pi^0} + 2|r^0|^2 \ov w_{\pi^0} \epm, \\
L_{\pi^0} R &= 2(\alpha^0)^2 \sum_{f \in \alpha, \Gamma} \|\phi\|_{\dot H^{1/2}}^{-2} \langle R, (d_{\pi} \partial^*_fW_{\pi^0}) \dot \pi^0 \rangle \partial_f W_{\pi^0} \\
&+ \alpha_0 \sum_{f \in \{v_k, D_k\}} \|\phi\|_{\dot H^{1/2}}^{-2} \langle R, (d_{\pi} \partial^*_fW_{\pi^0}) \dot \pi^0 \rangle \partial_f W_{\pi^0},
\end{aligned}\ee
\be\nonumber\begin{aligned}
N_{\pi^0}(R^0, W_{\pi^0}) &= 2(\alpha^0)^2 \sum_{f \in \{\alpha, \Gamma\}} \|\phi\|_{\dot H^{1/2}}^{-2} i\big\langle N(R^0, W_{\pi^0}), \partial^*_fW_{\pi^0} \big\rangle \partial_{f}W_{\pi^0} \\
&+ \alpha^0 \sum_{f \in \{v_k, D_k\}} \|\phi\|_{\dot H^{1/2}}^{-2} i\big\langle N(R^0, W_{\pi^0}), \partial^*_fW_{\pi^0}\big\rangle \partial_{f} W_{\pi^0}.
\end{aligned}
\ee
Our main objective is solving the system (\ref{ec_liniara}) in a suitable space, with a view toward applying a fixed point theorem to obtain $(R^0, \pi^0) = (R, \pi)$.

\begin{proof}
We substitute $R^0$ for $R$ and $\pi^0$ for $\pi$ into all higher order terms of (\ref{rr}) and (\ref{mod}). By this substitution, the nonlinear Schr\"{o}dinger equation (\ref{rr}) becomes the following equation, linear in both $R$ and $\dot \pi$:
\be
i \partial_t R - H_{\pi^0}(t) R = F.
\lb{nlsu}
\ee
Here $F= - i d_{\pi} W_{\pi^0}(t)\, \dot \pi(t) + N(R^0, W_{\pi^0})$ and
$$\begin{aligned}
d W_{\pi^0}(t)\, \dot \pi(t) &= \partial_{\Gamma}W_{\pi^0}(t)\, \dot \Gamma(t) + \partial_{\alpha}W_{\pi^0}(t)\, \dot \alpha(t) + \partial_{D_k}W_{\pi^0}(t)\, \dot D_k(t) + \partial_{v_k}W_{\pi^0}(t)\, \dot v_k(t).
\end{aligned}$$
The orthogonality condition we put on the linearized equation (\ref{nlsu}) is
\be\lb{ortolin}
\langle R(t), \partial^*_fW_{\pi^0}(t) \rangle = 0,\ f \in \{\alpha, \Gamma, v_k, D_k\}.
\ee
Taking the derivative in (\ref{ortolin}) with respect to $t$, we obtain the linearized version of the modulation equations of Lemma \ref{lemma_6}: (\ref{mod}) become
\be\begin{aligned}
\dot \alpha &= 2(\alpha^{0})^2 \|\phi\|_{\dot H^{1/2}}^{-2} \big(\langle R, (d \partial^*_{\alpha}W_{\pi^0}) \dot \pi^0 \rangle - i\langle N(R^{0}, \pi^{0}), \partial^*_{\alpha} W_{\pi^0}\rangle\big)\\
\dot \Gamma &= 2(\alpha^{0})^2 \|\phi\|_{\dot H^{1/2}}^{-2} \big(\langle R, (d \partial^*_{\Gamma}W_{\pi^0}) \dot \pi^0 \rangle - i\langle N(R^{0}, \pi^{0}), \partial^*_{\Gamma} W_{\pi^0}\rangle\big)\\
\dot v_k &= \alpha^0 \|\phi\|_{\dot H^{1/2}}^{-2} \big(\langle R, (d \partial^*_{v_k}W_{\pi^0}) \dot \pi ^0\rangle - i\langle N(R^{0}, \pi^{0}), \partial^*_{v_k}W_{\pi^0}\rangle\big)\\
\dot D_k &= \alpha^0 \|\phi\|_{\dot H^{1/2}}^{-2} \big(\langle R, (d \partial^*_{D_k}W_{\pi^0}) \dot \pi^0 \rangle - i\langle N(R^{0}, \pi^{0}), \partial^*_{D_k} W_{\pi^0}\rangle\big).
\lb{mod0}\end{aligned}\ee
The identity (\ref{modula}) translates into
\be
d W_{\pi^0}\, \dot \pi = L_{\pi^0} R - i N_{\pi^0}(R^0, \pi^0).
\lb{modula0}
\ee

Finally, we collect together (\ref{nlsu}) and (\ref{mod0}) and replace $d_{\pi} W_{\pi^0}\, \dot\pi$ on the right-hand side of (\ref{nlsu}) by its expression (\ref{modula}) in order to arrive at (\ref{ec_liniara}).
\end{proof}

\subsection{The spectrum of $H$}\lb{spectru} Let $\phi$ be the positive ground state soliton given by (\ref{phi}). Applying the symmetries of the equation to $\phi$, parameters $p=(\alpha, \Gamma, v_k, D_k) \in \set R^8$ define the general ground state soliton
\be\begin{aligned}
w&=w(p)=e^{i(\Gamma + v \cdot x)} \phi(x-D, \alpha), \\
W&=W(p)= \bpm w \\ \ov{w} \epm.
\end{aligned}\ee
The Hamiltonian obtained by linearizing (\ref{NLS}) around $w$ is
\be\lb{ham}
H(W) = H(p) := \bpm \Delta + 2|w|^2 & w^2 \\ -\ov w^2 & -\Delta -2|w|^2 \epm + 2iv \dl - (\alpha^2-|v|^2) \sigma_3.
\ee

All the operators $H(W)$ are in fact conjugate to one another, by the symmetries of the equation --- rescaling, boost, translation, and gauge ---, up to a constant factor of $\alpha^2$:
\be
\Dil_{1/\alpha} e^{-D\dl} T_D e^{-i(xv+\Gamma) \sigma_3} H(\alpha, \Gamma, v, D) e^{i(xv +\Gamma)\sigma_3} e^{D\dl} \Dil_{\alpha} = \alpha^2 H(1, 0, 0, 0).
\ee
Therefore, all $H(W)$ have the same spectrum up to dilation. Thus, it suffices to study $H = H(1, 0, 0, 0)$:
\be
H := \bpm \Delta - 1 + 2\phi^2 & \phi^2 \\ -\phi^2 & -\Delta + 1 - 2\phi^2 \epm.
\ee

We list the known spectral properties of $H$. As proved by Buslaev--Perelman in \cite{buslaev1} and Rodnianski--Schlag--Soffer in \cite{rod3}, under fairly general assumptions, $\sigma(H) \subset \R \cup i\R$ and is symmetric with respect to the coordinate axes. All eigenvalues are simple with the possible exception of $0$.

Furthermore, by Weyl's criterion $\sigma_{ess}(H) = (-\infty, -1] \cup [1, \infty)$.

Grillakis--Shatah--Strauss \cite{gril1} and Schlag \cite{schlag} showed that $H$ has only one pair of conjugate imaginary eigenvalues, call them $\pm i\sigma$, and that the corresponding eigenvectors decay exponentially. Hundertmark--Lee \cite{hund} removed the exponential decay assumption on the off-diagonal components of the potential.

The pair of conjugate imaginary eigenvalues $\pm i\sigma$ reflects the instability of the ground state soliton (\ref{phi}), coming from the fact that (\ref{NLS}) is $L^2$-supercritical.

The generalized eigenspace at $0$ arises due to the symmetries of the equation, which is invariant under Galilean coordinate changes, gauge transformations, and rescaling. Each symmetry gives rise to a generalized eigenstate of the Hamiltonian $H$ at $0$. Proving that there are no other zero eigenstates is harder and was done by Weinstein in \cite{wein1}, \cite{wein2}.

Using ideas of Perelman \cite{perelman}, Schlag \cite{schlag} showed that if the scalar operators
$$
L_{\pm} = -\Delta + 1 - 2\phi^2(\cdot, 1) \mp \phi^2(\cdot, 1)
$$
that arise by conjugating $H$ with $\bpm 1 & i \\ 1 & -i \epm$ have no eigenvalue in $(0, 1]$ and no resonance at $1$ --- the so-called gap property ---, then the real discrete spectrum of $H$ is $\{0\}$ and the edges $\pm 1$ of the essential spectrum are neither eigenvalues nor resonances.

Demanet--Schlag \cite{demanet} verified numerically that $L_{\pm}$ meet this condition. More recently, Costin--Huang--Schlag \cite{cohusc} have given a completely analytic proof to the gap property. Therefore, $H$ has no eigenvalues in $[-1, 1]$ and $\pm 1$ are neither eigenvalues nor resonances.

Adapting Agmon's method \cite{agmon} to the matrix case, \cite{erdsch2} and \cite{cuc2} proved that any resonances embedded in the interior of the essential spectrum (that is, in $(-\infty, -1) \cup (1, \infty)$) have to be eigenvalues, under very general assumptions.

Marzuola--Simpson \cite{marsim} have shown by an analytical, computer-assisted proof that no eigenvalue lies in the essential spectrum of $H$, thus completing, in a sense, the description of the spectrum.

To sum up, the spectrum of $H$ consists of a pair of conjugate purely imaginary eigenvalues $\pm i\sigma$, a generalized eigenspace at $0$, and the essential spectrum $(-\infty, -1] \cup [1, \infty)$.

Following \cite{schlag}, let $F^{\pm}$ be the eigenfunctions of $H$ corresponding to $\pm i \sigma$; then there exists $f^+ \in L^2$ such that
$$
F^+ = \bpm f^+ \\ \ov f^+ \epm,\ F^- = \ov F^+ = \bpm \ov f^+ \\ f^+ \epm.
$$
Due to the symmetry $\sigma_3 H \sigma_3 = H^*$ of the Hamiltonian $H$, the respective eigenfunctions of $H^*$ are $\sigma_3 F^{\pm}$. Then, the imaginary spectrum projection is given by
$$
P_{im} = P_+ + P_-,\ P_{\pm}= \langle \cdot, i\sigma_3 F^{\mp} \rangle F^{\pm}
$$
up to a constant; the constant becomes $1$ after the normalization
$$
\int_{\R^3} \Re f^+(x) \Im f^+(x) \dd x = -1/2.
$$

It helps to represent the discrete eigenspaces of $H(W)$ as explicitly as possible. Let $F^{\pm}(W)$ the eigenfunctions of $H(W)$ corresponding to the $\pm i \sigma(W)$ eigenvalues.

Schlag \cite{schlag} proved in a more general setting that $F^{\pm}(W)$, $L^2$ normalized, and $\sigma$ are locally Lipschitz continuous as a function of $W$ and that $F^{\pm}(W)$ are exponentially decaying.

There is no explicit formula for $F^{\pm}(W)$, but the dependence of $F^{\pm}(W)$ and $\sigma$ on $\pi$ is explicit, since the operators $H(W)$ are conjugate to one another:
\be\begin{aligned}\lb{2.50}
F^{\pm}(W) &= e^{i(xv+\Gamma)\sigma_3} e^{D\dl} \Dil_{\alpha} F^{\pm},\\
\sigma(W) &= \alpha^2 \sigma.
\end{aligned}\ee

Also observe that $\partial_f W$, where $f \in \{\alpha, \Gamma, v_k, D_k\}$, are the generalized eigenfunctions of $H(W)$ at zero and $\partial^*_fW$, defined as in (\ref{2.19}), fulfill the same role for~$H(W)^*$.

We then express the Riesz projections on the three components of the spectrum (continuous, null, and imaginary) of $H(W)$ as
\begin{align*}
P_{im}(W) &= P_+(W) + P_-(W),\ P_{\pm}(W) = \alpha^{-3} \langle \cdot, i\sigma_3 F^{\mp}(W) \rangle F^{\pm}(W), \\
P_{0}(W) &= 4\alpha \|W\|_2^{-2} (\langle \cdot, \partial^*_{\alpha} \rangle \partial_{\alpha}W + \langle \cdot, \partial^*_{\Gamma} \rangle \partial_{\Gamma}W) + \\
&\nonumber+ 2 \|W\|_2^{-2} \sum_k \big(\langle \cdot, \partial^*_{v_k} \rangle \partial_{v_k}W + \langle \cdot, \partial^*_{D_k} \rangle \partial_{D_k}W\big),\\
P_c(W) &= 1-P_{im}(W)-P_{0}(W).
\end{align*}


\subsection{The fixed point argument. Stability} Consider a small neighborhood of a soliton $w_0=w(p_0)$, determined by $p_0 = (\alpha_0, \Gamma_0, v_{k0}, D_{k0}) \in \R^8$, where
$$\begin{aligned}
w_0 &= e^{i(\Gamma_0 + v_0 \cdot x)} \phi(x-D_0, \alpha_0),\ W_0 &= \bpm w_0 \\ \ov w_0 \epm.
\end{aligned}$$
By Section \ref{spectru}, $W_0$ determines a Hamiltonian $H(W_0)$ of the form (\ref{ham}):
$$
H(W_0) = \bpm \Delta + 2 |w_0|^2 & w_0^2 \\ - \ov w_0^2 & -\Delta - 2|w_0|^2 \epm + 2iv_0\dl - (\alpha_0^2-|v_0|^2) \sigma_3.
$$
$H(W_0)$ has an associated null space projection $P_0(W_0)$, an imaginary spectrum projection $P_{im}(W_0) = P_+(W_0) + P_-(W_0)$, and an essential spectrum projection $P_c(W_0) = I - P_0(W_0) - P_{im}(W_0)$.

Up to quadratic corrections, the centre-stable submanifold at $W_0$ is the following affine subspace $\mc N_{lin}(W_0)$ of~$\dot H^{1/2}$:
$$\begin{aligned}
\mc N_{lin}(W_0) &:= W_0 + (P_c(W_0) + P_-(W_0)) \dot H^{1/2} \\
&= \big\{W_0 + R_0 \mid R_0 \in \dot H^{1/2},\ \big(P_0(W_0) + P_+(W_0)\big) R_0 = 0\big\}.
\end{aligned}$$
$\mc N_{lin}(W_0)$ has codimension nine within $\dot H^{1/2}$, so we need a supplementary argument, presented separately, to recover eight codimensions from the symmetries of (\ref{NLS}).

Take initial data of the form
\be\begin{aligned}
R(0) &= R_0 + h F^+(W_0), \pi(0)=\pi_0=(\alpha_0, \Gamma_0, (v_k)_0, (D_k)_0),
\lb{Z_initial}
\end{aligned}\ee
where $R_0 \in  (P_c + P_-) \dot H^{1/2}$. Thus, $W_0 + R_0 \in \mc N_{lin}(W_0)$ and the quadratic correction is $h F^+(W_0)$, made in the direction of $F^+(W_0)$. In particular,
\be
P_0(W_0) R(0) = 0,\ P_+(W_0) R(0) = h F^+(W_0).
\ee
(\ref{Z_initial}) also implies that $\|R(0)\|_{\dot H^{1/2}} \les \|R_0\|_{\dot H^{1/2}} + |h|$.

Consider the space
\be\begin{aligned}
X &= \{(R, \pi) \mid R \in L^{\infty}_t \dot H^{1/2}_x \cap L^2_t \dot W^{1/2, 6}_x,\ \dot \pi \in L^1\},
\lb{X}
\end{aligned}\ee
with the seminorm
\be
\|(R, \pi)\|_X = \|R\|_{L^{\infty}_t \dot H^{1/2}_x \cap L^2_t \dot W^{1/2, 6}_x} + \alpha_0^{-1} \|\dot \pi\|_{L^1_t},
\ee
where $\alpha_0$ is the scaling parameter of (\ref{Z_initial}).

We define a map $\Epsilon$ as follows:
\begin{definition}
$\Epsilon(R^0, \pi^0)$ is the map that associates, to a pair $(R^0, \pi^0)$ of auxiliary functions, the unique bounded solution $(R, \pi) \in X$ of the linearized equation system (\ref{ec_liniara}), 
with initial data (\ref{Z_initial}), where $R_0$ is given and $h$ is allowed to take any value:
\be
\Epsilon(R^0, \pi^0) := (R, \pi).
\lb{2.32}
\ee
\end{definition}
We show that, everything else being fixed, the solution $(R, \pi)$ of (\ref{ec_liniara}) is in $X$ for a unique value of $h := h(R_0, R^0, \pi^0)$. Thus the map $\Epsilon$ is well-defined.


$X$ is natural in the study of (\ref{NLS}) and of its linearized version (\ref{ec_liniara}). Indeed, since the Schr\"{o}dinger equation (\ref{NLS}) is $\dot H^{1/2}$-critical, we study it in the critical Strichartz space $L^{\infty}_t \dot H^{1/2}_x \cap L^2_t \dot W^{1/2, 6}_x$. Likewise, $\dot \pi \in L^1$ is a minimal assumption and there is no room for a stronger condition, when working in a critical space.

Fix a small $\delta_0 = \delta_0(W_0) >0$. We prove the following stability property: for $\delta \in (0, \delta_0)$, if $\|(R^0, \pi^0)\|_X < \delta$, then $\|(R, \pi)\|_X = \|\Epsilon(R^0, \pi^0)\|_X < \delta$.
\begin{proposition}[Stability]\lb{prop9}
Let $(R, \pi)$ be a solution of equation \eqref{ec_liniara} linearized around $(R^0, \pi^0)$ with initial data \eqref{Z_initial}. Assume that $(R^0, \pi^0)$ satisfy \eqref{Z_initial} and, for some $\delta \leq \delta_0$, $\|(R^0, \pi^0)\|_X \leq \delta$. Then $(R, \pi) \in X$ for a unique value of $h:=h(R_0, R^0, \pi^0)$ and
\be\lb{stabineq}\begin{aligned}
\|R\|_{L^{\infty}_t \dot H^{1/2}_x \cap L^2_t \dot W^{1/2, 6}_x} &\les \|R_0\|_{\dot H^{1/2}} + \delta^2, \\
\|\dot \pi\|_{L^1_t} + |h(R_0, R^0, \pi^0)| &\les \alpha_0 (\delta \|R_0\|_{\dot H^{1/2}} + \delta^2).
\end{aligned}\ee
\end{proposition}

If $\|R_0\|_{\dot H^{1/2}} \leq c \delta$ and $\delta$ is small, it follows that the set
\be\lb{stabset}
\{(R, \pi) \mid R(0) = R_0 + h F^+(W_0=W(\pi_0)),\ \pi(0) = \pi_0,\ \|(R, \pi)\|_X \leq \delta\}
\ee
is stable under an application of the map $\Epsilon$ and also that $|h(R_0, R^0, \pi^0)| \les \alpha_0 \delta^2$.


Proposition \ref{prop9} holds with a constant that depends on $W_0$, for sufficiently small $\delta < \delta_0(W_0)$. We are also interested in the dependence of the result on $W_0$.

Clearly, the result is invariant under translations and gauge transformations. The Strichartz norms are $\dot H^{1/2}$ scaling-invariant, while $\|\dot \pi\|_{L^1_t}$ and $h$ scale like $\alpha_0$. After accounting for this, the statement of Proposition \ref{prop9} also becomes scale-independent. The constants will still depend on the momentum of $W_0$, $P[W_0]$.

\begin{proof}
The modulation path $\pi^0:\set [0, \infty) \to \R^8$
defines the moving soliton $w_{\pi^0}(t)$ by (\ref{w0}). Linearizing around $w_{\pi^0}(t)$ yields the Hamiltonian $H_{\pi^0}(t)$, as per Lemma \ref{lemma1.1}.

Denote
$$\begin{aligned}
w(\pi^0(t)) &= e^{i(\Gamma^0(t) + v^0(t) \cdot x)} \phi\big(x-D^0(t), \alpha^0(t)\big),\ W(\pi^0(t)) = \bpm w(\pi^0(t)) \\ \ov w(\pi^0(t)) \epm;\ w(\pi^0(t)) \ne w_{\pi^0}(t).
\end{aligned}$$
In view of Theorem \ref{theorem_1.4}, consider the family of isometries
\be\lb{2.63}
U(t) = e^{\textstyle\int_0^t(2 v^0(s) \dl + i ((\alpha^0(s))^2-|v^0(s)|^2) \sigma_3) \dd s}.
\ee
With $H(W)$ given by (\ref{ham}), note that
$$\begin{aligned}
W(\pi^0(t)) = U(t) W_{\pi^0}(t),\ H(W(\pi^0(t))) = U(t)^{-1} H(W_{\pi^0}(t)) U(t),
\end{aligned}$$
and let $Z(t) := U(t) R(t)$.

Also let
\be\begin{aligned}\lb{2.67}
H(t) &:= H(W(\pi^0(t))), &P_0(t) &:= P_0(W(\pi^0(t))),\\
P_c(t) &:= P_c(W(\pi^0(t))), &P_{im}(t) &:= P_{im}(W(\pi^0(t))),\\
P_+(t) &:= P_+(W(\pi^0(t))), &P_-(t) &:= P_-(W(\pi^0(t))),\\
F^+(t) &:= F^+(W(\pi^0(t))), &F^-(t) &:= F^-(W(\pi^0(t))),\\
\partial^*_fW(t) &:= \partial^*_fW(\pi^0(t)), &\sigma(t) &:= \sigma(W(\pi^0(t))).
\end{aligned}\ee
Rewriting (\ref{ec_liniara}) 
with $Z$ as the unknown, we obtain
\be\lb{2.69}
i \partial_t Z - H(t) Z = U(t) F(t).
\ee
Instead of fixing a time-independent Hamiltonian for (\ref{2.69}), we take $H(t)$ itself as the time-dependent Hamiltonian.

The equation (\ref{2.69}) for $Z$ has three parts, which correspond to the three components of the spectrum of $H(t)$ --- continuous, null, and imaginary. Let
\be\begin{aligned}
I = P_c(t) + P_0(t) + P_{im}(t),\ P_{im}(t) = P_+(t) + P_-(t).
\end{aligned}\ee
Each component has a different behavior and has to be studied separately. 
By using a different method for each component, we prove the estimates that enable the contraction scheme.

The kernel component is bounded by the method of modulation. By (\ref{Z_initial}), the orthogonality condition
$$
\langle R(t), \partial^*_fW_{\pi^0}(t)\rangle = 0
$$
holds at time $t=0$ and due to the modulation equations (\ref{ec_liniara}) orthogonality still holds for all $t \in \set R$. Applying the isometry $U(t)$, the orthogonality condition becomes
$$
\langle Z(t), \partial^*_fW(t)\rangle = 0,
$$
which is equivalent to
\be\lb{2.68}
P_0(t) Z(t) = 0.
\ee

The continuous spectrum component of $Z$ fulfills the equation, derived from (\ref{2.69}),
\begin{align}
i\partial_t\big(P_c(t) Z\big) + H(t) P_c(t) Z &= P_c(t) U(t) F(t) + i \big(\partial_t P_c(t)\big) Z.
\lb{disp}
\end{align}
The inhomogenous term
$$
F=-i L_{\pi^0} R + N(R^0, \pi^0) - N_{\pi^0}(R^0, \pi^0)
$$
given by (\ref{221}) is bounded in the dual Strichartz norm by the fractional Leibniz rule.

Estimate the quadratic and cubic terms of $N(R^0, \pi^0)$ as follows:
$$\begin{aligned}
\|(r^0)^2 \ov w(\pi^0)\|_{L^2_t \dot W^{1/2, 6/5}_x} &\les \|(r^0)^2\|_{L^2_t \dot H^{1/2}_x} \|\ov w(\pi^0)\|_{L^{\infty}_t \dot H^{1/2}_x} \\
&\les \|R^0\|_{L^4_t \dot W^{1/2, 3}_x}^2 \|W(\pi^0)\|_{L^{\infty}_t \dot H^{1/2}_x}.
\end{aligned}$$
The same applies to the cubic term, when we replace $w(\pi^0)$ by $r^0$:
$$
\|(r^0)^2 \ov r^0\|_{L^2_t \dot W^{1/2, 6/5}_x} \les \|R^0\|_{L^4_t W^{1/2, 3}_x}^2 \|R^0\|_{L^{\infty}_t \dot H^{1/2}_x}.
$$
Since $\|R^0\|_{L^4_t \dot W^{1/2, 3}_x}^2 \leq \|R^0\|_{L^{\infty}_t \dot H^{1/2}_x} \|R^0\|_{L^2_t \dot W^{1/2, 6}_x}$,
$$
\|N(R^0, \pi^0) - N_{\pi^0}(R^0, \pi^0)\|_{L^2_t \dot W^{1/2, 6/5}_x} \les \|R^0\|_{L^2_t \dot W^{1/2, 6}_x} \big(\|R^0\|_{L^{\infty}_t \dot H^{1/2}_x} + \|R^0\|_{L^{\infty}_t \dot H^{1/2}_x}^2\big).
$$
Thus
$$
\|N(R^0, \pi^0) - N_{\pi^0}(R^0, \pi^0)\|_{L^2_t \dot W^{1/2, 6/5}_x} \les  \delta^2 + \delta^3.
$$
We bound $L_{\pi^0} R$ by Strichartz estimates:
$$\begin{aligned}
\|L_{\pi^0} R\|_{L^1_t \dot H^{1/2}_x} &\les \|\dot \pi\|_{L^1_t} \|R\|_{L^{\infty}_t \dot H^{1/2}_x} \les \delta \|R\|_{L^{\infty}_t \dot H^{1/2}_x} = C \delta \|Z\|_{L^{\infty}_t \dot H^{1/2}_x}.
\end{aligned}$$
In conclusion,
\be\lb{2.49}
\|F\|_{L^1_t \dot H^{1/2}_x + L^2_t \dot W^{1/2, 6/5}_x} \les \delta \|Z\|_{L^{\infty}_t \dot H^{1/2}_x} + \delta^2 + \delta^3
\ee
Furthermore,
\be\begin{aligned}
(\partial_t P_+(t)) Z &= (\partial_t \alpha^{-3}) \langle Z, i \sigma_3 F^-(t) \rangle F^+(t) + \alpha^{-3} \langle Z, i \sigma_3 \partial_t F^-(t) \rangle F^+(t) + \\
&+ \alpha^{-3} \langle Z, i \sigma_3 F^-(t) \rangle \partial_t F^+(t) \\
&= -3 \alpha^{-4} \dot \alpha \langle Z, i \sigma_3 F^-(t) \rangle F^+(t) + \alpha^{-3} \langle Z, i \sigma_3 d_{\pi} F^-(t) \dot \pi^0(t) \rangle F^+(t) + \\
&+ \alpha^{-3} \langle Z, i \sigma_3 F^-(t)\rangle d_{\pi} F^+(t) \dot \pi^0(t)
\end{aligned}\ee
and likewise for $P_-$ and $P_0$. Since $P_c = I - P_0 - P_{im}$ and $P_0(t) Z(t) = 0$,
\be\begin{aligned}\lb{2.72}
\|(\partial_t P_{\pm}(t)) Z\|_{L^2_t \dot W^{1/2, 6/5}_x} &\les \delta \|Z\|_{L^2_t \dot W^{1/2, 6}_x}.
\end{aligned}\ee

Provided $\|\dot \pi^0\|_1 < \delta$ is sufficiently small, Theorem \ref{theorem_1.4} leads to endpoint Strichartz estimates for $P_c(t) Z(y)$, through the following construction. Recall that
$$
H(t) = \bpm \Delta + 2|w(\pi^0(t))|^2 & w^2(\pi^0(t)) \\ -\ov w^2(\pi^0(t)) & -\Delta - 2|w(\pi^0(t))|^2\epm + 2iv^0(t) \dl - ((\alpha^0(t))^2 - |v^0(t)|^2) \sigma_3.
$$
Denote, for $\pi^0(0) = \pi_0$ following (\ref{Z_initial}),
$$
\tilde {H}(t) = \bpm \Delta + 2|w(\pi_0)|^2 & w^2(\pi_0) \\ -\ov w^2(\pi_0) & -\Delta - 2|w(\pi_0)|^2\epm  + 2iv^0(t) \dl - ((\alpha^0(t))^2 - |v^0(t)|^2) \sigma_3.
$$
We bound $(\tilde {H} - H) Z$ by endpoint Strichartz estimates. Because of the lack of an endpoint Sobolev embedding, we do this explicit computation: $\dot W^{1/2, 6} \subset \dot W^{1/2, 6-\epsilon} + \dot W^{1/2, 6+\epsilon}$ implies
\be\lb{endpoint}
\dot W^{1/2, 6} \subset \dot W^{1/2, 6} \cap \dot W^{1/2, 6-\epsilon} + \dot W^{1/2, 6}\cap \dot W^{1/2, 6+\epsilon}.
\ee
Then, by the fractional Leibniz rule,
$$
\|fg\|_{\dot W^{1/2, 6/5}} \les \|f\|_{\dot W^{1/2, 6} \cap \dot W^{1/2, 6-\epsilon}} \|g\|_{\dot W^{1/2, 6/5} \cap L^{3/2+\epsilon}}
$$
and, because $\dot W^{1/2, 6}\cap \dot W^{1/2, 6+\epsilon} \subset L^{\infty}$,
$$
\|fg\|_{\dot W^{1/2, 6/5}} \les \|f\|_{\dot W^{1/2, 6} \cap \dot W^{1/2, 6+\epsilon}} \|g\|_{\dot W^{1/2, 6/5}}.
$$
By decomposing $f \in \dot W^{1/2, 6}$ into two parts, one in $\dot W^{1/2, 6} \cap \dot W^{1/2, 6-\epsilon}$ and the other in $\dot W^{1/2, 6}\cap \dot W^{1/2, 6+\epsilon}$, in conclusion we have that
$$
\|fg\|_{\dot W^{1/2, 6/5}} \les \|f\|_{\dot W^{1/2, 6}} \|g\|_{\dot W^{1/2, 6/5} \cap L^{3/2+\epsilon}}.
$$
The difference $\tilde {H}(t) - H(t)$ is a multiplication operator and is uniformly small for all $t$ in any Schwartz seminorm; in particular,
$$
\|\tilde {H}(t) - H(t)\|_{L^{\infty}_t (\dot W^{1/2, 6/5}_x \cap L^{3/2+\epsilon}_x)} \les \|\pi^0(t)-\pi_0\|_{L^{\infty}_t} \leq \|\dot \pi^0(t)\|_{L^1_t} \leq \delta.
$$
Thus 
\be\begin{aligned}\lb{2.73}
\|(\tilde {H} - H) Z\|_{L^2_t \dot W^{1/2, 6/5}_x} &\leq \|\tilde {H}(t) - H(t)\|_{L^{\infty}_t (\dot W^{1/2, 6/5} \cap L^{3/2-\epsilon}_x)} \|Z\|_{L^2_t \dot W^{1/2, 6}_x} \\
&\les \delta \|Z\|_{L^2_t \dot W^{1/2, 6}_x}.
\end{aligned}\ee
By (\ref{2.49}), (\ref{2.72}), (\ref{2.73}), and Theorem \ref{theorem_1.4} applied to $\tilde H$, for small $\delta$
\be\begin{aligned}\lb{2.76}
\|P_c(t) Z\|_{L^{\infty}_t \dot H^{1/2}_x \cap L^2_t \dot W^{1/2, 6}_x} 
&\les \|Z(0)\|_{\dot H^{1/2}} + \delta^2 + \delta \|Z\|_{L^2_t \dot W^{1/2, 6}_x}.
\end{aligned}\ee

For the imaginary spectrum component $P_{im} Z$ write, following Section \ref{spectru},
\be
P_{im}(t) Z(t) = h^+(t) F^+(t) + h^-(t) F^-(t).
\lb{partea_finita}
\ee
Then, by taking the time derivative of
\be
h^{\pm}(t) := (\alpha^0(t))^{-3} \langle Z(t), i \sigma_3 F^{\mp}(t) \rangle,
\ee
we obtain
\be\begin{aligned}
\partial_t h^{\pm} &= (\alpha^0)^{-3} \langle \partial_t Z, i \sigma_3 F^{\mp} \rangle -3 \dot \alpha^0 (\alpha^0)^{-4} \langle Z, i \sigma_3 F^{\mp}(t) \rangle + \\
&+ (\alpha^0)^{-3} \langle Z, i \sigma_3 (d_{\pi} F^{\mp}) \dot \pi^0 \rangle.
\end{aligned}\ee
Replacing $\partial_t Z$ by its expression (\ref{ec_liniara}), we arrive at
\be\begin{aligned}\lb{hiperb}
\partial_t h^{\pm} &= \pm \sigma(t) h^{\pm} - \langle U F, \sigma_3 F^{\mp} \rangle - \\
& -3 \dot \alpha^0 (\alpha^0)^{-4} \langle Z, i \sigma_3 F^{\mp} \rangle +  (\alpha^0)^{-3} \langle Z, i \sigma_3 (d_{\pi} F^{\mp}) \dot \pi^0 \rangle.
\end{aligned}\ee
Thus $h^+(t)$ and $h^-(t)$ satisfy the equation
\be
\partial_t  \bpm h_-\\ h_+\epm + \bpm \sigma(t) & 0 \\ 0 & -\sigma(t)\epm \bpm h_-\\ h_+\epm = \bpm N_-(t) \\ N_+(t) \epm,
\lb{hiper}
\ee
where $\pm i\sigma(t)$ are the imaginary eigenvalues of $H(t)$, as in Section \ref{spectru}, and
$$
N_{\pm} = -\langle U F, \sigma_3 F^{\mp} \rangle -3 \dot \alpha^0 (\alpha^0)^{-4} \langle Z, i \sigma_3 F^{\mp} \rangle +  (\alpha^0)^{-3} \langle Z, i \sigma_3 (d_{\pi} F^{\mp}) \dot \pi^0 \rangle.
$$
(\ref{hiperb}) and (\ref{2.49}) imply
\be
\|N_{\pm}(t)\|_{L^2_t} \les \delta \|Z(t)\|_{L^2_t \dot W^{1/2, 6}_x} + \delta^2.
\ee

To control $P_{im}(t) Z(t)$ we use the following fact, see \cite{schlag}. It characterizes the bounded solution of the system (\ref{hiper}).

\begin{lemma}\lb{hyp}
Let $x(t) = \bpm x_1(t) \\ x_2(t) \epm$ and consider the equation
\be
\dot x - \bpm \sigma(t) & 0 \\ 0 & -\sigma(t)\epm x = f(t),
\ee
where $f(t) \in e^{\rho t} L^p_t(0, \infty)$, $f(t) = \bpm f_1(t) \\ f_2(t) \epm$,  and for all $t$ $\sigma(t) \geq \sigma_0 > \rho \geq 0$.

Then $x(t) \in e^{\rho t} L^p_t(0, \infty)$ if and only if
\be
x_1(0) =- \int_0^{\infty} e^{-\int_0^t \sigma(\tau) \dd \tau} f_1(t) \dd t.
\lb{condi}
\ee
In this case,
$$
|x_1(0)| \les \|f(t)\|_{e^{\rho t} L^p_t(0, \infty)},\ \|x(t)\|_{e^{\rho t} L^p_t(0, \infty)} \les |x_2(0)| + \|f(t)\|_{e^{\rho t} L^p_t(0, \infty)},
$$
and for all $t \geq 0$
$$\begin{aligned}\lb{2.83}
x_1(t) &= -\int_t^{\infty} e^{\int_s^t \sigma(\tau) \dd \tau} f_1(s) \dd s, \\
x_2(t) &= e^{-\int_0^t \sigma(\tau) \dd \tau} x_2(0) + \int_0^t e^{-\int_s^t \sigma(\tau) \dd \tau} f_2(s) \dd s.
\end{aligned}$$
\end{lemma}
\begin{proof}
By Duhamel's formula, any solution will be an explicit linear combination of the exponentially increasing solution and the exponentially decaying solution:
$$\begin{aligned}
x_1(t) &= e^{\int_0^t \sigma(\tau) \dd \tau} \Big(x_1(0) + \int_0^t e^{-\int_0^s \sigma(\tau) \dd \tau} f_1(s) \dd s\Big), \\
x_2(t) &= e^{-\int_0^t \sigma(\tau) \dd \tau} x_2(0) + \int_0^t e^{-\int_s^t \sigma(\tau) \dd \tau} f_2(s) \dd s.
\end{aligned}$$
If $x_1(t)$ remains bounded in $e^{\rho t} L^p_t(0, \infty)$, the expression in parantheses $x_1(0) + \int_0^t e^{-\int_0^s \sigma(\tau) \dd \tau} f_1(s)$ must converge to $0$, hence (\ref{condi}) must hold. Conversely, if (\ref{condi}) holds, then
\be
x_1(t) = -\int_t^{\infty} e^{\int_s^t \sigma(\tau) \dd \tau} f_1(s) \dd s.
\ee
Then $\lim_{t \to \infty} x_1(t) = 0$ and both $x_1$ and $x_2$ are bounded by the convolution of $f_1$, respectively $\delta_0(t) x_2(0) + f_2(t)$, with exponentially decreasing kernels:
$$\begin{aligned}
|x_1(t)| &\les \int_{-\infty}^t e^{-\sigma_0(t-s)} |f_1(s)| \dd s,\\
|x_2(t)| &\les \int_{-\infty}^t e^{-\sigma_0(t-s)} \big(\delta_0(s) |x_1(0)| + |f_2(s)|\big) \dd s.
\end{aligned}$$
Convolution with $e^{-\sigma_0 t}$ is a bounded operation on $e^{\rho t} L^p_t(0, \infty)$, $1 \leq p \leq \infty$, for $|\rho| < \sigma_0$. The conclusion follows.
\end{proof}

$\sigma \in (0, \infty)$ depends on the scaling parameter $\alpha$ by (\ref{2.50}). Then $\sigma(t)$ belongs to a compact subset $[a_1, a_2]$ of $(0, \infty)$, because $\alpha^0(t)$ belongs to a compact subset of $(0, \infty)$. By Lemma \ref{hyp}, equation (\ref{hiper}) has a bounded solution if and only if (\ref{condi}) holds:
\be
h^+(0) = -\int_0^{\infty} e^{-\int_0^t \sigma(\tau) \dd \tau} N_+(t) \dd t.
\lb{conditie}
\ee

Note that $h=h^+(0)$, where $h$ is given by (\ref{Z_initial}) and $h^+(0)$ is by (\ref{conditie}). $R$ is bounded in $L^{\infty}_t \dot H^{1/2}_x$ if and only if $Z$ is bounded. $Z$ is bounded if and only if each of its components is bounded, $P_{im} Z$ in particular. By (\ref{conditie}),  $\|Z\|_{L^{\infty}_t \dot H^{1/2}_x}$ is finite if and only if
\be\lb{hh}
h = -\int_0^{\infty} e^{-\int_0^t \sigma(\tau) \dd \tau} N_+(t) \dd t.
\ee

We next obtain a formula for $h$ that involves $R$ instead of $Z$. Recall that $i\sigma(\tau)$ is the imaginary eigenvalue of $H(W_{\pi^0}(\tau))$. Expanding $N_+$ and reverting the isometry $U$ within (\ref{hh}) leads to the explicit formula that we use later
\be\begin{aligned}\lb{2.100}
h &= -\int_0^{\infty} e^{-\int_0^t \sigma(W_{\pi^0}(\tau)) \dd \tau} \big(\langle F,  \sigma_3 F^-(W_{\pi^0}(t)) \rangle - \\
&- 3 \dot \alpha^0(t) (\alpha^0(t))^{-4} \langle R, i \sigma_3 F^-(W_{\pi^0}(t)) \rangle + \\
&+(\alpha^0(t))^{-3} \langle R, i \sigma_3 (d_{\pi} F^-(W_{\pi^0}(t))) \dot \pi^0(t) \rangle\big) \dd t.
\end{aligned}\ee

Denote this value of $h$ by $h(R_0, R^0, \pi^0)$. By Lemma \ref{hyp},
$$\begin{aligned}
|h(R_0, R^0, \pi^0)| \les \|N_+\|_{L^1_t + L^{\infty}_t} \les \delta \|Z\|_{L^2_t \dot W^{1/2, 6}_x} + \delta^2
\end{aligned}$$
and
\be\begin{aligned}\lb{2.90}
\|P_{im}(t) Z\|_{L^2_t \dot W^{1/2, 6}_x} &\leq \|b_{+}\|_{L^2_t} + \|b_{-}\|_{L^2_t} \\
&\les \|N_+\|_{L^2_t} + \|N_-\|_{L^2_t} + \|R_0\|_{\dot H^{1/2}} \\
&\les \|R_0\|_{\dot H^{1/2}} + \delta \|R\|_{L^2_t \dot W^{1/2, 6}_x} + \delta^2.
\end{aligned}\ee
For $h=h(R_0, R^0, \pi^0)$, $R$ is indeed bounded: putting (\ref{2.68}), (\ref{2.76}), and (\ref{2.90}) together and taking into account the fact that
$$
\|Z\|_{L^{\infty}_t \dot H^{1/2}_x \cap L^2_t \dot W^{1/2, 6}_x} = \|R\|_{L^{\infty}_t \dot H^{1/2}_x \cap L^2_t \dot W^{1/2, 6}_x},
$$
we obtain
$$
\|R\|_{L^2_t \dot W^{1/2, 6}_x} \les \|R_0\|_{\dot H^{1/2}} + \delta \|R\|_{L^2_t \dot W^{1/2, 6}_x} + \delta^2.
$$
Regarding the parameter path $\pi$, from the modulation equations (\ref{ec_liniara}) and (\ref{mod0}) we~get
\be\begin{aligned}\lb{modexp}
|\dot \pi(t)| &\les \langle \alpha^0 \rangle^2 \big(|\langle R(t), d_{\pi} d_{\pi}^* W_{\pi^0}(t))|\, |\dot \pi^0(t)| + |\langle N(R_0(t), W_{\pi^0}(t)), d_{\pi} W_{\pi^0}(t)\rangle| \big) \\
&\les \|R(t)\|_{\dot H^{1/2}_x} |\dot \pi^0(t)| + \|R_0(t)\|_{\dot W^{1/2, 6}_x}^2 (1 + \|R_0(t)\|_{\dot H^{1/2}_x}).
\end{aligned}\ee
Because of the lack of an endpoint Sobolev embedding, we again need to use (\ref{endpoint}).

Thus $\|\dot \pi\|_1 \les \delta \|R\|_{L^{\infty}_t \dot H^{1/2}_x} + \delta^2$. Therefore
$$
\|(R, \pi)\|_X = \|(Z, \pi)\|_X \les \|R_0\|_{\dot H^{1/2}_x} + \delta \|(R, \pi)\|_X + \delta^2.
$$
This proves (\ref{stabineq}), hence the stability of (\ref{stabset}) for small initial data $R_0 \in \dot H^{1/2}$ and for the unique suitable choice of $h = h(R_0, R^0, h^0)$.
\end{proof}

\subsection{The fixed point argument: contraction}\lb{Chapter_2.4}
Choose $\delta>0$ such that the scale $\alpha(t)$ belongs to a fixed compact subset of $(0, \infty)$, so the imaginary eigenvalue $i\sigma$ of $H(W(\pi^0(t))$ fulfills
$$
\sigma \in [a_1, a_2] \subset (0, \infty),
$$
for all $t$ and paths $\pi^0$ that we consider. Fix a constant $\rho$, $0 < \rho < a_1$.

For any two solutions of (\ref{ec_liniara}), $\Epsilon(R_j^0, \pi_j^0) = (R_j, \pi_j) \in X$, $j=1, 2$, such that
$$
\|(R_j^0, \pi_j^0)\|_X \leq \delta,
$$
we prove that $\Epsilon$ acts as a contraction in the following space $Y$:
$$\begin{aligned}
Y:= \{(R, \pi) \mid &\|e^{-t\rho} R(t)\|_{L^{\infty}_t \dot H^{1/2}_x \cap L^2_t \dot W^{1/2, 6}_x} + \|e^{-t\rho} \dot \pi(t)\|_{L^1_t} < \infty \}.
\lb{Y}
\end{aligned}$$
$Y$ is an affine space for fixed $\pi(0)=p_0$.

Furthermore, for fixed $R_0$ we prove that $h=h(R_0, R^0, \pi^0)$, for which the solution with initial data (\ref{Z_initial}) is bounded by Proposition \ref{prop9}, satisfies
$$
|h(R_0, R^0_1, \pi^0_1) - h(R_0, R^0_2, \pi^0_2)| \les \delta \|(R_1^0, \pi_1^0) - (R_2^0, \pi_2^0)\|_Y.
$$
This suffices to complete the contraction argument, see Proposition \ref{prop27}.

In proving the continuous dependence of the solution on the initial data, it is useful to also allow $R_1$ and $R_2$ to start from different initial data. This leads to the the following perturbation/contraction lemma that we use repeatedly in the sequel.

More generally, we prove a comparison result in spaces with polynomial weights. For $0 \leq a<\rho$, define the weights
\be\lb{2.230}
A_n(t) = \sum_{j=0}^n \frac{\langle at \rangle^j}{j!}.
\ee
Each weight $A_n$ is a polynomial of degree $n$; the case $n=0$ is the same as having no weight. $A_n$ have the property that
$$
\sum_{j=0}^n A_j(t) A_{n-j}(t) \leq C^n A_n(t);\ \int_0^t A_n(s) \dd s \les A_{n+1}(t);\ A_n(t) < e^{\langle at \rangle}.
$$
We then define the polynomially weighted spaces $Y_n \subset Y$ of weight $A_n(t)$:
$$\begin{aligned}
Y_n:= \{(R, \pi) \mid &\|R(t)\|_{A_n(t) L^{\infty}_t \dot H^{1/2}_x \cap A_n(t) L^2_t \dot W^{1/2, 6}_x} + \|\dot \pi(t)\|_{A_n(t) L^1_t} < \infty \}.
\lb{Y_n}
\end{aligned}$$
We prove two comparison estimates. One assumes that both solutions scatter, i.e.\ are in $X$, and evaluates the difference in $Y$ or $Y_n$. The second assumes one solution to (\ref{NLS}) scatters and the other is in $L^{\infty}_t \dot H^{1/2}_x$. In the second case, the estimate holds in $e^{t\rho} L^{\infty}_t \dot H^{1/2}_x \times e^{t\rho} L^1_t$.
\begin{lemma}[Contraction]\lb{lemma_5}
Consider two solutions, $(R_1, \pi_1)$ and $(R_2, \pi_2)$, of two distinct linearized equations of the form \eqref{ec_liniara}:
\be\begin{aligned}
&i \partial_t R_j + H_{\pi^0_j}(t) R_j = F_j,\ F_j= -i L_{\pi^0_j} R_j + N(R^0_j, W_{\pi^0_j}) - N_{\pi^0_j}(R^0_j, W_{\pi^0_j})\\
&\dot f_j = 2(\alpha^0_j)^2 \|\phi\|_{\dot H^{1/2}}^{-2} (\langle R_j, (d_{\pi} \partial^*_fW_{\pi^0_j}) \dot \pi^0_j \rangle - i\langle N(R^0_j, W_{\pi^0_j}), \partial^*_fW_{\pi^0_j}\rangle), f \in \{\alpha, \Gamma\}\\
&\dot f_j = \alpha^0_j \|\phi\|_{\dot H^{1/2}}^{-2} (\langle R_j, (d_{\pi} \partial^*_fW_{\pi^0_j}) \dot \pi^0_j \rangle - i\langle N(R^0_j, W_{\pi^0_j}), \partial^*_fW_{\pi^0_j}\rangle), f \in \{v_k, D_k\}
\end{aligned}\lb{ec_liniara12}\ee
for $j = 1, 2$, with initial data
\be
R_j(0) = R_{0j} + h_j F^+(W(p_{0j})),\ \pi_j(0) = \pi^0_j(0)=p_{0j} \text{ given}.
\ee
Assume in addition that
$$\|(R^0_1, \pi^0_1)\|_X \leq \delta,\ \|R^0_2\|_{L^{\infty}_t \dot H^{1/2}_x} \leq \delta,\ \|R_2\|_{L^{\infty}_t \dot H^{1/2}_x} \leq \delta,\ \|\dot \pi^0_2\|_{\infty} \leq \delta,\ \|\dot \pi_2\|_{\infty} \leq \delta,
$$
and $h_j := h(R_{0j}, R^0_j, \pi^0_j)$ are the unique values for which $P_{im} R_j \in L^{\infty}_t \dot H^{1/2}_x$ by Lemma~\ref{hyp}.

If $\delta>0$ is sufficiently small, then
\be\begin{aligned}\lb{comp_h12}
&\|(R_1, \pi_1) - (R_2, \pi_2)\|_{e^{t\rho} L^{\infty}_t \dot H^{1/2}_x \times \partial_t^{-1} e^{t\rho} L^{\infty}_t} \les \\
&\les \delta \|(R_1^0, \pi_1^0) - (R_2^0, \pi_2^0)\|_{e^{t\rho} L^{\infty}_t \dot H^{1/2}_x \times \partial_t^{-1} e^{t\rho} L^{\infty}_t} + \|R_{01}-R_{02}\|_{\dot H^{1/2}} + |p_{01} - p_{02}|, \\
|h_1-h_2| &\les \delta \big(\|(R_1^0, \pi_1^0) - (R_2^0, \pi_2^0)\|_{e^{t\rho} L^{\infty}_t \dot H^{1/2}_x \times \partial_t^{-1} e^{t\rho} L^{\infty}_t} + \|R_{01}-R_{02}\|_{\dot H^{1/2}} + |p_{01} - p_{02}|\big).
\end{aligned}\ee
If in addition $\|(R^0_2, \pi^0_2)\|_X \leq \delta$, hence $\|(R_2, \pi_2)\|_X \leq \delta$, then
\be\begin{aligned}\lb{2.117}
\|(R_1, \pi_1) - (R_2, \pi_2)\|_Y &\les \delta \|(R_1^0, \pi_1^0) - (R_2^0, \pi_2^0)\|_Y + \|R_{01}-R_{02}\|_{\dot H^{1/2}} + |p_{01} - p_{02}|.
\end{aligned}\ee
Moreover, uniformly for all $n \geq 0$,
\be\lb{polywei}\begin{aligned}
\|(R_1, \pi_1) - (R_2, \pi_2)\|_{Y_{n+2}} &\les \delta \|(R_1^0, \pi_1^0) - (R_2^0, \pi_2^0)\|_{Y_n} + \|R_{01}-R_{02}\|_{\dot H^{1/2}} + |p_{01} - p_{02}|.
\end{aligned}\ee
\end{lemma}
$\delta$ and various constants depend on the value of $\rho$ used in the definition of $Y$ (\ref{Y}) and on $p_0$. The permissible choice of $\rho$ is a function of the scaling parameter $\alpha_0$ of $p_0$.

Note that if $(R_j, \pi_j) = \Epsilon(R^0_j, \pi^0_j)$ and $\|(R^0_j, \pi^0_j)\|_X \leq \delta$, then by Proposition \ref{prop9} $\|(R_j, \pi_j)\|_X \leq \delta$ for $j=1, 2$. This is sufficient for the contraction argument to work in Proposition \ref{prop27}.

The more general statement (\ref{comp_h12}) is useful in the proof of Proposition \ref{prop_2.13}.

\begin{observation}
A comparison estimate also exists for different starting solitons, $\pi^0_j(0)=\pi_j(0)=p_{0j}$, $j=1,2$. However, one has to use $H^{1/2}$ instead of $\dot H^{1/2}$ for comparison purposes, because only $H^{1/2}$ is preserved by boost transformations $R(x) \mapsto e^{ivx} R(x)$.
\end{observation}

\begin{proof}[Proof of Lemma \ref{lemma_5}]
$R_j$, $j = 1, 2$, satisfy the equations
\be\lb{rj}
i \partial_t R_j + H_{\pi_j^0}(t) R_j = F_j,
\ee
with initial data
$$
R_j(0) = R_{0j} + h_j F^+(W(p_{0j})).
$$
Furthermore, $R_j$ satisfy the orthogonality conditions
$$
\langle R_j(t), \partial^*_fW_{\pi^0_j}(t) \rangle = 0
$$
at time $t=0$ and thus, due to equation (\ref{ec_liniara12}), at every time $t$.

Let $R = R_1-R_2$, $\pi = \pi_1-\pi_2$. By hypothesis, $\|R\|_{L^{\infty}_t \dot H^{1/2}_x} \les \delta$ and $\pi(0) = 0$. Subtracting the equations (\ref{rj}) of $R_1$ and $R_2$ from one another, we obtain an equation for $R$:
\be\lb{eqnr}
i \partial_t R + H_{\pi^0_1}(t) R = \tilde F,\ \tilde F = F_1 - F_2 - \big(H_{\pi^0_1}(t) - H_{\pi^0_2}(t)\big) R_2.
\ee
Choose the Hamiltonian $H_{\pi_1^0}(t)$ (the choice of $H_{\pi_1^0}(t)$ or $H_{\pi_2^0}(t)$ is arbitrary) and apply the isometry $U(t)$ defined by (\ref{2.63}-\ref{2.67}) to (\ref{eqnr}):
\be
U(t) = e^{\textstyle\int_0^t(2 v_1^0(s) \dl + i ((\alpha_1^0)^2(s)-|v_1^0(s)|^2) \sigma_3) \dd s}.
\ee
Let
\be
Z(t) := U(t) R(t).
\ee
We introduce notations similar to (\ref{2.68}):
\be\begin{aligned}
H(t) &:= H(W(\pi^0_1(t))), &P_0(t) &:= P_0(W(\pi^0_1(t))),\\
P_c(t) &:= P_c(W(\pi^0_1(t))), &P_{im}(t) &:= P_{im}(W(\pi^0_1(t))),\\
P_+(t) &:= P_+(W(\pi^0_1(t))), &P_-(t) &:= P_-(W(\pi^0_1(t))),\\
F^+(t) &:= F^+(W(\pi^0_1(t))), &F^-(t) &:= F^-(W(\pi^0_1(t))),\\
\partial^*_fW(t) &:= \partial^*_fW(\pi^0_1(t)), &\sigma(t) &:= \sigma(W(\pi^0_1(t))).
\end{aligned}\ee
The equation for $Z$ becomes
$$
i \partial_t Z + H(t) Z = U(t) \tilde F.
$$
The initial value $Z(0) = (h_1-h_2) F^+(W(p_{01})) + R_{01}-R_{02} + h_2 (F^+(W(p_{01})) - F^+(W(p_{02})))$ satisfies the bound
$$
\|Z(0)\|_{\dot H^{1/2}} \les |h_1-h_2| + \|R_{01}-R_{02}\|_{\dot H^{1/2}} + \delta |p_{01} - p_{02}|.
$$
Split $Z$ into three parts, according to the Hamiltonian's spectrum:
$$
Z(t) = P_c(t) Z(t) + P_0(t) Z(t) + P_{im}(t) Z(t).
$$
We bound each component of $Z$ as we did in the proof of Proposition \ref{prop9}, but here we use weighted norms in $t$ instead of uniform norms. Another difference is that $P_0(t) Z(t)$ is no longer null, because the orthogonality condition does not hold.

We estimate $\tilde F = F_1 - F_2 - (H_{\pi^0_1}(t) - H_{\pi^0_2}(t)) R_2$ term by term along the lines of (\ref{2.49}). Note that
$$
|\pi^0_1(t) - \pi^0_2(t)| \les |p_{01} - p_{02}| + \min(e^{t\rho} \|\dot \pi\|_{e^{t\rho} L^1_t}, \langle \rho \rangle^{-1} e^{t\rho} \|\dot \pi\|_{L^{\infty}_t}).
$$
Then in any Schwartz seminorm $\mc S_n$
$$\begin{aligned}\lb{expo}
&\|W_{\pi^0_1}(t) - W_{\pi^0_2}(t)\|_{e^{t\rho} L^{\infty}_t (\mc S_n)_x} \les \\
&\les |p_{01} - p_{02}| + \sup_t \Big(e^{-t\rho} |\pi^0_1(t) - \pi^0_2(t)| + e^{-t\rho} \int_0^t |\pi^0_1(s) - \pi^0_2(s)| \dd s\Big) \\
&\les |p_{01} - p_{02}| + \sup_t \min\Big(\|\dot \pi(t)\|_{e^{t\rho} L^1_t} \Big(1 + e^{-t\rho} \int_0^t e^{s\rho} \dd s\Big), \langle \rho^{-1} \rangle \|\dot \pi(t)\|_{e^{t\rho} L^{\infty}_t} \Big(1 + e^{-t\rho} \int_0^t e^{s\rho} \dd s\Big)\Big) \\
&\les |p_{01} - p_{02}| + \min(\langle \rho^{-1} \rangle \|(R_1^0, \pi_1^0) - (R_2^0, \pi_2^0)\|_Y, \langle \rho \rangle^{-2} \|\dot \pi^0_1 - \dot \pi^0_2\|_{e^{t\rho} L^{\infty}_t}).
\end{aligned}$$
Integrating in time preserves the spaces $e^{t\rho} L^p_t$, $1 \leq p \leq \infty$, at the cost of a factor of $\rho^{-1}$:
$$
\Big\|\int_0^t f(s) \dd s \Big\|_{e^{t\rho} L^p_t} \les \langle \rho^{-1} \rangle \|f\|_{e^{t\rho} L^p_t}.
$$
This inequality is equivalent to stating that convolution with $\chi_{t\geq 0} e^{-t\rho}$ is a bounded operation on $L^p$.

When using the polynomial weights $A_n(t)$ instead, integration in $t$ raises $n$ by $1$.

In particular, we obtain that when $\|R_2\|_{L^{\infty}_t \dot H^{1/2}_x} \leq \delta$
$$\begin{aligned}
\|(H_{\pi^0_1}(t) - H_{\pi^0_2}(t)) R_2\|_{e^{t \rho} L^{\infty}_t \dot H^{1/2}_x} &\les \|H_{\pi^0_1}(t) - H_{\pi^0_2}(t)\|_{e^{\rho t} L^{\infty}_t \dot H^{1/2}_x}  \|R_2\|_{L^{\infty}_t \dot H^{1/2}_x} \\
&\les \delta \langle \rho^{-2} \rangle \big(|p_{01} - p_{02}| + \|\dot \pi^0_1 - \dot \pi^0_2\|_{e^{t\rho} L^{\infty}_t}\big).
\end{aligned}$$
If in addition $\|(R_2, \pi_2)\|_X \leq \delta$, then
$$\begin{aligned}
\|(H_{\pi^0_1}(t) - H_{\pi^0_2}(t)) R_2\|_{e^{t \rho} L^2_t \dot W^{1/2, 6/5}_x} &\les \|H_{\pi^0_1}(t) - H_{\pi^0_2}(t)\|_{e^{\rho t} L^{\infty}_t (\dot W^{1/2, 6/5}_x \cap L^{3/2-\epsilon})} \|R_2\|_{L^2_t \dot W^{1/2, 6}_x} \\
&\les \delta \langle \rho^{-1}\rangle \big(|p_{01} - p_{02}| + \|\dot \pi^0_1 - \dot \pi^0_2\|_{e^{t\rho} L^1_t}\big).
\end{aligned}$$
The linear terms $L_{\pi^0_1} R_1 - L_{\pi^0_2} R_2$ satisfy similar inequalities, resulting in
$$
\|L_{\pi^0_1} R_1 - L_{\pi^0_2} R_2\|_{e^{t\rho} L^{\infty}_t \dot H^{1/2}_x} \les \delta \langle \rho^{-2} \rangle \big(\|R_1 - R_2\|_{L^{\infty}_t \dot H^{1/2}_x} + |p_{01} - p_{02}| + \|\dot \pi_1^0 - \dot \pi_2^0\|_{e^{t\rho} L^{\infty}_t}\big).
$$
Assuming that $\|(R^0_2, \pi^0_2)\|_X \leq \delta$, one has
$$\begin{aligned}
&\|L_{\pi^0_1} R_1 - L_{\pi^0_2} R_2\|_{e^{t\rho} L^2_t \dot W^{1/2, 6/5}_x} \les \\
&\les \delta \langle \rho^{-1} \rangle \big(\|(R_1, \pi_1) - (R_2, \pi_2)\|_Y + \|(R^0_1, \pi^0_1) - (R^0_2, \pi^0_2)\|_Y + |p_{01} - p_{02}|\big).
\end{aligned}$$
Evaluate the cubic terms through the following identity:
$$\begin{aligned}
r_1^2 \ov r_1 - r_2^2 \ov r_2 &= r_1^2 (r_1 - r_2) + 2|r_1|^2 (r_1 - r_2) - 2r_1 |r_1 - r_2|^2 - (r_1 - r_2)^2 \ov r_1 \\
&+ |r_1 - r_2|^2 (r_1 - r_2).
\end{aligned}$$
Then
$$
\|r_1 |r_1 - r_2|^2\|_{e^{t\rho} L^2_t \dot W^{1/2, 6/5}_x} \les \|r_1\|_{L^{\infty}_t \dot H^{1/2}_x} \big\|(r_1-r_2)(\ov r_1 - \ov r_2)\big\|_{e^{t\rho} L^2_t \dot H^{1/2}_x}.
$$
Then note that
$$\begin{aligned}
\|f g\|_{L^2_t \dot H^{1/2}_x} &\les \|(f+g)^2\|_{L^2_t \dot H^{1/2}_x} + \|f^2\|_{e^{t\rho} L^2_t \dot H^{1/2}_x} + \|g^2\|_{e^{t\rho} L^2_t \dot H^{1/2}_x} \\
& \les \|f+g\|_{e^{t\rho/2} L^4_t \dot W^{1/2, 3}_x}^2 + \|f\|_{e^{t\rho/2}L^4_t \dot W^{1/2, 3}_x}^2 + \|f\|_{e^{t\rho/2}L^4_t \dot W^{1/2, 3}_x}^2 \\
& \les (\|f\|_{L^{\infty}_t \dot H^{1/2}_x} + \|g\|_{L^{\infty}_t \dot H^{1/2}_x}) (\|f\|_{e^{t\rho} L^2_t \dot W^{1/2, 6}_x} + \|g\|_{e^{t\rho} L^2_t \dot W^{1/2, 6}_x}).
\end{aligned}$$
The other terms compute in the same manner. For example,
$$\begin{aligned}
w_{\pi_1^0} |r_1^0|^2 - w_{\pi_2^0} |r_2^0|^2 &= (w_{\pi_1^0} - w_{\pi_2^0}) |r_1^0|^2 + w_{\pi_2^0} (r_1^0 - r_2^0) \ov{r_1^0} + w_{\pi_2^0} r_1^0 (\ov r_1^0 - \ov r_2^0) \\
& - w_{\pi_2^0} |r_1^0 - r_2^0|^2.
\end{aligned}$$
We treat each component in the manner demonstrated above, i.e.
$$\begin{aligned}
\|w_{\pi_2^0} (r_1^0 - r_2^0) \ov{r_1^0}\|_{e^{t\rho} L^2_t \dot W^{1/2, 6/5}_x} &\les \|w_{\pi_2^0} \ov{r_1^0}\|_{L^2_t \dot H^{1/2}_x} \|e^{-t\rho} (r_1^0 - r_2^0)\|_{L^{\infty}_t \dot H^{1/2}_x} \\
&\les \delta^2 \|r_1^0 - r_2^0\|_{e^{t\rho} L^{\infty}_t \dot H^{1/2}_x}.
\end{aligned}$$
The nonlinear terms then satisfy the bound
$$\begin{aligned}
\big\|N(R_1^0, \pi_1^0) - N_{\pi_1^0}(R_1^0, \pi_1^0) - \big(N(R_2^0, \pi_2^0) - N_{\pi_2^0}(R_2^0, \pi_2^0)\big)\big\|_{e^{t\rho} L^2_t \dot W^{1/2, 6/5}_x} \les \\
\les \delta^2 \langle \rho \rangle^{-2} (\|R_1^0 - R_2^0\|_{e^{t\rho} L^{\infty}_t \dot H^{1/2}_x} + \|\dot \pi_1^0 - \dot \pi_2^0\|_{e^{\rho t} L^{\infty}_t} + |p_{01} - p_{02}|).
\end{aligned}$$
The conclusion is that
$$\begin{aligned}
&\|\tilde F\|_{e^{t\rho} L^2_t \dot W^{1/2, 6/5}_x + e^{t\rho} L^{\infty}_t \dot H^{1/2}_x} \les \\
&\les \delta \langle \rho^{-2} \rangle \big(\|R\|_{e^{t\rho} L^{\infty}_t \dot H^{1/2}_x} + \|R_1^0 - R_2^0\|_{e^{t\rho} L^{\infty}_t \dot H^{1/2}_x} + \|\dot \pi_1^0 - \dot \pi_2^0\|_{e^{t\rho} L^1_t} + |p_{01} - p_{02}|\big)
\end{aligned}$$
as well as
\be\begin{aligned}
\|\tilde F\|_{e^{t\rho} L^2_t \dot W^{1/2, 6/5}_x} \les \delta \langle \rho^{-1} \rangle \big(\|(R, \pi)\|_Y + \|(R_1^0, \pi_1^0) - (R_2^0, \pi_2^0)\|_Y + |p_{01} - p_{02}|\big).
\lb{2.79}
\end{aligned}\ee

The continuous spectrum projection $P_c(t) Z$ has the equation
$$\begin{aligned}
i\partial_t(P_c(t) Z) + H(t) P_c(t) Z &= P_c(t) U(t) \tilde F + i (\partial_t P_c(t)) Z.
\end{aligned}$$
By Corollary \ref{cor_exp}, we obtain an exponentially weighted Strichartz estimate:
%
\be\begin{aligned}
&\|P_c(t) Z\|_{e^{t\rho} L^{\infty}_t \dot H^{1/2}_x} \les \\
&\les \|P_c(0) Z(0)\|_{\dot H^{1/2}} + \langle \rho \rangle^{-1} \|\tilde F\|_{e^{t\rho} L^2_t \dot W^{1/2, 6/5}_x + e^{t\rho} L^{\infty}_t \dot H^{1/2}_x} +  \|(\partial_t P_c(t)) Z\|_{e^{t\rho} L^1_t \dot H^{1/2}_x} \\
&\les \|R_{01}-R_{02}\|_{\dot H^{1/2}} + \delta \langle \rho^{-3} \rangle \big(|p_{01} - p_{02}| + \|R\|_{e^{t\rho} L^{\infty}_t \dot H^{1/2}_x} + \|R_1^0 - R_2^0\|_{e^{t\rho} L^{\infty}_t \dot H^{1/2}_x} +  \|\dot \pi_1^0 - \dot \pi_2^0\|_{e^{t\rho} L^1_t}\big).
\lb{2.71}
\end{aligned}\ee
There is no contribution of $h_1-h_2$, since $P_c(0) Z(0)$ does not depend on it.

If $\|(R^0_2, \pi^0_2)\|_X \leq \delta$, then likewise
$$
\|P_c(t) Z\|_{e^{t\rho} L^{\infty}_t \dot H^{1/2}_x} \les \|R_{01}-R_{02}\|_{\dot H^{1/2}} + \delta \langle \rho^{-2} \rangle \big(|p_{01} - p_{02}| + \|(R, \pi)\|_Y + \|(R_1^0, \pi_1^0) - (R_2^0, \pi_2^0)\|_Y\big).
$$


Taking the difference of the orthogonality conditions satisfied by $R_1$ and $R_2$ leads to
\be
\langle R_1(t) - R_2(t), \partial^*_fW_{\pi^0_1}(t) \rangle = \langle R_2(t), \partial^*_fW_{\pi^0_2}(t) - \partial^*_fW_{\pi^0_1}(t) \rangle.
\lb{96}
\ee
Applying the isometry $U(t)$ to the left-hand side of (\ref{96}), one obtains that
$$
\langle Z(t), \partial^*_fW(t) \rangle \equiv \langle R_1(t) - R_2(t), \partial^*_fW_{\pi^0_1}(t) \rangle.
$$
Hence, according to whether $(R_2, \pi_2) \in X$,
\be\lb{2.75}\begin{aligned}
\|P_0(t) Z(t)\|_{e^{t\rho} L^{\infty}_t \dot H^{1/2}_x} &\les \delta \langle \rho^{-2} \rangle (|p_{01} - p_{02}| + \|\dot \pi_1^0 - \dot \pi_2^0\|_{e^{t\rho} L^1_t}),\\
\|P_0(t) Z(t)\|_{e^{t\rho} L^{\infty}_t \dot H^{1/2}_x \cap e^{t\rho} L^2_t \dot W^{1/2, 6/5}_x} &\les \delta \langle \rho^{-1} \rangle (|p_{01} - p_{02}| + \|(R_1^0, \pi_1^0) - (R_2^0, \pi_2^0)\|_Y).
\end{aligned}\ee
Lemma \ref{hyp} applies to the imaginary component $P_{im}(t) Z(t)$, because $P_{im}(t) Z(t) \in L^{\infty}_t \dot H^{1/2}_x$. 
Let $P_{im}(t) Z(t) = h^+(t) F^+(t) + h^-(t) F^-(t)$, see Section \ref{spectru}. Then
$$\begin{aligned}
\partial_t h^{\pm} &= \pm \sigma(t) b^{\pm} - \langle U \tilde F, \sigma_3 F^{\mp} \rangle - \\
& -3 \dot \alpha^0 (\alpha^0)^{-4} \langle Z, i \sigma_3 F^{\mp} \rangle +  (\alpha^0)^{-3} \langle Z, i \sigma_3 (d_{\pi} F^{\mp}) \dot \pi^0 \rangle.
\end{aligned}$$
Thus $h^+$ and $h^-$ solve the system of equations
$$
\partial_t  \bpm h^-\\ h^+\epm + \bpm \sigma(t) & 0 \\ 0 & -\sigma(t)\epm \bpm h^-\\ h^+\epm = \bpm N_-(t) \\ N_+(t) \epm,
$$
where
$$
N_{\pm} = -\langle U \tilde F, \sigma_3 F^{\mp} \rangle -3 \dot \alpha^0_1 (\alpha^0_1)^{-4} \langle Z, i \sigma_3 F^{\mp} \rangle +  (\alpha^0_1)^{-3} \langle Z, i \sigma_3 (d_{\pi} F^{\mp}) \dot \pi^0_1 \rangle.
$$
Thus
$$
\|N_{\pm}\|_{e^{t\rho} L^{\infty}_t} \les \delta \langle \rho^{-2} \rangle \big(\|R\|_{e^{t\rho} L^{\infty}_t \dot H^{1/2}_x} + \|\dot \pi\|_{e^{t\rho} L^1_t} + \|R_1^0 - R_2^0\|_{e^{t\rho} L^{\infty}_t \dot H^{1/2}_x} +  \|\dot \pi_1^0 - \dot \pi_2^0\|_{e^{t\rho} L^1_t}\big).
$$
By Lemma \ref{hyp},
$$\begin{aligned}
|h_1-h_2| = |h^+(0)| &\les
\delta \langle \rho^{-1} \rangle \big(\|(Z_1^0, \pi_1^0) - (Z_2^0, \pi_2^0)\|_Y+ \|Z\|_{e^{t\rho} \dot H^{1/2}_x}\big)
\end{aligned}$$
as well as
$$\begin{aligned}
|h_1-h_2| &\les
\delta \langle \rho^{-2} \rangle \big(\|R_1^0 - R_2^0\|_{L^{\infty}_t \dot H^{1/2}_x} + \|\dot \pi_1^0 - \dot \pi_2^0\|_{e^{t\rho} L^1_t} + \|R\|_{L^{\infty}_t \dot H^{1/2}_x} + \|\dot \pi\|_{e^{t\rho} L^1_t}\big).
\end{aligned}$$
Also, according to whether $(R_0^2, \pi_0^2) \in X$,
\be\begin{aligned}
\|P_{im} Z\|_{e^{t\rho} L^{\infty}_t \dot H^{1/2}_x} &\les \delta \langle \rho ^{-2} \rangle (|p_{01} - p_{02}| + \|R_1^0 - R_2^0\|_{L^{\infty}_t \dot H^{1/2}_x} + \|\dot \pi_1^0 - \dot \pi_2^0\|_{e^{t\rho} L^1_t} + \|R\|_{L^{\infty}_t \dot H^{1/2}_x} + \|\dot \pi\|_{e^{t\rho} L^1_t}),\\
\|P_{im} Z\|_{Y} &\les \delta \langle \rho^{-1} \rangle (|p_{01} - p_{02}| + \|(Z_1^0, \pi_1^0) - (Z_2^0, \pi_2^0)\|_Y+ \|(R, \pi)\|_Y).
\lb{2.77}
\end{aligned}\ee

Finally, we compare the modulation equations (\ref{ec_liniara12}) and evaluate the difference term by term as in (\ref{modexp}). $\dot \pi(t)$ on the left-hand side is bounded by terms of size $\int_0^t |\pi(s)| \dd s$ on the right-hand side. Handling both kinds of terms in the same space requires exponential weights, since an inequality of the form
$$
\Big\|\int_0^t f(s) \dd s\Big\|_B \les \|f\|_B
$$
holds only in exponentially weighted spaces $B$ such as $e^{t\rho} L^p_t$, $\rho>0$. Up to here, we could have used polynomial weights instead of exponential weights in the argument, but this makes exponential weights necessary to close the loop.

In spaces with polynomial weights, integrating twice in $t$ gains two powers of $t$. On the scale of $A_n$ weights defined by (\ref{2.230}), when the right-hand side is in $Y_n$, the left-hand side is in $Y_{n+2}$.

Then $\pi = \pi_1 - \pi_2$ fulfills
$$
\|\dot \pi\|_{e^{t\rho} L^{\infty}_t} \les \delta \langle \rho^{-2} \rangle \big(\|R_1^0 - R_2^0\|_{e^{t\rho} L^{\infty}_t \dot H^{1/2}_x} + \|\dot \pi_1^0 - \dot \pi_2^0\|_{e^{t\rho} L^{\infty}_t \dot H^{1/2}_x} + \|R_1 - R_2\|_{e^{t\rho} L^{\infty}_t \dot H^{1/2}_x}\big)
$$
as well as
\be\begin{aligned}
\|\dot \pi\|_{e^{t\rho} L^1_t} &\les \delta \langle \rho^{-2} \rangle \big(\|(Z_1^0, \pi_1^0) - (Z_2^0, \pi_2^0)\|_Y + \|(Z, \pi)\|_Y\big).
\lb{2.78}
\end{aligned}\ee

Putting (\ref{2.71}), (\ref{2.75}), (\ref{2.77}), and (\ref{2.78}) together we obtain (\ref{2.117}). Alternately, performing the same computation with polynomial weights leads to (\ref{polywei}) and, absent the assumption that $(R^0_2, \pi^0_2) \in X$, to (\ref{comp_h12}).
\end{proof}



\subsection{Invariance and analyticity of $\mc N$}
Take a small parameter $\delta_0>0$ and~let
$$\begin{aligned}
\mc N_{lin}(W) &= \{R_0 \in (P_c(W) + P_-(W)) \dot H^{1/2} \mid \|R_0\|_{\dot H^{1/2}} < \delta_0 \}.
\end{aligned}$$
We formally conclude the contraction argument.

\begin{proposition}[Asymptotical stability]\lb{prop27} For each soliton $W_0=W(\pi_0) \in \Sol$ there exists $\delta_0 = \delta_0(\alpha_0)$ such that there is a map $h: \mc N_{lin}(W_0) \to \R$ such that
\begin{enumerate}
\item $h$ is locally Lipschitz continuous in $R$,
\item $|h(R_0, W_0)| \les \alpha_0 \|R_0\|_{\dot H^{1/2}}^2$, where $\alpha_0$ is the scaling parameter of $W_0$,
\end{enumerate}
and initial data
\be
\Psi(0) := \mc F_{W_0}(R_0) = W_0 + R_0 + h(R_0, W_0) F^+(W_0)
\ee
gives rise to an asymptotically stable solution $\Psi$ to \eqref{NLS} with $\Psi(0) = \mc F(R_0, W_0)$ and $\Psi(t) = W_{\pi}(t) + R(t)$.\\
$W_{\pi}(t)$ is a moving soliton with $W_{\pi}(0) = W_0$, parametrized as in (\ref{1.1}) by a path $\pi$ such~that
$$
\|\dot \pi\|_1 \les \alpha_0 \|R_0\|_{\dot H^{1/2}}^2.
$$
$R$ has initial data $R(0) = R_0 + h(R_0, W_0) F^+(W_0)$ and is in the endpoint Strichartz space
$$
\|R\|_{L^{\infty}_t \dot H^{1/2}_x \cap L^2_t \dot W^{1/2, 6}_x} \les  \delta.
$$
\end{proposition}
Consider an asymptotically stable solution of (\ref{NLS}) obtained by Proposition \ref{prop27}. At a soliton $W_0=W(p_0)$, such solutions have initial values of the form
$$
\Psi = W_0 + R_0 + h(R_0, W_0) F^+(W_0).
$$
$F^+(W_0)$ is the normalized eigenvector of $H(W_0)$ for the eigenvalue $i\sigma(W_0)$; see Section~\ref{spectru}. $R_0$ belongs to the codimension-nine affine subspace $\mc N_{lin}(W_0)$ of $\dot H^{1/2}$ 
and $h(R_0, W_0)$ is defined by Proposition \ref{prop27}.

For sufficiently small $\delta_0 > 0$, by Lemma \ref{lemma_10} $\|R_0\|_{\dot H^{1/2}}$ is comparable to the distance from $\Psi(0)$ to $\Sol$
\be
\min_{W \in \Sol} \|\Psi(0)-W\|_{\dot H^{1/2}}.
\ee
Thus we can substitute one for the other in the statement of Proposition \ref{prop27}.

\begin{proof}[Proof of Proposition \ref{prop27}]
Take $\delta$ proportional to $\|R_0\|_{\dot H^{1/2}}$, such that $0<\delta<\delta_0$. From $(R_0(t), \pi_0(t)):=(0, p_0)$, recursively define the sequence
$$
(R_{n+1}, \pi_{n+1}) := \Epsilon(R_n, \pi_n)
$$
with initial data $R(0) = R_0 + h(R_0, R_n, \pi_n) F^+(W_0)$.

By Proposition \ref{prop9}, $\|(R_{n+1}, \pi_{n+1})\|_X \les \delta$ for every $n \geq 0$. Then, by Lemma \ref{lemma_5}, for sufficiently small $\delta$
$$\begin{aligned}
\|(R_{n+1}, \pi_{n+1}) - (R_n, \pi_n)\|_Y + |h(R_0, R_{n+1}, \pi_{n+1}) - h(R_0, R_n, \pi_n)| \leq \\
\leq \frac 1 2 \|(R_n, \pi_n) - (R_{n-1}, \pi_{n-1})\|_Y.
\end{aligned}$$
$(R_n, \pi_n)$ and the associated parameters $h_n: = h(R_0, R_n, \pi_n)$ form Cauchy sequences in $Y$, respectively in $\C$.

Let $(R, \pi):=\lim_{n \to \infty} (R_n, \pi_n)$ in $Y$. Then $(R, \pi)$ is a fixed point of $\Epsilon$ and, by Lemma \ref{lemma1.1}, a solution to (\ref{NLS}) with the specified initial data, weakly and locally in time, thus strongly and globally also. By passing to the strong $Y$ limit, hence to the weak limit in~$X$,
$$
\|(R, \pi)\|_X \leq \limsup_n \|(R_n, \pi_n)\|_X \leq \delta.
$$
This follows on every finite time interval $[0, T]$ and then in the limit on $[0, \infty)$.

Let $h = h(R_0, W_0) := \lim_{n \to \infty} h_n$. By passing to the limit in
$$
P_+(W_0) R_n(0) = h(R_0, R_n, \pi_n) F^+(W_0),
$$
we get that $P_+(W_0) R(0) = h(R_0, R, \pi) F^+(W_0)$. Thus $h=h(R_0, R, \pi)$.
\end{proof}

At this point we define the centre-stable manifold $\mc N$. Let $\delta_0$ be given by Proposition~\ref{prop27}.
\begin{definition}\lb{def3} Let $\Sol$ be the eight-dimensional soliton manifold and for $W \in \Sol$
\be\begin{aligned}
\mc N_{lin}(W) &:= \{R_0 \in (P_c(W) + P_-(W)) \dot H^{1/2} \mid \|R_0\|_{\dot H^{1/2}} < \delta_0 \} \\
\mc N(W) &:= \{W + R_0 + h(R_0, W) F^+(W) \mid R_0 \in \mc N_{lin}(W)\} \\
\mc N &:= \bigcup_{W \in \Sol} \mc N(W).
\end{aligned}\ee
\end{definition}
Note that solutions $\Psi$ defined by Proposition \ref{prop27} are exactly those with $\Psi(0) \in \mc N$.

For each fixed soliton $W_0 \in \Sol$, $\mc N_{lin}(W_0)$ is a codimension-nine closed linear subspace of $\dot H^{1/2}$ and $\mc N(W_0)$ is its image under the mapping
\be
\mc F_{W_0}(R_0) := W_0 + R_0 + h(R_0, W_0) F^+(W_0).
\lb{2.105}
\ee
Below we show that $\mc F_{W_0}$ is real-analytic, hence defines a real-analytic structure for $\mc N(W_0)$.

More generally, consider the exponential map $\exp_{W_0}:T_{W_0} \Sol \to \Sol$ and let $W_{R_0} := \exp_{W_0}(P_0(R_0))$. Define, for $\|R_0\|_{\dot H^{1/2}} \leq \delta_0 << 1$,
\be\lb{local_chart}
\mc F_{W_0}(R_0) := W_{R_0} + (I - P_0(W_{R_0})) R_0 + h(\big(P_c(W_{R_0}) + P_-(W_{R_0})\big) R_0, W_{R_0}) F^+(W_{R_0}).
\ee
Since $h(R_0, W_0)$ is real-analytic, so is $\mc F_{W_0}$. Furthermore, $d\mc F_{W_0}(0) = I$, because
$$\begin{aligned}
\mc F_{W_0}(R_0) &= W_0 + dW_0\, P_0(R_0) + (I - P_0(W_0)) R_0 + O(\|R_0\|_{\dot H^{1/2}}^2) \\
&= W_0 + R_0 + O(\|R_0\|_{\dot H^{1/2}}^2),
\end{aligned}$$
with quadratic corrections. Thus, $\mc F_{W_0}$ is locally invertible at $0 \in \dot H^{1/2}$.

$\mc N$ is locally the $\mc F_{W_0}$ image of the set $\{\|R_0\|_{\dot H^{1/2}} \leq \delta_0 \mid P_+(W_{R_0}) R_0 = 0\}$, which has a natural real-analytic manifold structure. Indeed, let
$$
P_+(W_{R_0}) R_0 = \langle R_0, i \sigma_3 F^-(W_{R_0}) \rangle F^+(W_{R_0}) := f(R_0) F^+(W_{R_0}).
$$
$f(R_0)=\langle R_0, i \sigma_3 F^-(W_{R_0}) \rangle$ is an analytic function. Its differential at $0 \in \dot H^{1/2}$ is $df(0)\, R = \langle R, i \sigma_3 F^-(W_0) \rangle$. Then, for $\epsilon <<1$ the set $\{\|R_0\|_{\dot H^{1/2}} \leq \delta \mid f(R_0) = \epsilon\}$ is a real-analytic manifold.

This shows both that $\mc N$ is real-analytic and that it is analytically embedded into~$\dot H^{1/2}$, with local coordinate charts $\mc F_{W_0}$ given by (\ref{local_chart}) on $\{\Psi \in \dot H^{1/2} \mid \|\Psi - W_0\|_{\dot H^{1/2}} \les \delta\}$.




Applying Lemma \ref{lemma_5} to solutions of (\ref{NLS}) obtained by Proposition \ref{prop27} leads to this:
\begin{proposition}\lb{proposition_10}[Continuous dependence on initial data] For $W_0 \in \Sol$, consider solutions $\Psi_1$ and $\Psi_2$ to (\ref{NLS}) given by Proposition \ref{prop27}, such that
$$
\|\Psi_1(0) - W_0\|_{\dot H^{1/2}} < \delta(W_0), \|\Psi_2(0) - W_0\| < \delta(W_0).
$$
Then
$$\begin{aligned}
\|\Psi_1-\Psi_2\|_{e^{t\rho} L^{\infty}_t \dot H^{1/2}_x \cap e^{t\rho} L^2_t \dot W^{1/2, 6/5}_x} \les \|\Psi_1(0)-\Psi_2(0)\|_{\dot H^{1/2}}
\end{aligned}$$
and
$$\begin{aligned}
&|h(R_{01}, W_{02}) - h(R_{02}, W_{02})| \les \big(d_{\dot H^{1/2}}(\Psi_1(0), \Sol) + d_{\dot H^{1/2}}(\Psi_2(0), \Sol)\big) \|\Psi_1(0)-\Psi_2(0)\|_{\dot H^{1/2}}.
\end{aligned}$$
\end{proposition}
%
The continuous dependence of solutions on the fiber, i.e.\ on $W_0$, is given by symmetry transformations for $H^{1/2}$ initial values.

In $\dot H^{1/2}$ boost is not a symmetry, so the dependence on initial data is given by symmetries only for five out of the eight modulation parameters. Proposition \ref{proposition_10} holds across fibers, with a constant that depends on $P[W_0]$.

\begin{proof} By Proposition \ref{prop27}, let $\Psi_j(t) = W_{\pi_j}(t) + R_j(t)$. Lemma \ref{lemma_5} yields
$$
\|(R_1, \pi_1) - (R_2, \pi_2)\|_Y \les \|R_{01}-R_{02}\|_{\dot H^{1/2}} + |\pi_{01} - \pi_{02}|,
$$
which implies the first conclusion, and
$$\begin{aligned}
|h(R_{01}, W_{01}) - h(R_{02}, W_{02})| \les \delta \big(\|R_{01}-R_{02}\|_{\dot H^{1/2}} + |p_{01} - p_{02}|\big).
\end{aligned}$$
Taking $\delta$ proportional to $\|R_{01}\|_{\dot H^{1/2}} + \|R_{02}\|_{\dot H^{1/2}}$ leads to the second conclusion.

By Lemma \ref{lemma_10}, $\|R_{01}-R_{02}\|_{\dot H^{1/2}} + |p_{01} - p_{02}|$ is comparable to $\|\Psi_1(0) - \Psi_2(0)\|_{\dot H^{1/2}}$.
\end{proof}

In particular, $\mc F$ defined by (\ref{2.105}) is locally Lipschitz continuous. In the sequel we establish the analiticity of $\mc F$.
\begin{definition}\lb{ana}
Given two Banach spaces $A$ and $B$, a map $f:A \to B$ is analytic if it admits a Taylor series expansion:
$$
f(a)=f_0 + f_1(a) + f_2(a, a) + \ldots,
$$
such that each $f_n$ is $n$-linear and there exists $C_1$ such that $\|f_n\|_{\mc B(A^{\otimes n}, B)} \les C_1^n$.
\end{definition}
With no loss of generality, $f_n$ are symmetric. Similar definitions work for differentiable, $C^n$, or smooth maps. Considering the Taylor series expansion of $f$
$$
f(a) = \sum_{n=0}^{\infty} d^n f(a_0)(a-a_0),
$$
note that $f_n(a_0) = \frac 1 {n!} d^n f(a_0)$.

The composition of analytic functions is analytic. Thus, one can define analytic maps on analytic manifolds by means of analytic local charts.

To show analyticity, we use the following characterization:
\begin{lemma} $f:A \to B$ is analytic if and only if there exist $n$-linear functions $f_n:A^{\otimes n} \to B$ such that, uniformly for all $n$,
$$
\Big\|f(a)-\sum_{n=0}^{\infty} f_n(a, \ldots, a)\Big\|_B \les \|a\|_A^{n+1}.
$$
\end{lemma}
\begin{proof} Apply Definition \ref{ana}.
\end{proof}

\begin{lemma} The map $\tilde F: \mc N_{lin} \times \R \to \dot H^{1/2}$,
\be
\tilde F(W, R, h) = W + R + h F^+(W),
\ee
is locally a real analytic diffeomorphism at each point where $h=0$, i.e.\ $(W, R, 0)$.
\lb{lemma_11}\end{lemma}
%

\begin{proof} Let $W_0 = W(p_0)$ and consider the first-order differential of $\tilde F$, $d\tilde F \mid_{(W_0, R_0, h_0)}: \R^8 \times \mc N_{lin}(W_0) \times \R \to \dot H^{1/2}$:
\be\begin{aligned}
d\tilde F \mid_{(W_0, R_0, h_0)}(\delta \pi, \delta R, \delta h) &= (d_{\pi} W_0) \delta \pi + \delta R + (\delta h) F^+(W_0) + h_0 d_{\pi} F^+(W_0)\,\delta \pi.
\end{aligned}\ee
$d\tilde F$ is bijective when $h_0=0$ due to the following identity:
\be\begin{aligned}
\Psi &= P_0(W_0) \Psi + P_c(W_0) \Psi + P_-(W_0) \Psi + P_+(W_0) \Psi \\
&= \sum_{f \in \{\alpha, \Gamma, v_k, D_k\}} \langle \Psi, \partial^*_fW_0 \rangle \partial_f W_0 + (P_c(W_0) + P_-(W_0)) \Psi + \\
&+ \alpha_0^{-3} \langle \Psi, i \sigma_3 F^-(W_0) \rangle F^+(W_0).
\end{aligned}\ee
This produces an explicit inverse for the linear first-order differential $d\tilde F$: if
\be
\Psi = d\tilde F \mid_{(W_0, R_0, h_0)}(\delta \pi, \delta R, \delta h),
\ee
then the inverse is given by
\be\begin{aligned}
\delta \pi &= (\langle \Psi, \partial^*_{\alpha}W_0 \rangle, \langle \Psi, \partial^*_{\Gamma}W_0 \rangle, \langle \Psi, \partial^*_{v_k}W_0 \rangle, \langle \Psi, \partial^*_{D_k}W_0 \rangle), \\
\delta R &= (P_c(W_0) + P_-(W_0)) \Psi, \\
\delta h &= \alpha_0^{-3} \langle \Psi, i \sigma_3 F^-(W_0) \rangle.
\end{aligned}\ee
The inverse function theorem shows that $\tilde F$ is then locally invertible.

Smoothness comes from the explicit forms of $\partial_f W_0$ and $\partial_f F^+(W_0)$. Indeed, In Lemma \ref{lemma_2.10} we prove the analiticity of $W(p)$ and of its derivatives up to any finite order, as well as that of $F^+(W(p))$ and its derivatives. This implies that $\tilde F$ is real analytic.
\end{proof}

The analyticity of $\tilde F$ is closely tied to that of the soliton $W(p)$, of its derivatives, and of the eigenvectors $F^+(W(p))$, considered as functions of the parameters $p$. The analyticity of $W(p)$ was proved by Li--Bona \cite{bona}, but here we need a different statement.
\begin{lemma}\lb{lemma_2.10}
For an exponentially decaying ground state $\phi$ of equation \eqref{phi}
$$
-\Delta \phi(\cdot, \alpha) + \alpha^2 \phi(\cdot, \alpha) = \phi^3(\cdot, \alpha),
$$
the soliton
\be\lb{2.170}
w(p)(x) = e^{i(x \cdot v+\Gamma)} \phi(x - D, \alpha)
\ee
is a real analytic function of $p$ in any Schwartz seminorm $\mc S_n$. Furthermore, for any multiindex $\beta = (\beta_{\alpha}, \beta_{\Gamma}, \beta_{v_k}, \beta_{D_k})$
\be\lb{analitic_sol}
\|\partial^{\beta_{\alpha}}_{\alpha} \partial^{\beta_{\Gamma}}_{\Gamma} \partial^{\beta_{v_k}}_{v_k} \partial^{\beta_{D_k}}_{D_k} w\|_{\mc S_n} \leq C_1 C_2^{|\beta|} \prod_{f \in \{\alpha, v_k\}} \beta_f!.
\ee
Also consider the imaginary eigenstate $F^+(W(p))$ that solves the equation
$$
(\Delta - 1 + 2 |w(p)|^2) f^+(w(p)) + \phi^2 \ov f^+(w(p)) = i \sigma(w(p)) f^+(w(p)).
$$
$F^+(W(p))$ is a real analytic function of $p$ in any Schwartz seminorm $\mc S_n$.
\end{lemma}
We distinguish between two kinds of parameters for $w(p)$. In $\alpha$ and $v_k$, the domain of analyticity is a strip of the form $|\Re z| \leq z_0$. On the other hand, as a function of $\Gamma$ and $D_k$, $w(p)$ extends to an analytic function of exponential type on $\C$.

\begin{proof}
We first show that $w$ is an analytic $e^{-(1-\epsilon)|x|} L^{\infty}_x$-valued map. Since the derivatives of an analytic map are also analytic and $\partial_{D_k} w(p) = -\partial_{x_k} w(p)$, by iterating it follows that arbitrarily many derivatives of $w(p)$ are analytic maps into this space.
 
The real analyticity of $w$ is equivalent to its complex analyticity and the joint complex analyticity in all variables is equivalent to separate analyticity in each variable.

Since $\phi(\cdot, \alpha)$ decays exponentially at the rate $e^{-(\alpha^2-\epsilon) |x|}$, by the Agmon bound, we can extend $w$ to an analytic function for all $\Gamma$ and on the strip $\{v \mid |\Im v| < \alpha^2\}$. This proves analyticity in regard to $\Gamma$ and $v$.

Analyticity in $D$ and $\alpha$ requires that $\phi(x, 1)$ is an analytic function in $x$, in the sense that there exist constants $C_1$ and $C_2$ such that for every multiindex $\beta = (\beta_{v_1}, \beta_{v_2}, \beta_{v_3})$
$$
\|\partial^{\beta} \phi\|_2 \leq C_1 |\beta|! C_2^{|\beta|}.
$$
$\phi$ is exponentially decaying, $|\phi(x)| \leq C e^{-(1-\epsilon)|x|}$, and satisfies the equation
\be\lb{eqg}
-\Delta \phi + \phi = g(\phi).
\ee
By taking $g(\phi) = \phi^3$, we retrieve equation (\ref{phi}). However, we make the more general assumption that $g$ is analytic and its derivatives grow subexponentially, meaning that, for any $C_4>0$,
$|g^{(j)}(0)| \les C_4^j$.

It is enough to impose this condition on $g$ itself, since its derivatives $g^{(j)}$ will then satisfy the same bound by Cauchy's formula.

Then, for any $C_4>0$ and fixed $f$ in the Sobolev space $H^2$, $\|g^{(j)} \circ f\|_{H^2} \leq C_1 C_4^j$.

This characterizes not only polynomials, such as $g(z) = z^3$, but also functions of subexponential growth such as $g(z)=\cosh(\sqrt z)$.

Applying $\partial^\beta$ to (\ref{eqg}), we obtain
$$\begin{aligned}
(-\Delta+1) \partial^{\beta} \phi &= \partial^{\beta} (g \circ \phi) \\
&= \sum_{j=1}^{|\beta|} \sum_{\beta_1 + \ldots + \beta_j = \beta} \partial^{\beta_1} \phi \cdot \ldots \cdot \partial^{\beta_j} \phi \cdot (g^{(j)} \circ \phi).
\end{aligned}$$

We show by induction on $|\beta|$ that, in the Banach algebra $H^2$, for a suitable $C_2$
$$
\|\partial^{\beta} \phi\|_{H^2} \leq \frac{C_3 C_2^{|\beta|}}{|\beta|^2+1}.
$$

Assuming that this induction hypothesis holds for all indices up to $|\beta|$, we obtain that
$$\begin{aligned}
\|\partial^{\beta} \phi\|_{H^4} & \les \|(-\Delta+1)(\partial^{\beta} \phi)\|_{H^2} \\
& \les C_3 C_2^{|\beta|} \sum_{j=1}^{|\beta|} \sum_{\beta_1 + \ldots + \beta_j = \beta} \frac{C_4^j C_1^j}{(|\beta_1|^2 + 1) \cdot \ldots \cdot (|\beta_j|^2 + 1)}.
\end{aligned}$$
Note that
$$
\sum_{k=0}^{n} \frac 1 {(k^2 +1)((n-k)^2 + 1)} < \frac 6 {n^2 + 1}.
$$
By induction, we obtain that
$$
\sum_{n_1 + \ldots + n_j = n} \frac 1 {(n_1^2 + 1) \cdot \ldots \cdot (n_j^2 + 1)} \leq \frac {6^{j-1}} {n^2 + 1}. 
$$
Therefore
$$
\|\partial^{\beta} \phi\|_{H^4} \les \frac {C_3 C_2^{|\beta|}}{|\beta|^2 + 1} \sum_{j=1}^{|\beta|} (6 C_4 C_1)^j.
$$
By making $C_2$ and $C_3$ sufficiently large and $C_4$ sufficiently small, we obtain that the sum is uniformly bounded for all $|\beta|$ and
$$
\|\partial^{\beta} \phi\|_{H^4} \les \frac {C_2^{|\beta|}}{|\beta|^2 + 1} \leq \frac {C_3 C_2^{|\beta|+1}} {|\beta|^2+1}.
$$
This proves that the induction hypothesis holds for derivatives of order $|\beta|+1$.

This proof shows that $D \mapsto e^{D\dl} \phi$ is an analytic $H^2$-valued map. More generally, the same argument works for any Banach algebra $A$ containing $\phi$ and with the property that
$$
\|\dl f\|_A \les \|(-\Delta + 1) f\|_A.
$$
Thus, it suffices that $A$ is invariant under convolution with the kernel
$$
\frac {x_k e^{-|x|}}{|x|^3} + \frac {x_k e^{-|x|}}{|x|^2}.
$$
For $\epsilon>0$, the Banach algebra $A := e^{-(1-\epsilon)|x|}L^{\infty}$ fulfills this condition, hence $D \mapsto e^{D\dl} \phi$ is an $A$-valued analytic map.

This ensures the joint analyticity of $w(p)$ given by (\ref{2.170}) in the variables $D$, $v$, and $\Gamma$ around the point $\phi = w(p=(1, 0, 0, 0))$. By symmetry transformations, this implies analyticity around any other point.

$\partial_\alpha \phi$ is given by the generator of dilations, which is a combination of multiplication by $x$ and taking the gradient $\dl$, applied to $\phi$:
$$
\partial_{\alpha} \phi = (\alpha^{-1} + x \dl) \phi.
$$
Therefore, analyticity with respect to $\alpha$ follows from that with respect to translations and boost, i.e.\ $D$ and $v$.

When $w(p)$ is analytic, so are its derivatives, by definition. Hence $w(p)$ is analytic in any Schwartz seminorm $\mc S_n$.
%

The analyticity of $F^+(W(p))$ as a function of $p$ reduces to that of a fixed eigenfunction $F^+(W(p_0))$ with respect to the symmetry transformations. $F^+ = \bpm f^+ \\ \ov f^+ \epm$ (see Section \ref{spectru}) satisfies the equation
$$
(\Delta - 1 + 2 \phi^2) f^+ + \phi^2 \ov f^+ = i \sigma f^+.
$$
Using the fact that $\phi$ is analytic, the proof proceeds in the same manner: one shows by induction that
$$
\|\partial^{\beta} f^+\|_{e^{-(1-\epsilon)|x|}L^{\infty}} \leq \frac {C_3 C_2^{|\beta|}}{|\beta|^2 + 1}.
$$
\end{proof}

Lemma \ref{lemma_11} has the following consequence:
\begin{proposition} $\mc F_{W_0}: \mc N_{lin}(W_0) \to \mc N$ given by \eqref{2.105} is locally one-to-one and Lipschiz. $\mc F_{W_0}^{-1}: \mc N \to \mc N_{lin}$ defined on $\Ran \mc F$ is also locally Lipschitz.
\lb{prop_9}
\end{proposition}
\begin{proof} $\mc F$ is Lipschitz because $h(R, W)$ is so too, by Proposition \ref{proposition_10}, and because
$$
\mc F_{W_0}(R) = \tilde F(W_0, R, h(R, W_0)).
$$
The local invertibility of $\mc F_{W_0}$ is due to Lemma \ref{lemma_11}: for a sufficiently small $\delta_0$, $h(R, W_0)$ is close to zero and the previous lemma applies. To establish the Lipschitz property of $\mc F^{-1}$, we discard the parameter $h$.
\end{proof}
Thus, $\mc N$ is a Lipschitz manifold and its embedding into $\dot H^{1/2}$ is also locally Lipschitz. If $h$ is $C^n$, smooth, or analytic, the same is true for $\mc N$ and its embedding into $\dot H^{1/2}$.

Lemma \ref{lemma_11} also implies that, if $\Psi \in \dot H^{1/2}$ is sufficiently close to the manifold of solitons $\Sol$, there exists a unique soliton nearest to $\Psi$ in the $\dot H^{1/2}$ norm, which we call its projection on $\Sol$. This notion should be taken with a grain of salt, because there are several ways to define a natural $\dot H^{1/2}$ norm for (\ref{NLS}).

We recall that $R$ and all other $\C^2$-valued functions we employ have the form $R = \bpm r \\ \ov r \epm$, so the dot product is real-valued.

\begin{lemma} For every $W \in \Sol$ there exists $\delta > 0$ such that, for $\Psi \in \dot H^{1/2}$, whenever $\|\Psi-W\|_{\dot H^{1/2}} < \delta$ there exists $W_1$ such that $P_0(W_1)(\Psi-W_1) = 0$ and
\be\lb{eq_lemma_10}
\|\Psi - W_1\|_{\dot H^{1/2}} \leq C \|\Psi-W\|_{\dot H^{1/2}}.
\ee
Furthermore, $W_1$ depends Lipschitz continuously on $\Psi \in \dot H^{1/2}$.
\lb{lemma_10}
\end{lemma}
Due to symmetry transformations, $\delta$ can be chosen independently of $W$.

This also holds for $\dot W^{1/2, 6}$, with the same proof, but we do not use it in the sequel.

On a Hilbert space, this statement is straightforward when $P_0$ is an orthogonal projection and the constant can then be taken to be $1$. However, this proof does not use orthogonality.
\begin{proof}
If $\delta$ is sufficiently small, Lemma \ref{lemma_2.10} implies that $\Psi = \tilde F(W_1, R, h)$ for some $(W_1, R, h)$ close to $(W, 0, 0)$: since $\tilde F^{-1}$ is bounded,
\be
\|W-W_1\|_{\dot H^{1/2}} + \|R\|_{\dot H^{1/2}} + |h| \les \|\Psi-W\|_{\dot H^{1/2}} \les \delta.
\ee
Since $R \in \mc N_{lin}(W_1)$, by definition $P_0(W_1)(\Psi-W_1) = 0$.

To a first order $W-W_1$ lies in a direction tangent to $\Sol$, i.e.\ within the range of $P_0(W_1)$, so by the Taylor expansion
$$
\|(I-P_0(W_1))(W-W_1)\|_{\dot H^{1/2}} \les \|W-W_1\|_{\dot H^{1/2}}^2.
$$
Thus
\be\begin{aligned}
\|\Psi - W_1\|_{\dot H^{1/2}} & = \|(I-P_0(W_1))(\Psi-W_1)\|_{\dot H^{1/2}} \\
& \leq C \|\Psi-W\|_{\dot H^{1/2}} + \|(I-P_0(W_1))(W-W_1)\|_{\dot H^{1/2}} \\
& \les \|\Psi-W\|_{\dot H^{1/2}} + \|W-W_1\|_{\dot H^{1/2}}^2.
\end{aligned}\ee
Moreover,
\be\begin{aligned}
\|W-W_1\|_{\dot H^{1/2}}^2 & \leq (\|\Psi-W\|_{\dot H^{1/2}} + \|\Psi - W_1\|_{\dot H^{1/2}})^2 \\
& \les \|\Psi-W\|_{\dot H^{1/2}} + \delta \|\Psi - W_1\|_{\dot H^{1/2}}.
\end{aligned}\ee
For sufficiently small $\delta$ we retrieve (\ref{eq_lemma_10}).

The proof of continuous dependence is completely analogous.
\end{proof}

The solutions $\psi$ constructed by Proposition \ref{prop27} arise by a very specific and rather technical construction. We next give a more general definition of asymptotically stable solutions and prove that all such solutions can be retrieved by means of Proposition \ref{prop27}.

\begin{definition} Fix a small parameter $\delta_1>0$. $\psi(x, t)$ is a \emph{small asymptotically stable} solution to (\ref{NLS}) if $\sup_{t \geq 0} d_{\dot H^{1/2}}(\psi(t), \Sol) < \delta_1$ and $\sup_{t \geq 0} \inf_{\Gamma, D} \|\psi(t) - e^{i\Gamma+D\dl} \psi(0)\|_{\dot H^{1/2}}<~2 \|\phi\|_{\dot H^{1/2}}$.
\lb{definition_4}\end{definition}
Thus, solutions in this class live near the soliton manifold $\sol$ for all $t \geq 0$ and travel at most a finite, but not necessarily small, distance in the parameter space in the $v$ and $\alpha$ directions. Here $\phi$ is the solution to (\ref{phi}).








Every solution with initial data in $\mc N$ is small and asymptotically stable. A partial converse is also true: we show that Definition \ref{definition_4} characterizes the set $\mc N$.
\begin{proposition}\lb{prop_2.13}
There exists $\delta_1>0$ such that, if $\psi(0)$ is the initial value of a small asymptotically stable solution $\psi$ of \eqref{NLS} that satisfies Definition \ref{definition_4}, then $\psi(0) \in \mc N$.
\end{proposition}
This remains true after replacing $0$ with any $t \geq 0$.
\begin{proof} Write $\Psi = W_{\pi}(t) + R(t)$, $\pi(0) = p_0$, satisfying the orthogonality condition
\be\lb{orto}
P_0(W_{\pi}(t)) R(t) = 0.
\ee
By Lemma \ref{lemma_10}, there exists $\pi(t)$ such that $P_0(W_\pi(t)) R(t) = 0$. This is equivalent to the orthogonality condition (\ref{orto}).

Let $\pi(t)=(\alpha(t), \Gamma(t), v_k(t), D_k(t))$. Definition \ref{definition_4} implies that for all $t \geq 0$
$$
\|w(\pi(t)) - w(\pi(0))\|_{\dot H^{1/2}} \les 2\delta_1 + \sup_{t \geq 0} \|\psi(t) - \psi(0)\|_{\dot H^{1/2}} < 2 \|\phi\|_{\dot H^{1/2}}.
$$
When $\alpha_1>>\alpha_2$, $\|\phi(\cdot, \alpha_1) -\phi(\cdot, \alpha_2)\| \to 2 \|\psi\|_{\dot H^{1/2}}$. Therefore, for all $t \geq 0$ $\alpha(t)$ belongs to a fixed compact set.

Due to the orthogonality condition (\ref{orto}), the conservation of momentum implies that
$$
P[w_{\pi}(t)] + P[r(t)] = P[w_{\pi}(0)] + P[r(0)].
$$
Thus $\sup_{t \geq 0} \big|P[w_{\pi}(t)] - P[w_{\pi}(0)]\big| < 2\delta_1$. Therefore $v(t)$ belongs to a fixed compact set for all $t$.

Thus, $\sigma(W(\pi(t))$ belongs to a compact set in $(0, \infty)$ and the solitons $W(\pi(t))$ are uniformly bounded in each Schwartz seminorm.

(\ref{orto}) translates into modulation equations for $\pi$ of the form (\ref{mod}). Then $(R, \pi)$ satisfy the equation system, akin to (\ref{ec_liniara}),
$$\begin{aligned}
&i \partial_t R + H_{\pi}(t) R = F,\ F= -i L_{\pi} R + N(R, W_{\pi}) - N_{\pi}(R, W_{\pi})\\
&\dot f = 2\alpha^2 \|\phi\|_{\dot H^{1/2}}^{-2} \big(\langle R, (d_{\pi} \partial^*_fW_\pi)\, \dot \pi \rangle - i\langle N(R, W_\pi), \partial^*_fW_\pi\rangle\big), f \in \{\alpha, \Gamma\}\\
&\dot f = \alpha \|\phi\|_{\dot H^{1/2}}^{-2} \big(\langle R, (d_{\pi} \partial^*_fW_\pi)\, \dot \pi \rangle - i\langle N(R, W_\pi), \partial^*_fW_\pi\rangle\big), f \in \{v_k, D_k\}.
\end{aligned}$$
By the second and the third equation, we infer that
$$
\|\dot \pi\|_{L^{\infty}} \les \|R\|_{L^{\infty}_t \dot H^{1/2}_x} \|\dot \pi\|_{L^{\infty}} + \|R\|_{L^{\infty}_t \dot H^{1/2}_x}^2 + \|R\|_{L^{\infty}_t \dot H^{1/2}_x}^3.
$$
Hence, for small $\delta_1>0$, $\|\dot \pi\|_{L^{\infty}_t}$ is finite:
$$
\|\dot \pi\|_{L^{\infty}} \les \delta^2 + \delta^3.
$$
Due to (\ref{orto}), $R$ has no null eigenspace component. Thus $R$ decomposes into two components, $R(t) = P_c(t) R(t) + P_{im}(t) R(t)$. Since $\sigma(W_\pi(t)) = \sigma(W(\pi(t))$ is bounded from below, Lemma \ref{hyp} applies. Then for
$$\begin{aligned}
h_0 F^+(W_\pi(0)) &= P_+(W_\pi(0)) R(0),\\
h^+(t) F^+(W(\pi(t)) &= P_+(W_\pi(t)) R(t),\\
h^-(t) F^-(W(\pi(t)) &= P_-(W_\pi(t)) R(t),
\end{aligned}$$
we have that $h_0=h^+(0)$ and
$$\begin{aligned}
h^+(t) &= -\int_t^{-\infty} e^{\int_s^t \sigma(\tau) \dd \tau} N^+(s) \dd t,\\
h^-(t) &= e^{-\int_0^t \sigma(\tau) \dd \tau} h^-(0) + \int_0^t e^{-\int_s^t \sigma(\tau) \dd \tau} N^-(s) \dd t.
\end{aligned}$$
Then, both the initial data
$$
\Psi(0) = W(p_0) + R_0 + h_0 F^+(W_0)
$$
and, by Proposition \ref{prop27}, the initial data
$$
\tilde \Psi(0) = W_0 + R_0 + h(R_0, W_0) F^+(W_0)
$$
give rise to the small asymptotically stable solutions $\Psi$, $\tilde \Psi \in X$,
$$
\Psi(t) = W_{\pi}(t) + R(t),\ \tilde \Psi(t) = W_{\tilde \pi}(t) + \tilde R(t).
$$
By the comparison Lemma \ref{lemma_5}, since $W_0$ and $R_0$ are the same for both solutions,
\be
\|(R, \pi) - (\tilde R, \tilde \pi)\|_{e^{t\rho} L^{\infty}_t \dot H^{1/2}_x \times \partial_t^{-1}  e^{t\rho} L^{\infty}_t} + |h_0 - h(R_0, W_0)| \les \delta \|(R, \pi) - (\tilde R, \tilde \pi)\|_{e^{t\rho} L^{\infty}_t \dot H^{1/2}_x \times \partial_t^{-1}  e^{t\rho} L^{\infty}_t}.
\ee
For sufficiently small $\delta$, we obtain that $(R, \pi) = (\tilde R, \tilde \pi)$, hence $h_0 = h(R_0, W_0)$. Therefore
$$
\Psi(0) = \tilde \Psi(0) = W_0 + R_0 + h(R_0, W_0) F^+(W_0) \in \mc N.
$$
\end{proof}

\begin{corollary} If $\Psi$ is a solution to \eqref{NLS} whose initial data $\Psi(0)$ belongs to $\mc N$, then $\Psi(t)$ also belongs to $\mc N$ for all $t>0$ and for sufficiently small $t<0$.
\lb{stable}
\end{corollary}
\begin{proof}
Both for $t>0$ and for sufficiently small $t<0$, $\psi(t)$ exists by Proposition \ref{prop27} for $t>0$ and due to the local existence theory for small $t<0$.

$\psi(t)$ gives rise to a small asymptotically stable solution that conforms to Definition \ref{definition_4} --- that solution being $\psi$ itself.

Proposition \ref{prop_2.13} then shows that $\psi(t) \in \mc N$.
\end{proof}

Five maps describe asymptotically stable solutions to (\ref{NLS}) on $t \in [0, \infty)$:
\begin{align}
\nonumber h(W_0, R_0)&: \mc N_{lin}(W_0) \to \R, \\
\nonumber \mc F_{W_0}(R_0)&: \mc N_{lin}(W_0) \to \mc N,\ \mc F(R_0, W_0) = \tilde F(W_0, R_0, h(R_0, W_0)) := W_0 + R_0 + h(R_0, W_0) F^+(W_0),
\intertext{and the actual solution map}
\nonumber \Psi(R_0, W_0)&: \mc N_{lin}(W_0) \to e^{t\rho} L^{\infty}_t \dot H^{1/2}_x \cap e^{t\rho} L^2_t \dot W^{1/2, 6/5}_x := W_{\pi(R_0, W_0)} + R(R_0, W_0), \\
\nonumber \pi(R_0, W_0)&: \mc N_{lin}(W_0) \to e^{t\rho} \dot W^{1, 1}_t, \\
\nonumber R(R_0, W_0)&: \mc N_{lin}(W_0) \to e^{t\rho} L^{\infty}_t \dot H^{1/2}_x \cap e^{t\rho} L^2_t \dot W^{1/2, 6/5}_x.
\end{align}

Note that $\mc F_{W_0}(R_0) = \Psi(W_0, R_0)(0)$ and $\Psi(W_0, R_0) = R(W_0, R_0) + W_{\pi(W_0, R_0)}$.

We prove that all five maps are real analytic, hence $h \circ \mc F^{-1}$, $\Psi \circ \mc F^{-1}$, $\pi \circ \mc F^{-1}$, and $R \circ \mc F^{-1}$ are also real-analytic. This is equivalent to changing the dependent variable from $R_0 \in \mc N_{lin}(W_0)$ to $\Psi(0) = W_0 + R_0 + h(W_0, R_0) \in \mc N(W_0)$.

Hence the solution $\Psi(t)$, its dispersive part $R(t)$, and the moving soliton $W_{\pi}(t)$ will depend analytically on $\Psi(0) \in \mc N(W_0)$.

The dependence of these maps on $W_0$, across the fibers of $\mc N$, is explicitly given by symmetry transformations, which are analytic in the $H^{1/2}$ setting.

Analyticity across the fibers of $\mc N$ (i.e.\ as a function of $W_0$) is then easy to obtain in $H^{1/2}$, but becomes more delicate in $\dot H^{1/2}$, because boost transformations $\psi \mapsto e^{ivx} \psi$ are not $\dot H^{1/2}$-bounded.

\begin{proposition}\lb{prop_an} For fixed $W_0 \in \Sol$ there exists $\rho>0$ such that $h(W_0, R_0)$, $\mc F_{W_0}(R_0) \in \dot H^{1/2}_x$, $(R, \pi)$, and $\Psi(R_0, W_0)$ are real analytic functions of 
$\Psi(0) \in \mc N(W_0)$.
\end{proposition}

Analyticity of $\Psi(R_0, W_0)$ is closely tied to that of the nonlinearity in (\ref{NLS}) and (\ref{NLS'}). If the nonlinearity were of class $C^k$, then the manifold $\mc N$, its embedding into $\dot H^{1/2}$, and $\Psi(R_0, W_0)$ would also be of class $C^k$. The cubic nonlinearity of (\ref{NLS}) is analytic.

\begin{proof}[Proof of Proposition \ref{prop_an}] 
Proving the analyticity of $h$, $\mc F$, $(R, \pi)$, and $\Psi$ on $\mc N$ reduces to showing that $(R, \pi)$ and $h$ are analytic as functions of $R(0) \in \mc N$, in the local coordinates determined by the chart $\mc F_{W_0}$ (\ref{local_chart}) --- thus as functions of $R_0$.

Let $W_{R_0} := \exp_{W_0} P_0(R_0)$. Recall that $\mc N$ is locally the image of the set $\{W_0 + R_0 \mid P_+(W_{R_0}) R_0 = 0,\ \|R_0\|_{\dot H^{1/2}} \leq \delta_0 << 1\}$ under the analytic map
$$
\mc F_{W_0}(R_0) := W_{R_0} + (I - P_0(W_{R_0}))R_0 + h\big(\big(P_c(W_{R_0}) + P_-(W_{R_0})\big)R_0, W_{R_0}\big) F^+(W_{R_0}).
$$
Take initial data for (\ref{NLS}) of the form $\Psi(0) = W_{\pi}(0) + R(0)$, where
$$
\pi(0) = p_0 + P_0(W_{R_0}) R_0,\ R(0)=\big(P_c(W_{R_0}) + P_-(W_{R_0})\big)R_0.
$$
$\Psi(0) \in\mc N$ is an analytic function of $R_0$: $\Psi(0) = \sum_{n=0}^{\infty} \Psi^n(0)$, where $\Psi^0(0) = W_0$, $\Psi^1(0) = R_0$, and
$$
\Psi^n(0) = d^n \big(W_{R_0} - P_0(W_{R_0}) R_0\big) + \sum_{m=0}^n h^m(R_0) (F^+)^{n-m}(W_{R_0}).
$$
Thus $\|\Psi^n(0)\| \les C_1^n \|R_0\|_{\dot H^{1/2}}^n$. Because $W$ and $F^{\pm}(W)$ are analytic, $R(0)$ and $\pi(0)$ are also analytic functions of $R_0$:
$$
R(0) = \sum_{m=0}^{\infty} R^m(0),\ \pi(0) = \sum_{m=0}^{\infty} \pi^m(0),
$$
where
\be\lb{cond_init}
|\pi^n(0)| + \|R^n(0)\|_{\dot H^{1/2}} \les \|R_0\|_{\dot H^{1/2}}^n.
\ee
In particular
$$
\pi^0(0) = p_0,\ \pi^1(0) = P_0(W_0) R_0,\ R^0(0) \equiv 0,\ R^1(0) = \big(P_c(W_0) + P_-(W_0)\big) R_0.
$$



Denote the solution to (\ref{NLS}) provided by Proposition \ref{prop27} for the initial value $\Psi(0)$ by $\Psi(t) := W_{\pi}(t) + R(t)$. $R(t)$ and $\pi(t)$ satisfy a system of the same form as (\ref{ec_liniara}):
\be\begin{aligned}\lb{ec_r0}
&i \partial_t R + H_{\pi}(t) R = F(\pi, R), \\
&\dot f = F_f(\pi, R),\ f \in \{\alpha, \Gamma, v_k, D_k\}.
\end{aligned}\ee
Here
$$\begin{aligned}
H_{\pi}(t) &= \Delta \sigma_3 + V_{\pi}(t),\\
F(\pi, R) &= -i L_{\pi} R + N(R, W_{\pi}) - N_{\pi}(R, W_{\pi})
\end{aligned}$$
and we introduced the notation
$$
F_f(\pi, R) := \left\{\begin{aligned}
4\alpha \|W_{\pi}\|_2^{-2} &(\langle R, d_{\pi} \partial^*_fW_{\pi}\, \dot \pi \rangle \\
&- i\langle N(R, W_{\pi}), \partial^*_fW_{\pi}\rangle),\ f \in \{\alpha, \Gamma\}\\
2\|W_{\pi}\|_2^{-2} &(\langle R, d_{\pi} \partial^*_fW_{\pi}\, \dot \pi \rangle \\
&- i\langle N(R, W_{\pi}), \partial^*_fW_{\pi}\rangle),\ f \in \{v_k, D_k\}.
\end{aligned}\right.
$$

We then find an asymptotic expansion around $0$ of the form
\be\lb{expansion}
R(R_0) \sim \sum_{n=0}^{\infty} R^n(R_0),\ \pi(R_0) \sim \sum_{n=0}^{\infty} \pi^n(R_0),\ h(R_0) \sim \sum_{n=0}^{\infty} h^n(R_0).
\ee
$R^n = \frac 1 {n!}\, d^n R(R_0)$, $\pi^n = \frac 1 {n!}\, d^n \pi(R_0)$, and $h^n = \frac 1 {n!}\, d^n h(R_0)$ are $n$-linear functions of $R_0$ and
\be\begin{aligned}\lb{analitic}
\Big\|R(R_0) - \sum_{m=0}^n R^m(R_0)\Big\|_{e^{t\rho}L^{\infty}_t \dot H^{1/2}_x \cap e^{t\rho} L^2_t \dot W^{1/2, 6}_x} &\leq C_n \|R_0\|_{\dot H^{1/2}}^{n+1}, \\
\Big\|\dot \pi(R_0) - \sum_{m=0}^n \dot \pi^m(R_0)\Big\|_{e^{t\rho} L^1_t} &\leq C_n \|R_0\|_{\dot H^{1/2}}^{n+1}, \\
\Big|h(R_0) - \sum_{m=0}^n h^m(R_0)\Big| &\leq C_n \|R_0\|_{\dot H^{1/2}}^{n+1}.
\end{aligned}\ee
We define $R^n$, $\pi^n$, and $h^n$ by identifying the $n$-th-order terms --- which will involve $R^m$, $\pi^m$, and $h^m$ for $1 \leq m <n$ ---  in the equation (\ref{ec_r0}), within the initial data $R(0)$ and within the nonlinearities $F(\pi, R)$ and $F_f(\pi, R)$. We then solve (\ref{ec_r0}) with this input and show that the remainders have order $n+1$ in $R_0$.

Finally, we prove that $h^n$, $(R^n, \pi^n)$, and the remainders in (\ref{expansion}) grow at most exponentially with $n$, hence that $h$ and $(R, \pi)$ equal the sum of their Taylor series inside a positive convergence radius.
In fact, $R$, $\pi$, and $h$ are analytic in the ball $B(W_0, C_1^{-1}) \cap \mc N \subset \dot H^{1/2}$, when $C_n \les C_1^n$ uniformly for all $n$ in (\ref{analitic}).

The constant zeroth-order terms in the expansion are
$$
R^0(t) = 0,\ \pi^0(t) \equiv p_0 \in \set R^8,\ h^0=0.
$$

We then find the first-order differentials $(dR, d\pi)$ and $dh$ and check that
$$
\|(R, \pi) - (dR, \pi^0 + d\pi)\|_Y \les \|R_0\|_{\dot H^{1/2}}^2;\ |h - dh| \les \|R_0\|_{\dot H^{1/2}}^2.
$$


Equation (\ref{ec_r0}) makes the orthogonality condition
\be\lb{2.207}
\langle R(t), \partial^*_fW_{\pi}(t) \rangle = 0
\ee
valid for all $t$. Indeed, since (\ref{2.207}) holds for $t=0$, (\ref{ec_r0}) ensures that (\ref{2.207}) holds for all $t$.

Define the first-order terms $R^1$ and $\pi^1$ as solutions of equation (\ref{ec_r1}), obtained by linearizing (\ref{ec_r0}) around $(R^0, \pi^0)$:
\be\lb{ec_r1}\begin{aligned}
&i \partial_t R^1 + H_{\pi^0}(t) R^1 = \partial_{\pi} F(\pi^0, R^0)\, \pi^1 + \partial_R F(\pi^0, R^0)\, R^1 - (\partial_{\pi} V_{\pi^0}(t)\, \dot \pi^1)\, R^0 \\
&\dot {f^1} = \partial_{\pi} F_f(\pi^0, R^0)\, \dot \pi^1 +  \partial_R F_f(\pi^0, R^0)\, R^1,\ f \in \{\alpha, \Gamma, v_k, D_k\}.
\end{aligned}\ee
For initial data, take
\be
R^1(0) := \big(P_c(W_0) + P_-(W_0)\big) R_0 + h^1 F^+(W_0),\ \pi^1(0) := P_0(W_0) R_0.
\ee
Note that, because $R^0(t) \equiv 0$, $\dot \pi^0(t) \equiv 0$, the following terms cancel in (\ref{ec_r1}):
$$\begin{aligned}
&(\partial_{\pi} V_{\pi^0}(t)\, \dot \pi^1)\, R^0 = 0,\ \partial_{\pi} F_f(\pi^0, R^0) = 0,\ \partial_R F_f(\pi^0, R^0) = 0, \\
&\partial_{\pi} F(\pi^0, R^0) = 0,\ \partial_R F(\pi^0, R^0) = 0.
\end{aligned}$$
This improvement only holds when taking the differential at $0$ and is due to the nonlinearity being of order higher than two, i.e.\ cubical in (\ref{NLS}).

By (\ref{ec_r1}) then $\dot \pi^1(t) \equiv 0$ and
\be\begin{aligned}\lb{2.213}
&i \partial_t R^1 + H_{\pi^0}(t) R^1 = 0.
\end{aligned}\ee
The orthogonality condition $\langle R, \partial^*_fW_{\pi} \rangle = 0$ becomes to a first order
$$
\langle R^1, \partial^*_fW_{\pi^0} \rangle + \langle R^0, \partial_{\pi} \partial^*_fW_{\pi^0}\, \pi^1 \rangle = 0.
$$
Consequently, as $R^0 \equiv 0$, the orthogonality condition also holds for $R^1$:
\be\lb{2.216}
\langle R^1, \partial^*_fW_{\pi^0} \rangle = 0.
\ee
By endpoint Strichartz estimates and Lemma \ref{hyp}, we obtain that $R^1$ is bounded for a unique value of $h^1$ and then it satisfies
$$
\|R^1\|_{L^{\infty}_t \dot H^{1/2}_x \cap L^2_t \dot W^{1/2, 6}_x} \les \|R_0\|_{\dot H^{1/2}}.
$$
Note that $\|R^1\|_{L^{\infty}_t \dot H^{1/2}_x \cap L^2_t \dot W^{1/2, 6}_x}$ can be made arbitrarily small and $\|\dot \pi^1\|_{L^1_t}=0$.

To show that $R^1(t) = dR$ and $\pi^1(t) = d\pi$, for $R^1(t)$ given by (\ref{2.213}) and $\pi^1(t) \equiv p_0$, consider
$$
S^1 = R^0 + R^1,\ \Sigma^1 = \pi^0 + \pi^1.
$$
The fact that $F(\pi, R)$, $V_{\pi}$, and $F_f(\pi, R)$ are differentiable means that
\be\begin{aligned}\lb{2.221}
F(\Sigma^1, S^1) &= F(R^0, \pi^0) + \partial_{\pi} F(\pi^0, R^0)\, \pi^1 + \partial_R F(\pi^0, R^0)\, R^1 + o_1(R^1, \pi^1), \\
V_{\Sigma^1} &= V_{\pi^0} + \partial_{\pi} V_{\pi^0}\, \pi^1 + o_2(\pi^1), \\
F_f(\Sigma^1, S^1) &= F_f(\pi^0, R^0) + \partial_{\pi} F_f(\pi^0, R^0)\, \pi^1 + \partial_R F_f(\pi^0, R^0)\, R^1 + o_3(R^1, \pi^1).
\end{aligned}\ee
In fact, all three quantities are analytic, being explicitly given by analytic functions.

The main issue is that the soliton $W_{\pi}(t)$ depends not just on $\pi(t)$ and $\dot \pi(t)$, but also on $\int_0^{t} \pi(s) \dd s$. Thus, 
each derivative increments the power of growth in $t$.

Due to analyticity, the error terms in (\ref{2.221}) are quadratic:
$$\begin{aligned}
 \|o_1(R^1, \pi^1)\|_{\langle t \rangle L^2_t W^{1/2, 6/5}_x} &\les \|(R^1, \pi^1)\|_X^2 \\
 \|o_2(\pi^1)\|_{\langle t \rangle L^{\infty}_t \dot W^{1/2, 6/5}_x \cap \langle t \rangle L^{\infty}_t L^{3/2-\epsilon}_x} &\les \|(R^1, \pi^1)\|_X^2 \\
\|o_3(R^1, \pi^1)\|_{\langle t \rangle L^1_t} &\les \|(R^1, \pi^1)\|_X^2.
\end{aligned}$$
Therefore
\be\begin{aligned}\lb{2.225}
&i \partial_t S^1 + H_{\Sigma^1}(t) S^1 = F(\Sigma^1, S^1) + O^2(R^1, \pi^1) \\
&\dot f_{\Sigma^1} = F_f(\Sigma^1, S^1) + O_f^2(R^1, \pi^1),\ f \in \{\alpha, \Gamma, v_k, D_k\},
\end{aligned}\ee
where $O^2$, $O_f^2$ are bounded by
\be
\|O^2(R^1, \pi^1)\|_{\langle t \rangle L^2_t W^{1/2, 6/5}_x} + \|O_f^2(R^1, \pi^1)\|_{\langle t \rangle L^1_t} \les \|(R^1, \pi^1)\|_X^2.
\ee
By (\ref{2.216}), $S^1$ fulfills the approximate orthogonality relation
$$
\|\langle S^1(t), \partial^*_fW_{\Sigma^1}(t) \rangle\|_{\langle t \rangle L^{\infty}_t} \les \|(R^1, \pi^1)\|_X^2.
$$
Comparing the equation system (\ref{2.225}) giving $S^1$ and $\Sigma^1$ to (\ref{ec_r0}) satisfied by $R$ and $\pi$, as in Lemma \ref{lemma_5} we obtain that
$$
\|(R, \pi) - (S^1, \Sigma^1)\|_Y \les \|(R^1, \pi^1)\|_X^2 \les  \|R_0\|_{\dot H^{1/2}}^2.
$$

{\bfseries Higher-order terms.} We repeat this process for the higher-order terms in the expansion of $R$, $\pi$, and $h$. 

Let $n \geq 1$ and consider a variation $\pi(t) = \pi^0 + \delta \pi(t)$ of $\pi^0$. By Lemma \ref{lemma_2.10}, solitons are analytic functions of their parameters: in any Schwartz seminorm $\mc S_n$,
\be
\|d^n_{\pi} W(\pi^0)\, \delta \pi\|_{\mc S_n} \les C_2^n n! (\|\delta \dot \pi\|_{L^1_t} + |\delta \pi(0)|)^n.
\ee
By (\ref{analitic_sol}) only derivatives in $\alpha$ and $v_k$ increase factorially:
$$
\|\partial^{\beta_{\alpha}}_{\alpha} \partial^{\beta_{\Gamma}}_{\Gamma} \partial^{\beta_{v_k}}_{v_k} \partial^{\beta_{D_k}}_{D_k} w(p)\|_{\mc S_n} \les C_2^{|\beta|} \prod_{f \in \{\alpha, v_k\}} \beta_f!.
$$
Recall that $W_{\pi^0}$ given by (\ref{1.1}) and (\ref{w0}) has the form
$$\begin{aligned}
W_{\pi^0}(t) &= \bpm w_{\pi^0}(t) \\ \ov w_{\pi^0}(t) \epm,\\
w_{\pi^0}(t) &= w\Big(\alpha^0(t), \Gamma^0(t) + \int_0^t ((\alpha^0(s))^2-|v^0(s)|^2) \dd s, v^0(t), D(t) + 2\int_0^t v^0(s) \dd s \Big)
\end{aligned}$$
We obtain for $W_{\pi^0}$, $\partial_f W_{\pi^0}$, $F^+(W_{\pi^0})$, and all other quantities that depend on $W_{\pi^0}$ the following explicit expression of analyticity: for any multiindex $\beta = (\beta_{\alpha}, \beta_{\Gamma}, \beta_{v_k}, \beta_{D_k})$
\be\lb{2.229}
\|\partial^{\beta_{\alpha}}_{\alpha} \partial^{\beta_{\Gamma}}_{\Gamma} \partial^{\beta_{v_k}}_{v_k} \partial ^{D_k}_{D_k} W_{\pi^0}(t)\, \delta \pi\|_{\mc S_n} \leq C_1 C_2^{|\beta|} \prod_{f \in \{\alpha, v_k\}} \beta_f! \prod_{f \in \{\Gamma, D_k\}} \langle t \rangle^{\beta_f} (\|\delta \dot \pi\|_1 + |\delta \pi(0)|)^{|\beta|}.
\ee
While in $\alpha$ and $v_k$ $W_{\pi^0}(t)$ is analytic on a strip, in $\Gamma$ and $D_k$ $W_{\pi^0}(t)$ extends to an analytic function of exponential type on the whole complex plane. This stronger property is necessary in the sequel.

Pick $a$ and $a_1$ such that $a< a_1 \leq\sigma(\alpha(t))$, for any time $t$, where $\alpha(t)$ is the scaling component of $\pi(t)$ and $\pi$ is any path admissible in this proof.

Recall that the weights $A_n(t)$ given by (\ref{2.230}) have the form
$$
A_n(t) = \sum_{j=0}^n \frac{\langle at \rangle^j}{j!} < e^{\langle at \rangle}
$$
and the property that $\sum_{j=0}^n A_j(t) A_{n-j}(t) \leq C^n A_n(t)$.

By (\ref{2.229}), $W_{\pi^0}$ then has a Taylor series $W_{\pi^0}(t) = \sum_{n=0}^\infty W_{\pi^0}^n(t)$, such that in any Schwartz seminorm $(\mc S_n)_x$
$$
\|W^n_{\pi^0}(t)(\delta \pi)\|_{A_n(t) L^{\infty}_t (\mc S_n)_x} \leq C_1 C_2^n (\|\delta \dot \pi\|_1 + |\delta \pi(0)|)^n.
$$
This yields a power series expansion for $V_{\pi^0}$: $V_{\pi^0}(t) = \sum_{n=0}^\infty V_{\pi^0}^n(t)$, such that
$$
\|V^n_{\pi^0}(\delta \pi)\|_{A_n(t) L^{\infty}_t (\mc S_n)_x} \leq C_1 C_2^n (\|\delta \dot \pi\|_1 + |\delta \pi(0)|)^n.
$$
In this weighted space setting, let $m_1 + \ldots + m_n = m$ and take $n$ path variations $\delta \pi_1$ to $\delta \pi_n$ such that $\delta \pi_j(0) = 0$, $j=\ov{1, n}$. Then in any Schwartz seminorm $S_N$
$$\begin{aligned}
&\|V^n_{\pi^0}(\delta \pi_1, \ldots, \delta \pi_n)\|_{A_{m+n}(t) L^{\infty}_t (\mc S_N)_x} \les \\
&\les C_1^n \big(\|\delta \dot \pi_1\|_{A_{m_1}(t) L^1_t} + |\delta \pi_1(0)|\big) \cdot \ldots \cdot \big(\|\delta \dot \pi_n\|_{A_{m_n}(t) L^1_t} + |\delta \pi_n(0)|\big).
\end{aligned}$$
Starting with the explicit form of $F(R, \pi)$ and $F_f(R, \pi)$, note that the zeroth- and first-order terms cancel:
$$\lb{2nd_order}
F(R^0, \pi^0) = dF(R^0, \pi^0) = 0,\ F_f(R^0, \pi^0) = dF_f(R^0, \pi^0) = 0.
$$
This is due to the fact that $F(R, \pi)$ and $F_f(R, \pi)$ arise from linearizing a nonlinear term of order higher than two around $R^0 = 0$, $\dot \pi^0 = 0$.

We obtain the power series expansion $F = \sum_{m=0}^n F^m$, $F_f = \sum_{m=0}^n F_f^m$. $F^n$ and $F_f^n$ are $n$-linear and
$$\begin{aligned}
\|F^n\|_{A_{n-2}(t) L^2_t \dot W^{1/2, 6/5}_x} &\leq C_1 C_2^n (\|(\delta R, \delta \pi)\|_X + |\delta \pi(0)|)^n, \\
\|F_f^n\|_{A_{n-2}(t) L^1_t} &\leq C_1 C_2^n (\|(\delta R, \delta \pi)\|_X + |\delta \pi(0)|)^n. 
\end{aligned}$$

The gain from $n$ to $n-2$ takes place because the highest $\dot \pi$ derivative present in the expressions of $F^n$ and $F_f^n$ has order $n-2$. In turn, the reason for this is that the first nonzero derivatives of $F$ or $F_f$, in either $R$ or $\dot \pi$, are the second-order ones, and each derivative of order $n-k$ is paired with derivatives of order at most $k$.

Consider the scale of weighted spaces
$$
X_n=\{(R, \pi) \mid \|R\|_{A_n(t) L^{\infty}_t \dot H^{1/2}_x \cap A_n(t) L^2_t W^{1/2, 6}_x} + \|\dot \pi\|_{A_n(t) L^1_t} + |\pi(0)| < \infty \}.
$$
Let $m_1 + \ldots + m_n = m$ and note an extended analyticity property for $F^n$:
\be\begin{aligned}
\|F^n\|_{A_{m+n-2}(t) L^2_t \dot W^{1/2, 6/5}_x} &\leq C_1 C_2^n \|(\delta R_1, \delta \pi_1)\|_{X^{m_1}} \cdot \ldots \cdot \|(\delta R_n, \delta \pi_n)\|_{X^{m_n}} \\
\|F^n\|_{A_{m+n-2}(t) L^1_t} &\leq C_1 C_2^n \|(\delta R_1, \delta \pi_1)\|_{X^{m_1}} \cdot \ldots \cdot \|(\delta R_n, \delta \pi_n)\|_{X^{m_n}}.
\end{aligned}\ee
In the Taylor expansion of $F(\pi(R_0), R(R_0))$, where $\pi(R_0) \sim \sum_{m=0}^\infty \pi^m(R_0)$ and $R(R_0) \sim \sum_{m=0}^\infty R^m(R_0)$ are functions of $R_0$, the $n$-th-order term in $R_0$ is
\be
\sum_{j=1}^n \sum_{n_1 + \ldots + n_j=n} F^j\big((R^{n_1}, \pi^{n_1}), \ldots, (R^{n_j}, \pi^{n_j})\big).
\ee
Writing only the $n$-th-order terms in the Taylor expansion of (\ref{ec_r0}), we obtain the subsequent equation for $R^n$ and $\pi^n$:
\be\begin{aligned}\lb{2.231}
i \partial_t R^n + H_{\pi^0}(t) R^n &= \sum_{j=2}^n \sum_{n_1 + \ldots + n_j=n} F^j((R^{n_1}, \pi^{n_1}), \ldots, (R^{n_j}, \pi^{n_j})) \\
&- \sum_{j=1}^n \sum_{n_1 + \ldots + n_j + \tilde n=n} (\partial_{\pi}^j V_{\pi^0})(t) (\pi^{n_1}, \ldots, \pi^{n_j}) R^{\tilde n} := T_1 - T_2 \\
\dot f^n &= \sum_{j=2}^n \sum_{n_1 + \ldots + n_j=n} F_f^j((R^{n_1}, \pi^{n_1}), \ldots, (R^{n_j}, \pi^{n_j})) := T_3.
\end{aligned}\ee
The initial conditions are $R^n(0)$ and $\pi^n(0)$ given by (\ref{cond_init}).

Since the expansion is around zero, all terms directly containing $R^0$ or $\pi^0$ cancel:
\be\begin{aligned}\lb{2.246}
F^1(\pi^0, R^0) = 0,\ F_f^1(\pi^0, R^0) = 0,\ (\partial_{\pi} V_{\pi^0}(t)\, \pi^n)\, R^0 = 0.
\end{aligned}\ee
Then highest-order terms, i.e.\ containing $R^n$ and $\pi^n$, are absent from $T_1$, $T_2$, and $T_3$, which are then determined by terms of order up to $n-1$.

Finally, we prove by induction that for every $n$
$$
\|(R^n, \pi^n)\|_{X^{n-1}} \leq k_1 k_2^n \|R_0\|_{\dot H^{1/2}}^n.
$$

To this purpose, assume that the \emph{induction hypothesis}
\be\lb{indhyp}
\|(R^{m}, \pi^{m})\|_{X^{m-1}} \leq k_1 k_2^m \|R_0\|_{\dot H^{1/2}}^m
\ee
is true for every $1 \leq m<n$ and some constants $k_1$ and $k_2$ to be set later.

We take $k_1<<1$, then set $k_2$ sufficiently large so that (\ref{indhyp}) holds for $m<n=2$:
$$
\|(R^1, \pi^1)\|_{X^0} \leq k_1 k_2 \|R_0\|_{\dot H^{1/2}}.
$$
Then the inhomogenous terms in (\ref{2.231}) obey the bounds
$$\begin{aligned}\lb{2.248}
\|T_1\|_{A_{n-1}(t) L^2_t \dot W^{1/2, 6/5}_x} = \bigg\|\sum_{j=2}^n \sum_{n_1 + \ldots + n_j=n} F^j((R^{n_1}, \pi^{n_1}), \ldots, (R^{n_j}, \pi^{n_j}))\bigg\|_{A_{n-1}(t) L^2_t \dot W^{1/2, 6/5}_x} \leq \\
\leq C_1 \sum_{j=2}^n (C_2 k_1) ^j k_2^n \|R_0\|_{\dot H^{1/2}}^n
\end{aligned}$$
and likewise for the other two terms, that is
$$\begin{aligned}\lb{2.249}
\|T_2\|_{A_{n-1}(t) L^2_t \dot W^{1/2, 6/5}_x} = \bigg\|\sum_{j=1}^n \sum_{n_1 + \ldots + n_j + \tilde n=n} (\partial_{\pi}^j V_{\pi^0})(t) (\pi^{n_1}, \ldots, \pi^{n_j}) R^{\tilde n} \bigg\|_{A_{n-1}(t) L^2_t \dot W^{1/2, 6/5}_x} \leq \\
\leq C_1 \sum_{j=2}^n (C_2 k_1) ^j k_2^n \|R_0\|_{\dot H^{1/2}}^n,
\end{aligned}$$
respectively
\be\begin{aligned}\lb{2.250}
\|T_3\|_{A_{n-1}(t) L^1_t} = \bigg\| \sum_{j=2}^n \sum_{n_1 + \ldots + n_j=n} F_f^j((R^{n_1}, \pi^{n_1}), \ldots, (R^{n_j}, \pi^{n_j})) \bigg\|_{A_{n-1}(t) L^1_t} \leq \\
\leq C_1 \sum_{j=2}^n (C_2 k_1) ^j k_2^n \|R_0\|_{\dot H^{1/2}}^n.
\end{aligned}\ee

Importantly, summation starts from $j=2$ within $T_1$, $T_2$, and $T_3$  because the first-order differentials vanish by (\ref{2.246}). The powers of $t$ in $T_1$, $T_2$, and $T_3$ 
are at most $n-2$ since
\be
j-2 + \sum_{j=2}^n (n_j-1) = n-2 < n-1.
\ee
Making $k_1$ sufficiently small, the sum $\sum_{j=2}^n (C_2 k_1) ^j$ becomes bounded and uniformly small for all $n$, so
$$
\|T_1\|_{A_{n-1}(t) L^2_t \dot W^{1/2, 6/5}_x} + \|T_2\|_{A_{n-1}(t) L^2_t \dot W^{1/2, 6/5}_x} + \|T_3\|_{A_{n-1}(t) L^1_t} \les \tilde C k_1^2 k_2^n \|R_0\|_{\dot H^{1/2}}^n.
$$

We solve the equation system (\ref{2.231}) for $R^n$ and $\pi^n$ as we did before.

For the modulation path $\pi^n$, (\ref{2.250}) directly shows that $\|\dot \pi^n\|_{A_n(t) L^1_t} \les \|R_0\|_{\dot H^{1/2}}^n$.

As in (\ref{2.63}-\ref{2.67}), we apply a unitary transformation $U(t)$ to (\ref{2.231}) of the form
$$\begin{aligned}
U(t) &= e^{\textstyle\int_0^t(2 v^0(s) \dl + i ((\alpha^0)^2(s)-|v^0(s)|^2) \sigma_3) \dd s} \\
Z^n(t) &= U(t) R^n(t) \\
W(\pi^0(t)) &= U(t) W_{\pi^0}(t).
\end{aligned}$$
In (\ref{2.231}) this transformation takes a particular form, as $\alpha^0(t)$ and $v^0(t)$ are constant.

We split $Z^n$ into its three spectral projections $P_c(t) Z^n(t)$, $P_0(t) Z^n(t)$, and $P_{im}(t) Z^n(t)$ --- and estimate each piece in the weighted Strichartz norm $A_n(t) L^{\infty}_t \dot H^{1/2}_x \cap A_n(t) L^2_t \dot W^{1/2, 6}_x$.

For the continuous spectrum component $P_c(t) Z^n(t)$, Strichartz estimates imply that
\be
\|P_c(t) Z^n(t)\|_{A_n(t) L^{\infty}_t \dot H^{1/2}_x \cap A_n(t) L^2_t \dot W^{1/2, 6}_x} \leq \tilde C k_1^2 k_2^n \|R_0\|_{\dot H^{1/2}}^n.
\ee
Regarding the imaginary component $P_{im}(t) Z^n(t)$, recall the convolution estimates
$$\begin{aligned}
\int_0^t e^{-a_1 (t-s)} A_n(s) \dd s &\leq C A_n(t), \\
\int_t^{\infty} e^{a_1 (t-s)} A_n(s) \dd s &= \int_0^{\infty} e^{-a_1 s} A_n(t+s) \dd s \leq C A_n(t).
\end{aligned}$$
Both estimates hold uniformly in $n$.

Since the right-hand side of (\ref{2.231}) has polynomial growth, Lemma \ref{hyp} shows that there exists a unique subexponential solution to the equation of $P_{im}(t) Z(t)$, for a suitable value of the parameter $h^n$ and then
$$
\|P_{im}(t) Z(t)\|_{A_n(t) L^{\infty}_t \dot H^{1/2}_x \cap A_n(t) L^2_t \dot W^{1/2, 6}_x} \leq \tilde C k_1^2 k_2^n \|R_0\|_{\dot H^{1/2}}^n.
$$

$R^n$ fulfills an approximate orthogonality condition:
$$
\langle R^n, \partial^*_fW_{\pi^0} \rangle + \sum_{j=1}^n \sum_{n_1 + \ldots + n_j + \tilde n=n} \big\langle R^{\tilde n}, (d_{\pi}^j \partial^*_fW_{\pi^0}) (\pi^{n_1}, \ldots, \pi^{n_j}) \big\rangle = 0.
$$
By the induction hypothesis (\ref{indhyp}) we hence obtain a bound for $P_0(t) Z(t)$:
\be
\|P_0(t) Z^n(t)\|_{A_n(t) L^{\infty}_t \dot H^{1/2}_x \cap A_n(t) L^2_t \dot W^{1/2, 6}_x} \leq \tilde C k_1^2 k_2^n \|R_0\|_{\dot H^{1/2}}^n.
\ee

Combining the three estimates, for the unique suitable value of the parameter $h^n$ given by Lemma \ref{hyp}
\be
\|(R^n, \pi^n)\|_{X^{n-1}} \leq \tilde C k_1^2 k_2^n \|R_0\|_{\dot H^{1/2}}^n.
\ee
Setting $k_1 << 1$ we obtain
\be
\|(R^n, \pi^n)\|_{X^{n-1}} \leq k_1 k_2^n \|R_0\|_{\dot H^{1/2}}^n.
\ee

Next, we verify that $R^n = \frac 1 {n!} d^n R$ and $\pi^n = \frac 1 {n!} d^n \pi$. Denote
\be
S^n = R^0 + R^1 + \ldots + R^n,\ \Sigma^n = \pi^0 + \pi^1 + \ldots + \pi^n.
\ee
Up to error terms $O^{n+1}$ and $O_f^{n+1}$, $S^n$ and $\Sigma^n$ solve an equation system of the form (\ref{2.225}):
\be\begin{aligned}\lb{2.237}
&i \partial_t S^n + H_{\Sigma^n}(t) S^n = F(\Sigma^n, S^n) + O^{n+1} \\
&\dot f_{\Sigma^n} = F_f(\Sigma^n, S^n) + O_f^{n+1},\ f \in \{\alpha, \Gamma, v_k, D_k\}.
\end{aligned}\ee
The error terms are of size
$$
\|O^{n+1}\|_{\langle t \rangle^n L^2_t W^{1/2, 6/5}_x} + \|O_f^{n+1}\|_{\langle t \rangle^n L^1_t} \les \|R_0\|_{\dot H^{1/2}}^{n+1}.
$$
Comparing $(S^n, \Sigma^n)$ and $(R, \pi)$ in $Y$ as in Lemma \ref{lemma_5}, then
\be
\|(R, \pi) - (S^n, \Sigma^n)\|_Y \les \|R_0\|_{\dot H^{1/2}}^{n+1}.
\ee
This concludes the proof of analyticity for $R$ and $\pi$.


Finally, recall that $h(R_0, W_0)$ is given by formula (\ref{2.100}), which becomes
\be\begin{aligned}\lb{2.266}
h(R_0, W_0) &= -\int_0^{\infty} e^{-\int_0^t \sigma(W_{\pi}(\tau)) \dd \tau} \big(\langle F,  \sigma_3 F^-(W_{\pi}(t)) \rangle - \\
&- 3 \dot \alpha(t) (\alpha(t))^{-4} \langle R, i \sigma_3 F^-(W_{\pi}(t)) \rangle + \\
&+(\alpha(t))^{-3} \langle R, i \sigma_3 (d_{\pi} F^-(W_{\pi}(t))) \dot \pi(t) \rangle\big) \dd t.
\end{aligned}\ee
Since all the functions in this formula are analytic and grow more slowly than $e^{ta_1}$, we directly obtain a power series expansion for $h$.

Note that in the power series $h=\sum_{n=0}^{\infty} h^n$, each term $h^n$ is the unique value that makes $(R^n, \pi^n)$ bounded in equation (\ref{2.231}). By (\ref{2.266}), $h^n$ is given by an $n$-linear form of $R_0$ for each $n$, one that grows exponentially in norm with $n$.

Using Lemma \ref{lemma_5} to compare $(R, \pi)$ and the sum of the first $n$ terms $(S^n, \Sigma^n)$ given by (\ref{2.237}), we obtain
\be
|h-(h^0 + \ldots + h^n)| \leq C_1 C_2^n \|R_0\|_{\dot H^{1/2}}^{n+1}.
\ee
This construction of $h^n$ proves the analyticity of $h$ more explicitly.
\end{proof}

\subsection{The centre-stable manifold}
Finally, we relate the manifold $\mc N$ issued from Proposition \ref{prop27} to the centre-stable manifold of \cite{bates}.
\begin{proposition}\lb{batjon} $\mc N$ is a $\dot H^{1/2}$ centre-stable manifold for \eqref{NLS} in the sense of Definition \ref{centr} and~\cite{bates}.
\end{proposition}
\begin{proof}
To begin with, we rewrite (\ref{NLS}) to fit it within the stable manifold theory of \cite{bates}.

Consider the soliton $W_{\pi_0}(t)$ parametrized by a constant path, which we take to be $\pi_0(t) = (1, 0, 0, 0)$ with no loss of generality. Thus
\be
W_{\pi_0}(t) = \bpm e^{it} \phi(\cdot, 1) \\ e^{-it} \phi(\cdot, 1) \epm
\ee
for all $t$. We linearize the equation around this constant path. For $R = \Psi - W_{\pi_0}$, equation (\ref{rr}) takes the form
\be
\partial_t R - i H_{\pi_0} R = N(R, W_{\pi_0}).
\ee
Making the substitution $Z = e^{-it\sigma_3} R$, $W_{\pi_0}(t)$ becomes the constant soliton $W(\pi_0(t)):=W_0$ and $Z$ has the equation
\be
i \partial_t Z + H(W_0) Z = N(Z, W_0),
\lb{eq_2188}
\ee
where
\be
H(W_0) := \bpm \Delta+2\phi^2(\cdot, 1) -1 & \phi^2(\cdot, 1) \\ -\phi^2(\cdot, 1) & -\Delta-2\phi^2(\cdot, 1)+1 \epm
\ee
and
\be
N(Z, W_0) = \bpm -|z|^2 z - z^2 \phi(\cdot, 1) - 2|z|^2 \phi(\cdot, 1) \\ |z|^2 z + z^2 \phi(\cdot, 1) + 2|z|^2 \phi(\cdot, 1) \epm.
\ee
Note that the right-hand side terms are at least quadratic in $Z$, due to linearizing around a constant path.

The spectrum of $H$ is $\sigma(H) = (-\infty, -1] \cup [1, \infty) \cup \{0, \pm i \sigma\}$, see Section \ref{spectru}. The stable spectrum is $\{-i\sigma\}$, the unstable spectrum is $\{i\sigma\}$, and $(-\infty, -1] \cup [1, \infty)$ is the centre.



The manifold $\mc N \subset \dot H^{1/2}$ of Definition \ref{def3} 
is invariant under the time evolution (\ref{NLS}) by Corollary \ref{stable}. Let
$$
\tilde{\mc N} := \mc N - W_0 = \{\Psi - W_0 \mid \Psi \in \mc N\},\ W_0 = \bpm \phi(\cdot, 1) \\ \phi(\cdot, 1) \epm.
$$
$\tilde {\mc N}$ is the image of $\mc N$ under the mapping $\Psi \mapsto Z(\Psi) = e^{-it\sigma_3}\Psi - W_0$.

In the sequel we show that $\Psi \in \mc N$ if and only if $Z(\Psi) \in \mc N_{BaJo}$.
We show $\tilde{\mc N}$ is a centre-stable manifold for (\ref{eq_2188}) in the sense of Definition \ref{centr}, relative to a small $\dot H^{1/2}$ neighborhood $\mc V = \{Z \mid \|Z\|_{\dot H^{1/2}} < \delta_0\}$ of the origin.



We verify the three properties listed in Definition \ref{centr}: $\tilde{\mc N}$ is $t$-invariant with respect to $\mc V$, $\pi^{cs}(\tilde{\mc N})$ contains a neighborhood of $0$ in $X^c \oplus X^s$, and $\tilde{\mc N} \cap W^u = \{0\}$.

The $t$-invariance of $\tilde{\mc N}$ within $\mc V$ follows from Corollary \ref{stable}. Indeed, Corollary \ref{stable} is stronger than $t$-invariance, as it holds unconditionally and globally in time.

We next prove the second element of Definition \ref{centr}. Here $\pi^{cs} = P_c + P_0 + P_-$ and
$$
X_c \oplus X^s = \big(P_c(W_0) + P_0(W_0) + P_-(W_0)\big) \dot H^{1/2}.
$$
By Proposition \ref{prop27}, for each $R_0 \in \mc N_{lin}(W_0) = (P_c(W_0) + P_-(W_0)) \dot H^{1/2}$, $R_0 = \pi^{cs}\mc F(R_0, W_0)$, where $\mc F(R_0, W_0) \in \mc N$.

Moreover, let $R_1 = R_0 + P_0(W_0) R_1$, where $R_1 \in X^c \oplus X^s$, $R_0 \in \mc N_{lin}$, and using the exponential map $\exp_{W_0}:T_{W_0} \Sol \to \Sol$ define
$$
G(R_1) = \pi^{cs}\mc F\big(R_0, W_0 + \exp_{W_0}(P_0(W_0) R_1)\big).
$$
$G$ is nonlinear, but it is analytic and $dG(0) = I$. Thus $G$ is locally invertible at $0$, hence $\pi^{cs}$ is surjective on a $\dot H^{1/2}$ neighborhood of $0$.


Lastly, we show that $\tilde{\mc N} \cap W^u = \{0\}$, where $W^u$ is the unstable manifold given by Definition \ref{unstable}. To this purpose we follow the proof of \cite{bec}.

Consider a solution
$$
Z := e^{-it\sigma_3} \Psi - W_0 \in \tilde N \cap W^u
$$
of (\ref{eq_2188}). By Definition \ref{unstable}, $Z^0(t)$ exists for all $t<0$, $\|Z^0(t)\|_{\dot H^{1/2}} < \delta_0$ for some small $\delta_0$ and all $t<0$, and $Z^0$ decays exponentially as $t \to -\infty$, meaning that there exists $C_1$ such that for all $t\leq 0$ $\|Z^0(t)\|_{\dot H^{1/2}} \les e^{C_1 t}$.

Henceforth assume that $Z \not \equiv 0$ in order to obtain a contradiction.

Any rate of decay, rather than exponential decay, will suffice to obtain the contradiction.

Since $\ov {\Psi(0)}$ fulfills Definition \ref{definition_4}, it gives rise to a small asymptotically stable solution $\ov{\Psi(-t)}$ by Proposition \ref{prop27} and Proposition \ref{prop_2.13}.


Thus $\ov{\Psi(-t)}$ fulfills Definition \ref{definition_4}, meaning that $\Psi(t) = e^{it\sigma_3} Z^0(t) + W_{\pi^0}(t)$ is a small asymptotically stable solution of (\ref{NLS}) as $t$ goes to $-\infty$. 

Therefore, for $t \leq 0$ $\Psi = R(t) + W_{\pi}(t)$, the orthogonality condition $P_0(W_{\pi}(t)) R(t) = 0$ is satisfied, and $\|R(t)\|_{\dot H^{1/2}} \les \delta$. Endpoint Strichartz estimates imply that
\be
\|R\|_{L^{\infty}_t(-\infty, T] \dot H^{1/2}_x \cap L^2_t(-\infty, T] \dot W^{1/2, 6}_x} \les e^{C_1 T}. 
\ee
The path $\pi$ satisfies modulation equations, so $\|\dot \pi\|_{L^1_t(-\infty, T]} \les e^{C_2 T}$ as well.

Since $\Psi(0) \in \mc N$, $\pi(t)$ exists for all $t \geq 0$ and $R$ extends to a solution bounded in the Strichartz norm:
\be
\|R\|_{L^{\infty}_t[0, \infty) \dot H^{1/2}_x \cap L^2_t[0, \infty) \dot W^{1/2, 6}_x} \leq C \delta.
\ee
Let $U(t)$ be the family of isometries defined for $\pi(t) = (\alpha(t), \Gamma(t), v(t), D(t))$ by
$$
U(t) = e^{\textstyle\int_0^t(2 v(s) \dl + i (\alpha^2(s)-|v(s)|^2) \sigma_3) \dd s}
$$
and set $Z(t) = U(t) R(t)$. $Z(t)$ satisfies the equation
$$
i \partial_t Z - H(t) Z = F.
$$
The orthogonality condition holds for $Z$ by construction, so $P_0(t) Z(t) = 0$.

Let $Z(t) = P_c(t) Z(t) + P_{im}(t) Z(t)$ and
$$
\delta(T) := \|Z\|_{L^{\infty}_t(-\infty, T] \dot H^{1/2}_x \cap L^2_t(-\infty, T] \dot W^{1/2, 6}_x} + \|\dot {\pi}\|_{L^1_t(-\infty, T]}.
$$
Note that $\delta(t) \to 0$ as $t \to -\infty$, so we can assume $\delta(t)$ to be arbitrarily small.

By endpoint Strichartz estimates, for $T \leq 0$
\be\lb{bates_c}
\begin{aligned}
\|P_c(t) Z(t)\|_{L^2_t(-\infty, T] \dot W^{1/2, 6}_x \cap L^{\infty}_t(-\infty, T] \dot H^{1/2}_x} &\les \|F\|_{L^2_t(-\infty, T] \dot W^{1/2, 6/5}_x + L^{1}_t(-\infty, T] \dot H^{1/2}_x} \\
&\les \delta(T) \|Z\|_{L^{\infty}_t (-\infty, T] \dot H^{1/2}_x},
\end{aligned}
\ee
because the right-hand side contains only quadratic and higher degree terms.

Let $P_{im}(t) Z(t) = h^-(t) F^-(t) + h^+(t) F^+(t)$. Since $Z(t)$ is bounded as $t \to -\infty$, applying Lemma \ref{hyp} to $\ov{P_{im}(-t) Z(-t)}$ we obtain
\be\begin{aligned}
h^-(t) &= -\int_{-\infty}^t e^{-\sigma(t-s)} N_-(s) \dd s\\
h^+(t) &= e^{(t-T)\sigma} b_+(T) - \int_{t}^T e^{(t-s)\sigma} N_+(s) \dd s.
\end{aligned}\lb{PplusU}\ee
Therefore
$$
\|P_{im}(t) Z(t)\|_{L^{\infty}_t(-\infty, T] \dot H^{1/2}_x} \leq \|P_+ Z(T)\|_{\dot H^{1/2}_x} + \delta(T) \|Z\|_{L^{\infty}_t(-\infty, T] \dot H^{1/2}_x}.
$$
Combining (\ref{bates_c}) and (\ref{PplusU}), we obtain
\be\lb{asdfg}
\|Z\|_{L^{\infty}_t(-\infty, T] \dot H^{1/2}_x} \les \delta(T) \|Z\|_{L^{\infty}_t(-\infty, T] \dot H^{1/2}_x} + \|P_+ Z(T)\|_{\dot H^{1/2}}.
\ee
For sufficiently negative $T_0$, it follows that $\|Z(T)\|_{\dot H^{1/2}} \les \|P_+ Z(T)\|_{\dot H^{1/2}}$, for any $T \leq T_0$. The converse is also true, so the two norms are comparable.

Next, we estimate $(I-P_+(t)) Z(t)$, in a manner that parallels (\ref{bates_c}) and (\ref{PplusU}). Using $\|Z(T)\|_{\dot H^{1/2}} \les \|P_+ Z(T)\|_{\dot H^{1/2}}$ in the evaluation leads to this result for $(I-P_+(t)) Z(t)$:
\be\lb{2.294}
\|(I-P_+(t)) Z(t)\|_{\dot H^{1/2}} \les \delta(t) \|P_+Z(t)\|_{\dot H^{1/2}}.
\ee


Recall that $\Psi(0) \in \mc N$. Then, when in Definition \ref{definition_4} of $\mc N$ $\delta_0>0$ is sufficiently small, $\|Z(t)\|_{\dot H^{1/2}}$ is bounded from below as $t \to \infty$. Indeed, to a first order, $Z(t)$ is given by the free time evolution of the initial data; all other terms are quadratic in size. The first-order term is bounded away from zero, unless $Z(0) \equiv 0$.

On the other hand, by Lemma \ref{hyp}
\be
\|P_+ Z(t)\|_{\dot H^{1/2}} \leq \int_t^{\infty} e^{(t-s)\sigma} |N_+(s)| \dd s.
\ee
Hence $\|P_+ Z(t)\|_{\dot H^{1/2}}$ goes to zero and becomes arbitrarily small as $t \to \infty$.

Lemma 2.4 of \cite{bates} implies that if the ratio $\|P_+ Z(T_0)\|_{\dot H^{1/2}}/\|(I-P_+) Z(T_0)\|_{\dot H^{1/2}}$ is small enough, it will stay bounded when $t \leq T_0$. The proof of this lemma is based on Gronwall's inequality.

However, this contradicts (\ref{2.294}), as
\be
\|(I-P_+(t)) Z(t)\|_{\dot H^{1/2}}/\|P_+(t) Z(t)\|_{\dot H^{1/2}}  \leq C \delta(t)
\ee
goes to $0$ as $t$ approaches $-\infty$. This contradiction shows that  $Z \equiv 0$.

This proves that $\tilde{\mc N} \cap W^u = \{0\}$: there are no exponentially unstable solutions in $\tilde{\mc N}$ in the sense of \cite{bates}.

By Definition \ref{centr}, this shows that $\tilde {\mc N}$ is a centre-stable manifold.
\end{proof}

\subsection{Scattering}\lb{sect_sc}
Strichartz space bounds and estimates imply that the radiation term scatters like the solution of the free equation, meaning
\be
r(t) = e^{-it\Delta} r_{free} + o_{\dot H^{1/2}}(1)
\ee
for some $r_{free} \in \dot H^{1/2}$.

As a reminder, $R$ satisfies the equation (\ref{rr})
\be\lb{2.196}
i \partial_t R - H_{\pi}(t) R = F,
\ee
where $H_{\pi}(t) = \Delta \sigma_3 + V_{\pi}(t)$, and $R$ has finite Strichartz norm,
\be
\|R\|_{L^{\infty}_t \dot H^{1/2}_x \cap L^2_t \dot W^{1/2, 6}_x} < \infty,
\ee
while $F$ has finite dual Strichartz norm,
\be
\|F\|_{L^2_t \dot W^{1/2, 6/5}_x} < \infty.
\ee
Rewrite (\ref{2.196}) as
\be
i \partial_t R - \Delta \sigma_3 R = F - V_{\pi}(t) R.
\ee
By Duhamel's formula,
\be
R(t) = e^{-it\Delta \sigma_3} R(0) -i \int_0^t e^{-i(t-s)\Delta \sigma_3} (F(s) - V_{\pi}(s) R(s)) \dd s.
\ee
Let
\be
R_{free} = R(0) -i \int_0^{\infty} e^{it\Delta \sigma_3} (F(t) - V_{\pi}(t) R(t)) \dd t.
\ee
Then
\be
R(t) - e^{-it\Delta \sigma_3} R_{free} = ie^{-it\Delta \sigma_3} \int_t^{\infty} e^{is\Delta \sigma_3} (F(s) - V_{\pi}(s) R(s)) \dd t.
\ee
Note that
\be\begin{aligned}
& \|F - V_{\pi} R\|_{L^2_t \dot W^{1/2, 6/5}_x} \leq \\
\leq & \|F\|_{L^2_t \dot W^{1/2, 6/5}_x} + \|V_{\pi}\|_{L^{\infty}_t (\dot W^{1/2, 6/5-\epsilon}_x \cap \dot W^{1/2, 6/5+\epsilon}_x)} \|R\|_{L^2_t \dot W^{1/2, 6}_x} < \infty.
\end{aligned}\ee
implies
\be
\lim_{t \to \infty} \|\chi_{[t, \infty)}(s) (F(s) - V_{\pi}(s) R(s))\|_{L^2_t \dot W^{1/2, 6/5}_x} = 0.
\ee
$e^{-it\Delta \sigma_3}$ being an isometry, it follows that $R(t) - e^{-it\Delta \sigma_3} R_{free} \to 0$ in $\dot H^{1/2}$.

This leads to the same conclusion in the scalar case, after passing to the scalar functions $r$ and $r_{free}$, where $R = \bpm r \\ \ov r \epm$ and $R_{free} = \bpm r_{free} \\ \ov r_{free} \epm$.

\section{Linear estimates}
\subsection{Notations and basic results}

We prove a dispersive estimate for the linear time-dependent equation that includes terms of the form $v(t) \dl Z(t)$, where $\|\dot v(t)\|_{L^1_t}$ is small. $Z$ is a solution of the Schr\"{o}dinger equation, of finite Strichartz norm.

Such results have been proved in \cite{bec2}, \cite{bec3}, in a sharp, scaling-invariant setting. We adapt those results to the current problem by proving they also hold in $\dot H^{1/2}$.

Consider the linear Schr\"{o}dinger equation in $\R^3$
\be
i \partial_t Z + H Z = F,\ Z(0) \text{ given},
\ee
where
\be
H = H_0 + V = \bpm \Delta-\mu & 0 \\ 0 & -\Delta+\mu \epm + \bpm W_1 & W_2 \\ -W_2 & -W_1 \epm.
\lb{3.2}
\ee
$W_1$ and $W_2$ are real-valued. We also assume that $W_1$ and $W_2$ are of Schwartz class $\mc S$.

Following the nonlinear setting, we take $H$ with no eigenvalues, nor resonances in $(-\infty, -\mu] \cup [\mu, \infty)$.


The resolvent of $H_0 = -\Delta$, defined by $R_0(\lambda) := (H_0 - \lambda)^{-1}$, is complex analytic on the Riemann surface of $\sqrt {\lambda-\mu} + \sqrt{\lambda+\mu}$, so we express it as a function of $\sqrt{\lambda-\mu}$. Then it has the integral kernel
\be\lb{eq_3.56}
R_0(\lambda^2+\mu)(x, y) = \frac 1 {4\pi} \bpm -\frac {e^{-\sqrt{\lambda^2+2\mu} |x-y|}}{|x-y|} & 0 \\ 0 & \frac {e^{i \lambda |x-y|}}{|x-y|} \epm.
\ee

In the complex plane, $R_0(\lambda)$ is an analytic function on $\C \setminus \sigma(H_0)$. By the limiting absorption principle, it extends continuously up to the boundary in the closed lower half-plane or upper half-plane, but not in both at once, due to the jump discontinuity on $(-\infty, -\mu] \cup [\mu, \infty)$.

Let $\mc S$ be the Schwartz space --- a locally convex space defined by the family of seminorms
$$
\|f\|_{\mc S_n} := \sum_{k=1}^n \|\langle x \rangle^{n-k} \langle \Delta \rangle^k f\|_2.
$$
$f \in \mc S$ means that, for every $n$, $|f|_{\mc S_n}$ is finite.

Summarizing results of \cite{schlag} and \cite{erdsch2}, we list the spectral properties of $H$.
\begin{proposition}\lb{hspectru} Let $H = H_0 + V$ be given by \eqref{3.2} and take $V \in \mc S$. Then $\sigma(H) \subset \R \cup i \R$ and $\sigma_{ac}(H) = (-\infty, -\mu] \cup [\mu, \infty)$.

Assume that $H$ has no eigenvalues or resonances in $(-\infty, -\mu] \cup [\mu, \infty)$. Then the point spectrum of $H$ consists of simple eigenvalues, with the possible exception of $0$, where $H$ may have a nontrivial Jordan form.

The finitely many Riesz projections $P_{\zeta_j}$, $1 \leq j \leq n$, corresponding to the eigenvalues $\zeta_j$ are given by
$$
P_{\zeta_j} = \frac 1 {2\pi i} \int_{|z-\zeta_j| = \epsilon} R_V(z) \dd z.
$$
$P_{\zeta_j}$ and $P_{\zeta_j}^*$ have finite rank and their ranges are spanned by Schwartz-class functions.

The continuous spectrum projection $P_c$ is given by $P_c = I - \sum_{j=1} P_{\zeta_j}$.
\end{proposition}
We also state the limiting absorption principle, following \cite{agmon} and \cite{ionsch}, thusly.
\begin{proposition}
Let $H = H_0 + V$ be as in (\ref{3.2}) and assume that $V \in \mc S$. Assume that $H$ has no eigenvalues or resonances embedded in $(-\infty, -\mu] \cup [\mu, \infty)$. Then
\be
\sup_{\lambda \in \R} \|R_V(\lambda \pm i0)\|_{\mc B(\dot W^{1/2, 6/5}, \dot W^{1/2, 6})} < \infty.
\ee
\end{proposition}
\begin{proof}
Write $V=V_1 V_2$, where
\be
V_1 = \sigma_3 \bpm W_1 & W_2 \\ W_2 & W_1 \epm^{1/2},\ V_2 = \bpm W_1 & W_2 \\ W_2 & W_1 \epm^{1/2},\ \sigma_3 = \bpm 1 & 0 \\ 0 & -1 \epm.
\ee
Note that
\be\begin{aligned}
R_V(\lambda) &= R_0(\lambda) - R_0(\lambda) V R_0(\lambda) + R_0(\lambda) V_1 (I + V_2 R_0(\lambda) V_1)^{-1} V_2 R_0(\lambda),
\end{aligned}\ee
Thus $(I + V_2 R_0(\lambda\pm i0) V_1)^{-1} \in L^{\infty}_\lambda \mc B(\dot H^{1/2}, \dot H^{1/2})$ implies $R_V(\lambda\pm i0) \in L^{\infty}_\lambda \mc B(\dot W^{1/2, 6/5}, \dot W^{1/2, 6})$.

Since $V_2 R_0(\lambda\pm i0) V_1$ is compact on $\dot H^{1/2}$, by Fredholm's alternative, this is implied by the nonexistence of a function $f \in \dot H^{1/2}$ such that
\be
f + V_2 R_0(\lambda \pm i0) V_1 f = 0.
\ee
Indeed, $V_1 f$ would have to be an eigenstate or resonance for $H$ at $\lambda \in (-\infty, -\mu] \cup [\mu, \infty)$, contradicting our spectral assumption. Note that only $\pm \mu$ could actually be resonances.

Finally, the uniform boundedness of $\|V_2 R_0(\lambda\pm i0) V_1\|_{\mc B(\dot H^{1/2}, \dot H^{1/2})}$ follows from the norm-continuity of $V_2 R_0(\lambda\pm i0) V_1$ and the fact that
\be
\lim_{\lambda \to \infty} \|V_2 R_0(\lambda\pm i0) V_1\|_{\mc B(\dot H^{1/2}, \dot H^{1/2})} = 0.
\ee
\end{proof}

\noindent{\bfseries Notations.} The computations take place in Lebesgue and Sobolev spaces of functions on $\R^3 \times \R$ and occasionally in Lorentz spaces $L^{p, q}$, for which see \cite{bec3}.

$L^p$ denote Lebesgue spaces, with norm $\|f\|_p$, $1 \leq p \leq \infty$. When $n$ is an integer, Sobolev spaces of order $n$ are defined by
\be
\|f\|_{W^{n, p}} = \bigg(\sum_{|\alpha| \leq n} \|\partial^{\alpha} f\|_p^p\bigg)^{1/p}
\ee
for $1 \leq p < \infty$ and $\|f\|_{W^{n, \infty}} = \sup_{|\alpha| \leq n} \|\partial^{\alpha} f\|_{\infty}$ when $p=\infty$.

Homogenous and inhomogenous Sobolev spaces of fractional order, $\dot W^{s, p}$ and $W^{s, p}$, are defined by interpolation:
\be
\|f\|_{W^{s, p}} = \big\|\langle \dl \rangle^s f\big\|_p, \text{ respectively } \|f\|_{\dot W^{s, p}} = \big\||\dl|^s f\big\|_p.
\ee
Here $\langle \dl \rangle^s$ and $|\dl|^s$ denote Fourier multipliers:
$$
\langle \dl \rangle^s f = \big((1+|\xi|^2)^{s/2} \widehat f(\xi)\big)^{\vee},\ |\dl|^s f = \big(|\xi|^s \widehat f(\xi)\big)^{\vee}.
$$

When $p=2$, we denote $H^s := W^{s, 2}$ and $\dot H^s := \dot W^{s, 2}$.

$\mc B(X, Y)$ is the Banach space of bounded operators from $X$ to $Y$.

Strichartz estimates are proved in mixed space-time norms of the form
\be
\|f\|_{L^p_t \dot W^{s, q}_x} = \Big(\int_{-\infty}^{\infty} \|f(x, t)\|_{\dot W^{s, q}_x}^p \dd t\Big)^{1/p}.
\ee

\subsection{Strichartz estimates} We start by recalling the connection between the free evolution $e^{it H_0}$ and the resolvent $R_0(\lambda) = (H_0 - \lambda)^{-1}$ of $H_0$.
%

When $H = -\Delta + V$, the Fourier transform of $e^{itH}$ exists on a strip $\Im \lambda < -y_0$, by Gronwall's inequality.
\begin{lemma}\lb{lemma_24}
Assume that $V \in L^{\infty}$ and $H = H_0 + V$ has the form \eqref{3.2}. Then the~equation
\be
i \partial_t Z + H Z = F,\ Z(0) \text{ given},
\lb{3.58}\ee
admits a weak solution $Z \in L^{\infty}_{loc} L^2_t$ for $Z(0) \in L^2$ and $F \in L^{\infty}_t L^2_x$. For $t \geq 0$
\be
\|Z(t)\|_2 \leq C e^{t \|V\|_{\infty}} \|Z(0)\|_2 + \int_0^t e^{(t-s) \|V\|_{\infty}} \|F(s)\|_2 \dd s.
\lb{3.59}\ee
Furthermore, $R_V(\lambda) \in \mc B(L^{6/5, L^6})$ is the Fourier transform of $e^{it H}$ for $\Im \lambda < -\|V\|_{\infty}$:
\be\lb{3.14}
\lim_{\rho \to \infty} \int_0^{\rho} e^{-it\lambda} e^{it H} f \dd t = i R_V(\lambda) f.
\ee
\end{lemma}
We refer the reader to Lemma 2.10 in \cite{bec3} for the proof of this statement.

In order to obtain Strichartz estimates in the time-independent case, we use the following representation formula, proved in \cite{schlag} under more restrictive assumptions.
\begin{lemma}\lb{lemma_3.4} Let $V \in L^1 \cap L^{\infty}$  and $H$ be given by \eqref{3.2}. Assume that $H$ has no eigenstates or resonances embedded in its essential spectrum $\sigma_{ac}(H) = (-\infty, -\mu] \cup [\mu, \infty)$.

Then for any $f$, $g \in L^2$ $\big\langle \big(R_V(\lambda-i0) - R_V(\lambda+i0)\big) f, g\big\rangle$ is absolutely integrable and
\be\lb{3.16}
\langle P_c f, g \rangle = \frac 1 {2 \pi i} \int_{(-\infty, -\mu] \cup [\mu, \infty)} \big\langle \big(R_V(\lambda-i0) - R_V(\lambda+i0)\big) f, g \big\rangle \dd \lambda.
\ee
\end{lemma}
We refer the reader to Lemma 2.12 in \cite{bec3} for the proof.

From this representation, we derive the necessary $\dot H^{1/2}$ time-dependent Strichartz estimates. We also require the following technical lemma.
\begin{lemma}\lb{lemma32}
Consider $H = H_0 + V$ as in (\ref{3.2}) such that $V$ is of Schwartz class, $V \in \mc S$. Assume that $H$ has no eigenvalues or resonances embedded in $\sigma(H_0)$. Then there exists a decomposition
\be
V - P_p H = \tilde V_1 \tilde V_2,
\ee
where $P_p = I - P_c$, such that $\tilde V_1 \in \mc B(L^2, \langle x \rangle^{-\tilde N}L^2)$, $\tilde V_2 \in \mc B(\langle x \rangle^{\tilde N} L^2, L^2)$.
Moreover,
\be
\tilde V_1 \in \mc B(\dot H^{1/2}, \langle x \rangle^{-\tilde N} \dot H^{1/2}),\ \tilde V_2 \in \mc B(\langle x \rangle^{\tilde N} \dot H^{1/2}, \dot H^{1/2}).
\ee
\end{lemma}

\begin{proof} Take, for some large $n$,
\be\begin{aligned}
\tilde V_1 = (V - P_p H) \langle x \rangle^n,\ \tilde V_2 = \langle x \rangle^{-n}.
\end{aligned}\ee
$P_p H$ and $P_p^* H^*$ are finite-rank operators, whose ranges are spanned by the eigenfunctions of $H$, respectively of $H^*$.

All these eigenfunctions decay exponentially due to Agmon's bound \cite{agmon} and are smooth because the potential $V$ itself is smooth.
\end{proof}
This construction suffices in the case of Schwartz-class potentials. For a more general approach, see Lemma 2.13 in \cite{bec3}.

\begin{theorem}[$\dot H^{1/2}$ Strichrtz estimates]\lb{theorem_26}
Let $Z$ be a solution of the linear Schr\"{o}dinger equation
\be
i \partial_t Z + H Z = F,\ Z(0) \text{ given}.
\ee
Consider $H = H_0 + V$, $V \in \mc S$ is given by (\ref{3.2}), and assume that $H$ has no eigenvalues or resonances in $(-\infty, -\mu] \cup [\mu, \infty)$. Then
\be\lb{strichartz_ind}
\|P_c Z\|_{L^{\infty}_t \dot H^{1/2}_x \cap L^2_t \dot W^{1/2, 6}_x} \leq C \big(\|Z(0)\|_{\dot H^{1/2}} + \|F\|_{L^1_t \dot H^{1/2}_x + L^2_t \dot W^{1/2, 6/5}_x}\big).
\ee
\end{theorem}
\begin{proof}
We start from (\ref{3.30}) and use the fact that
$$
\sup_{\lambda \in \R} \|R_V(\lambda-i0) P_c\|_{\mc B(\dot W^{1/2, 6/5}, \dot W^{1/2, 6})} < \infty.
$$
Let $F$, $G \in L^{\infty}_t (L^1_x \cap L^2_x)$ have compact support in $t$ and consider the forward time evolution
$$
(T F)(t) = \int_{t>s} e^{i(t-s) H} P_c F(s) \dd s. 
$$
$T F(t)$ is in $L^2_x$ for all $t$ and grows at most exponentially, so its Fourier transform is well-defined for $\Im \lambda < -\|V\|_{\infty}$ (where, in particular, $\|R_V(\lambda)\|_{2 \to 2}$ is bounded):
$$
\widehat {TF}(\lambda) = i R_V(\lambda) P_c \widehat F(\lambda).
$$
By the representation formula (\ref{3.16}), for $P_c$ given by Proposition \ref{hspectru} and $f$, $g \in H^{1/2}$
$$\begin{aligned}
\langle R_V(\lambda_0) P_c f, g \rangle &= \frac 1 {2 \pi i} \int_{(-\infty, -\mu] \cup [\mu, \infty)} \big\langle R_V(\lambda_0) (R_V(\lambda-i0) - R_V(\lambda+i0)) f, g \big\rangle \dd \lambda \\
&= \frac 1 {2 \pi i} \int_{(-\infty, -\mu] \cup [\mu, \infty)} \Big\langle \frac 1 {\lambda-\lambda_0} (R_V(\lambda-i0) - R_V(\lambda+i0)) f, g \Big\rangle \dd \lambda.
\end{aligned}$$
Here we used the resolvent identity $R_V(\lambda_1) - R_V(\lambda_2) = (\lambda_1-\lambda_2) R_V(\lambda_1) R_V(\lambda_2)$.

For some fixed $\lambda_1 \not \in \sigma(H_0)$, $R_V(\lambda_1)$ is bounded from $L^{6/5}$ to $L^{6}$. Then, for any $\lambda_2 \ne \lambda_1$, $\lambda_2 \not \in \sigma(H_0)$
$$\begin{aligned}
&\langle R_V(\lambda_1) P_c f, g \rangle - \langle R_V(\lambda_2) P_c f, g \rangle = \\
&= \frac 1 {2 \pi i} \int_{\sigma(H_0)} \Big\langle \frac {\lambda_2 - \lambda_1} {(\lambda-\lambda_1)(\lambda-\lambda_2)} (R_V(\lambda-i0) - R_V(\lambda+i0)) f, g \Big\rangle \dd \lambda.
\end{aligned}$$
By the limiting absorption principle,
$$
\sup_{\lambda \in (-\infty, -\mu] \cup [\mu, \infty)} \|R_V(\lambda \pm i0)\|_{\dot W^{1/2, 6/5} \to \dot W^{1/2,6}} < \infty.
$$
Since the integrand decays like $\lambda^{-2}$, it follows that
\be\lb{3.39}
\sup_{\lambda \in \C} \|R_V(\lambda) P_c\|_{\dot W^{1/2, 6/5} \to \dot W^{1/2, 6}} < \infty.
\ee



When $y > \|V\|_{\infty}$, $e^{-yt} (T F)(t) \in L^2_t \dot H^{1/2}_x$ and $e^{yt} G(t) \in L^2_t \dot H^{1/2}_x$. Taking the Fourier transform in $t$, by Plancherel's theorem
\be\begin{aligned}\lb{3.41}
\int_{\R} \langle (T F)(t),  G(t) \rangle \dd t &= \frac 1 {2\pi} \int_{\R} \big\langle \big(e^{-yt} (T F)(t) \big)^{\wedge}, \big(e^{yt} G(t)\big)^{\wedge} \big\rangle \dd \lambda \\
&= \frac 1 {2 \pi i} \int_{\R} \big\langle R_V(\lambda - iy) P_c \widehat F(\lambda-iy), \widehat {G(-t)}(\lambda-iy) \big \rangle \dd \lambda.
\end{aligned}\ee
Here $\langle \cdot, \cdot \rangle$ is the dot product, in the real Hilbert space sense.

The pairing makes sense because $R_V(\lambda) P_c \in \mc B(L^2, L^2)$ is analytic in $\lambda$ for $\Im \lambda <0$. By the resolvent identity, we express $R_V P_c$ in (\ref{3.41}) as
\be\lb{3.42}
R_V P_c = R_0 - R_0 (V-P_p H) R_0 + R_0 F_1 (V) \big(\tilde V_2 R_V P_c \tilde V_1\big) \tilde V_2 R_0.
\ee
The first term represents the free Schr\"{o}dinger evolution, which is bounded by the endpoint Strichartz estimates of \cite{tao}:
$$
\frac 1 {2 \pi i} \int_{\R} \big\langle R_0(\lambda - iy) \widehat F(\lambda-iy), \widehat {G(-t)}(\lambda-iy) \big \rangle \dd \lambda.
$$
For the same reason it is true that
\be\lb{3.44}\begin{aligned}
\big\|\tilde V_2 R_0(\lambda-i0) \widehat F(\lambda)\big\|_{L^2_{\lambda, x}} \leq C \|F\|_{L^2_t L^{6/5}_x},\\
\big\|\tilde V_1 R_0(\lambda-i0) \widehat G(\lambda)\big\|_{L^2_{\lambda, x}} \leq C \|F\|_{L^2_t L^{6/5}_x}.
\end{aligned}\ee
Since $F(t)$ and $G(t)$ have compact support in $t$, it follows that $\widehat F(\lambda)$ and $\widehat G(\lambda)$ are analytic. For every $y \in \R$, $\tilde V_2 R_0(\lambda+iy) \widehat F(\lambda+iy)$ and $\tilde V_1^* R_0(\lambda + iy) \widehat G(\lambda+iy)$ are in $L^2_{\lambda, x}$.
%
This allows shifting the integration line toward the real axis. We obtain
\be\lb{3.30}
\int_{\R} \langle T F(t),  G(t) \rangle \dd t = \frac 1 {2 \pi i} \int_{\R} \big\langle R_V(\lambda - i0) P_c \widehat F(\lambda), \widehat {G(-t)}(\lambda) \big \rangle \dd \lambda.
\ee
Following (\ref{3.42}) and (\ref{3.44}), this implies
\be\begin{aligned}
\int_{\R} \langle T F(t),  G(t) \rangle \dd t &\les \big(1+ \sup_{\lambda \in \R} \|R_V(\lambda-i0) P_c\|_{\mc B(\dot W^{1/2, 6/5}, \dot W^{1/2, 6})}\big) \|F\|_{L^2_{t} \dot W^{1/2, 6/5}_x} \|G\|_{L^2_{t} \dot W^{1/2, 6/5}_x} \\
&\les \|F\|_{L^2_t \dot W^{1/2, 6/5}_x} \|G\|_{L^2_t \dot W^{1/2, 6/5}_x}.
\end{aligned}\ee
By approximation, we then remove the assumption that $F$ and $G$ have compact support and replace $H^{1/2}$ by $\dot H^{1/2}$. This establishes the inhomogenous Strichartz estimate
\be
\Big\|\int_{t>s} e^{i(t-s) H} P_c F(s) \dd s \Big\|_{L^2_t \dot W^{1/2, 6}_x} \les \|F\|_{L^2_t \dot W^{1/2, 6/5}_x}.
\ee
By iterating and applying Duhamel's formula we obtain (\ref{strichartz_ind}).
\end{proof}

\subsection{Strichartz estimates with time-dependent potentials}\lb{sec3.7}
We prove endpoint $\dot H^{1/2}$ Stric\-hartz estimates for the specific time-dependent problem obtained by linearizing (\ref{NLS}) around a moving soliton. We follow the framework of \cite{bec3} and only address the specifics of the $\dot H^{1/2}$ case here.

Given parameters $A(t)$ and $v(t) = (v_1(t), v_2(t), v_3(t))$, consider the family of isometries
\be\lb{3.155}
U(t) =  e^{\textstyle\int_0^t (2v(s) \dl + i A(s) \sigma_3) \dd s}.
\ee
The rate of change of $U(t)$ is then controlled by $\|A(t)\|_{L^{\infty}_t} + \|v(t)\|_{L^{\infty}_t}$.

The time-dependent linearized Schr\"{o}dinger equation has the form
\be\lb{3.157}
i \partial_t R(t) + (H_0 + U(t)^{-1} V U(t)) R(t) = F(t),\ R(0) \text{ given}.
\ee
The Hamiltonian at time $t$ is precisely $U(t)^{-1} (H_0 + V) U(t)$, since $H$ and $U$ commute, i.e.\ it is $H$ conjugated by $U(t)$.

Let $Z(t) = U(t) R(t)$. We rewrite the equation in the variable $Z$, obtaining
\be\lb{3.158}
i \partial_t Z(t) - i \partial_t U(t) U(t)^{-1} Z(t) + H_0 Z(t) + V Z(t) = U(t) F(t),\ Z(0) = R(0).
\ee
In the next lemma we list all the properties of $U(t)$ that we use in the study of (\ref{3.157}) and~(\ref{3.158}).

\begin{lemma}\lb{lem_32} Let $U(t)$ be defined by \eqref{3.155}.
\begin{list}{\labelitemi}{\leftmargin=1em}
\item[1.] $U(t)$ is a strongly continuous family of $\dot W^{s, p}$ isometries, for $s \in \R$, $1 \leq p < \infty$.
\item[2.] For every $t, s \geq 0$, $U(t)$ and $U(s)$ commute with $H_0$ and each other.
\item[3.] There exist $N$, $\epsilon(N)>0$ such that for all $0 \leq \sigma \leq 1$
\be\lb{3.163}\begin{aligned}
\|\langle x \rangle^{-N} (U(t) U(s)^{-1} e^{i(t-s) H_0} - e^{i(t-s) H_0}) \chi_{t > s} \langle x \rangle^{-N}\|_{\mc B(L^2_{\tau} \dot H^\sigma_x, L^2_t \dot H^\sigma_x)} \leq \\
\leq C (\|A(t)\|_{L^{\infty}_t} + \|v(t)\|_{L^{\infty}_t})^{\epsilon(N)}.
\end{aligned}\ee
\end{list}
\end{lemma}
For a parallel statement and its proof, also see Lemma 2.16 of \cite{bec3}.
\begin{proof} The first two properties are easy to check directly. As for the third, we compare $T(t, s) = e^{i(t-s) H_0} \chi_{t>s}$
and $\tilde T(t, s) = e^{i(t-s) H_0} e^{\int_s^t (2v(\tau) \dl + i A(\tau) \sigma_3) \dd \tau} \chi_{t>s}$.

Due to the pointwise decay of the kernel,
\be\lb{3.167}
\|\tilde T(t, s) - T(t, s)\|_{1 \to \infty} \les |t-s|^{-3/2}.
\ee
On the other hand, $\|\tilde T(t, s) - T(t, s)\|_{2 \to 2} \leq C$. It follows that for $N>1$
$$
\|\langle x \rangle^{-N} (T - \tilde T) \langle x \rangle^{-N}\|_{\mc B(L^2_{t, x}, L^2_{t, x})} < C,
$$
with a constant independent of $A$ and $v$.

Assume $\|A\|_{\infty}$, $\|v\|_{\infty} \leq 1$. Consider first the case where $v(t) \equiv 0$; hence let
\be\begin{aligned}
\tilde T_{osc}(t, s) &:= e^{i(t-s) H_0} e^{\int_s^t i A(\tau) \sigma_3 \dd \tau} \chi_{t>s}.
\end{aligned}\ee
Since $|e^{ia} -1| \leq \min(1, a)$,
\be\lb{3.172}
\|\tilde T_{osc}(t, s) - T(t, s)\|_{2 \to 2} \les \min(1, \|A\|_{\infty} |t-s|).
\ee
For sufficiently large $N$, (\ref{3.167}) and (\ref{3.172}) imply that
$$\begin{aligned}
&\|\langle x \rangle^{-N} (\tilde T_{osc} - T) \langle x \rangle^{-N}\|_{\mc B(L^2_{t, x}, L^2_{t, x})} \les \\
&\les \int_{t-\|A\|_{\infty}^{-2/5}}^t \|A\|_{\infty} |t-s| \dd s + \int_{-\infty}^{t-\|A\|_{\infty}^{-2/5}} |t-s|^{-3/2} \dd s \les \|A\|_{\infty}^{1/5}.
\end{aligned}$$
Next, consider the case when $v(t) \not \equiv 0$. Let $d(t) = \int_0^t v(\tau) \dd \tau$. Then
\be\begin{aligned}
&e^{-i(t-s) \Delta} e^{\int_s^t 2v(\tau)\dl \dd \tau} = \\
&\textstyle= \frac 1 {(-4\pi i)^{3/2}}(t-s)^{-3/2} e^{i\big(\frac{|x-y|^2}{4(t-s)} - \frac{(x-y) (d(t) - d(s))}{t-s} + \frac {(d(t)-d(s))^2}{t-s}\big)}.
\end{aligned}\ee
We treat $e^{i \frac {(d(t)-d(s))^2}{t-s}}$ as above. Consider the kernel
$$\begin{aligned}
\tilde T_1(t, s) &:= e^{i(t-s) H_0} e^{\int_s^t iA(\tau) \sigma_3 \dd \tau} e^{i \frac {(d(t)-d(s))^2}{t-s}} \chi_{t>s}.
\end{aligned}$$
Since $\frac {(d(t)-d(s))^2}{t-s} \leq \|v\|_{\infty} |t-s|$ and $|e^{ia}-1| \leq \min(1, a)$,
$$\begin{aligned}
&\|\langle x \rangle^{-N} (\tilde T_1 - \tilde T_{osc}) \langle x \rangle^{-N}\|_{\mc B(L^2_{t, x}, L^2_{t, x})} \les \\
&\les \int_{t-\|v\|_{\infty}^{-2/5}}^t \|v\|_{\infty} |t-s| \dd s + \int_{-\infty}^{t-\|v\|_{\infty}^{-2/5}} |t-s|^{-3/2} \dd s \les \|v\|_{\infty}^{1/5}.
\end{aligned}$$
Then
\be\begin{aligned}\lb{T1}
&\|\langle x \rangle^{-N} (\tilde T_1 - T) \langle x \rangle^{-N}\|_{\mc B(L^2_{t, x}, L^2_{t, x})} \les \\
&\les \|\langle x \rangle^{-N} (\tilde T_{osc} - T) \langle x \rangle^{-N}\|_{\mc B(L^2_{t, x}, L^2_{t, x})} + \|\langle x \rangle^{-N} (\tilde T_1 - \tilde T_{osc}) \langle x \rangle^{-N}\|_{\mc B(L^2_{t, x}, L^2_{t, x})}\\
&\les \|A\|_{\infty}^{1/5} + \|v\|_{\infty}^{1/5}.
\end{aligned}\ee
Considering that $
\big|e^{i \frac{(x-y) (d(t) - d(s))}{t-s}} -1\big| \les \min(1, \|v\|_{\infty} (|x| + |y|))$,
it follows that for sufficiently large $N$
$$\begin{aligned}
\|\langle x \rangle^{-N} (\tilde T(t, s) - \tilde T_1(t, s)) \langle x \rangle^{-N}\|_{2 \to 2} \les \|v\|_{\infty} |t-s|^{-3/2}.
\end{aligned}$$
Also note that $\|\tilde T(t, s) - \tilde T_1(t, s)\|_{2 \to 2}\leq C$. Therefore
\be\begin{aligned}\lb{TT}
&\|\langle x \rangle^{-N} (\tilde T_1 - \tilde T) \langle x \rangle^{-N}\|_{\mc B(L^2_{t, x}, L^2_{t, x})} \les \\
&\les \int_{t-\|v\|_{\infty}^{2/5}}^t |t-s| \dd s + \int_{-\infty}^{t-\|v\|_{\infty}^{2/5}} \|v\|_{\infty} |t-s|^{-3/2} \dd s \les \|v\|_{\infty}^{4/5}.
\end{aligned}\ee
By (\ref{T1}) and (\ref{TT}), for $\|v\|_{\infty} \leq 1$
\be\begin{aligned}\lb{L2}
&\|\langle x \rangle^{-N} (\tilde T - T) \langle x \rangle^{-N}\|_{\mc B(L^2_{t, x}, L^2_{t, x})} \les \\
&\les \|\langle x \rangle^{-N} (\tilde T - \tilde T_1) \langle x \rangle^{-N}\|_{\mc B(L^2_{t, x}, L^2_{t, x})} + \|\langle x \rangle^{-N} (\tilde T_1 - T) \langle x \rangle^{-N}\|_{\mc B(L^2_{t, x}, L^2_{t, x})}\\
&\les \|A\|_{\infty}^{1/5} + \|v\|_{\infty}^{1/5}.
\end{aligned}\ee
Also note that
$$\begin{aligned}
\|\langle x \rangle^{-N} (\tilde T - T) \langle x \rangle^{-N} F\|_{L^2_t \dot H^1_x} &\les \|\langle x \rangle^{-N-1} (\tilde T - T) \langle x \rangle^{-N} F\|_{L^2_{t, x}} + \\
&+ \|\langle x \rangle^{-N} (\tilde T - T) \langle x \rangle^{-N-1} F\|_{L^2_{t, x}} + \\
&+ \|\langle x \rangle^{-N} (\tilde T - T) \langle x \rangle^{-N} \dl F\|_{L^2_{t, x}} \\
&\les \big(\|A\|_{\infty}^{\epsilon/3} + \|v\|_{\infty}^{\epsilon/3}\big) \|F\|_{L^2_t \dot H^1_x},
\end{aligned}$$
by the Sobolev embedding $\dot H^1 \subset L^{6, 2} \subset \langle x \rangle L^2$. Thus, for sufficiently large $N$,
\be\lb{3.64}
\|\langle x \rangle^{-N} (\tilde T - T) \langle x \rangle^{-N} F\|_{L^2_t \dot H^1_x} \les (\|A\|_{\infty}^{1/5} + \|v\|_{\infty}^{1/5}) \|F\|_{L^2_t \dot H^1_x}.
\ee
Interpolating between (\ref{3.64}) and (\ref{L2}), we obtain (\ref{3.163}).
\end{proof}

\begin{theorem}\lb{theorem_13}
Consider equation \eqref{3.157}, with $H = H_0 + V$ given by \eqref{3.2} and $V \in \mc S$:
$$\begin{aligned}
i \partial_t Z - i v(t) \dl Z + A(t) \sigma_3 Z + H Z = F,\ Z(0) \text{ given},
\end{aligned}$$
$$\begin{aligned}
H = \bpm \Delta - \mu & 0 \\ 0 & -\Delta + \mu \epm + \bpm W_1 & W_2 \\ -W_2 & W_1 \epm.
\end{aligned}$$
Assume that $\|A\|_{\infty}$ and $\|v\|_{\infty}$ are sufficiently small in a manner that depends on $\|V\|_{\mc S_n}$, for sufficiently large $n$, and that there are no eigenvalues or resonances of $H$ in $(-\infty, -\mu] \cup [\mu, \infty)$. Then
\be
\|P_c Z\|_{L^{\infty}_t \dot H^{1/2}_x \cap L^2_t \dot W^{1/2, 6}_x} \les \|Z(0)\|_{\dot H^{1/2}} + \|F\|_{L^1_t \dot H^{1/2}_x + L^2_t \dot W^{1/2, 6/5}_x}.
\ee
\end{theorem}

We also use exponentially weighted Strichartz inequalities.
\begin{corollary}\lb{cor_exp}
Consider equation \eqref{3.157}, with $H = H_0 + V$ given by \eqref{3.2} and $V \in \mc S$. Assume that $\|A\|_{\infty}$ and $\|v\|_{\infty}$ are sufficiently small in a manner that depends on $\|V\|_{\mc S_n}$ and that there are no eigenvalues or resonances of $H$ in $(-\infty, -\mu] \cup [\mu, \infty)$. Then
$$
\|P_c Z\|_{e^{t\rho} L^{\infty}_t \dot H^{1/2}_x \cap e^{t\rho} L^2_t \dot W^{1/2, 6}_x} \les \|Z(0)\|_{\dot H^{1/2}} + \|F\|_{e^{t\rho} L^1_t \dot H^{1/2}_x + e^{t\rho} L^2_t \dot W^{1/2, 6/5}_x}.
$$
Furthermore,
$$
\|P_c Z\|_{e^{t\rho} L^{\infty}_t \dot H^{1/2}_x} \les \langle \rho \rangle^{-1} \|P_c Z\|_{e^{t\rho} L^{\infty}_t \dot H^{1/2}_x}.
$$
\end{corollary}
\begin{proof} The first inequality simply corresponds to using the same proof for the kernels
$$
T(t, s) = e^{i(t-s)H_0-\rho(t-s)} \chi_{t>s},\ \tilde T(t, s) = e^{i(t-s)H_0-\rho(t-s)} e^{\int_s^t(2v(\tau)\dl+iA(\tau)\sigma_3)\dd \tau} \chi_{t>s}.
$$
The result is then entirely analogous.

The second conclusion is proved by Minkowski's inequality.
\end{proof}

\begin{proof}[Proof of Theorem \ref{theorem_13}] To begin with, by Lemma \ref{lemma_24}, $V \in \mc S$
guarantees the existence of a solution $R$, albeit one that may grow exponentially.

As in the time-independent case, let
\be\begin{aligned}
&\tilde Z = P_c Z,\ \tilde F = P_c F - 2iv(t)[P_c, \dl] \tilde Z + A(t) [P_c, \sigma_3] \tilde Z.
\end{aligned}\ee
The equation becomes
$$\begin{aligned}
&i \partial_t \tilde Z - i v(t) \dl \tilde Z + A(t) \sigma_3 \tilde Z + H \tilde Z = \tilde F,\ \tilde Z(0) = P_c R(0) \text{ given}.
\end{aligned}$$
The commutation terms
$$
2iv(t)[P_c, \dl] \tilde Z,\ A(t) [P_c, \sigma_3] \tilde Z
$$
are small in the dual Strichartz norm for small $\|v\|_{\infty}$ and $\|A\|_{\infty}$ and thus can be controlled by a fixed point argument in the endpoint Strichartz norm.

Lemma \ref{lemma32} provides the decomposition $V - P_p {H} = \tilde V_1 \tilde V_2$, where $\tilde V_1$ and $\tilde V_2^*$ are bounded from $\dot H^{1/2}$ to $\langle x \rangle^{-N} \dot H^{1/2}$. Denote, for $U$ given by (\ref{3.155}),
\be\lb{3.144}\begin{aligned}
\tilde T_{\tilde V_2, \tilde V_1} F(t) &= \int_{-\infty}^t \tilde V_2 P_c e^{i(t-s){H}_0} U(t) U(s)^{-1} \tilde V_1 F(s) \dd s, 
\end{aligned}\ee
respectively
$$\begin{aligned}
\tilde T_{\tilde V_2, I} F(t) &= \int_{-\infty}^t \tilde V_2 P_c e^{i(t-s){H}_0} U(t) U(s)^{-1} F(s) \dd s.
\end{aligned}$$
By Duhamel's formula,
\be\lb{3.186}\begin{aligned}
\tilde V_2 \tilde Z(t) &= i \tilde T_{\tilde V_2, \tilde V_1} \tilde V_2 \tilde Z(t) + \tilde T_{\tilde V_2, I} (-i\tilde F(s) + \delta_{s=0} \tilde Z(0)).
\end{aligned}\ee
We compare the time-dependent kernel $\tilde T_{\tilde V_2, \tilde V_1}$ with the time-independent one
\be\begin{aligned}
T_{\tilde V_2, \tilde V_1} F(t) &= \int_{-\infty}^t \tilde V_2 P_c e^{i(t-s){H}_0} \tilde V_1 F(s) \dd s. 
\end{aligned}\ee
By Lemma \ref{lem_32} we obtain that
\be
\lim_{\substack{\|A\|_{\infty} \to 0 \\ \|v\|_{\infty} \to 0}} \|T_{\tilde V_2, \tilde V_1} - \tilde T_{\tilde V_2, \tilde V_1}\|_{\mc B(L^2_t \dot H^{1/2}_x, L^2_t \dot H^{1/2}_x)} = 0.
\ee
The operator $I - i T_{\tilde V_1, \tilde V_2}$ is invertible in $\mc B(L^2_t \dot H^{1/2}_x, L^2_t \dot H^{1/2}_x)$. Indeed, its inverse is
\be\begin{aligned}
(I - i T_{\tilde V_1, \tilde V_2})^{-1} F(t) &= F(t) - i\int_{-\infty}^t \tilde V_2 P_c e^{i(t-s){H}} \tilde V_1 F(s) \dd s.
\end{aligned}\ee
The right-hand side belongs to $\mc B(L^2_{t, x}, L^2_{t, x})$ due to Theorem \ref{theorem_26} and the Duhamel formula proves that the right-hand side is the inverse of the left-hand side, as claimed.

Hence, when $\|A\|_{\infty}$ and $\|v\|_{\infty}$ are small enough, $I - i \tilde T_{\tilde V_1, \tilde V_2}$ is also invertible in $\mc B(L^2_{t, x}, L^2_{t, x})$. 
Once we invert $\tilde T_{\tilde V_2, \tilde V_1}$ in $\mc B(L^2_t \dot H^{1/2}_x, L^2_t \dot H^{1/2}_x)$, the proof of $\dot H^{1/2}$ endpoint Strichartz estimates proceeds by
\be\lb{3.87}
\tilde Z = \big(\tilde T_{I, I} + \tilde T_{\tilde V_1, I} (I - i T_{\tilde V_1, \tilde V_2})^{-1} \tilde T_{I, \tilde V_2}\big) (\delta_{t=0} \tilde Z(0) -i \tilde F).
\ee

\end{proof}








\end{document}